\documentclass[11pt,a4paper]{amsart}

\usepackage{graphicx}
\usepackage[inline]{enumitem}
\usepackage{amsmath, amssymb, amsthm, amscd}
\usepackage{xcolor}
\usepackage{caption}
\usepackage{subcaption}
\usepackage{multirow}
\usepackage{stmaryrd}
\usepackage{mathtools}
\usepackage[export]{adjustbox}[2011/08/13]

\theoremstyle{plain}
\newtheorem{theorem}{Theorem}[section]
\newtheorem{lemma}[theorem]{Lemma}
\newtheorem{corollary}[theorem]{Corollary}
\newtheorem{proposition}[theorem]{Proposition}

\theoremstyle{definition}
\newtheorem{definition}[theorem]{Definition}

\theoremstyle{remark}
\newtheorem{remark}[theorem]{Remark}

\numberwithin{equation}{section}

\def\bold#1{\mbox{\boldmath $#1$}}
\newcommand{\uu}[1]{\bold{#1}}

\newcommand{\abs}[1]{\lvert#1\rvert}
\newcommand{\D}{\partial}
\newcommand{\dd}{\mathrm{d}}
\newcommand{\dive}{\mathrm{div}}
\newcommand{\bdive}{\uu{\mathrm{div}}}

\newcommand{\Dt}{\partial_t}

\newcommand{\gm}{\gamma}

\newcommand{\grd}{\nabla}
\newcommand{\bgrd}{\uu{\nabla}}

\newcommand{\mbb}{\mathbb}

\newcommand{\mcal}{\mathcal}

\newcommand{\norm}[1]{\lVert#1\rVert}
\newcommand{\Norm}[1]{{\left\vert\kern-0.25ex\left\vert\kern-0.25ex\left\vert #1 
    \right\vert\kern-0.25ex\right\vert\kern-0.25ex\right\vert}}

\newcommand{\half}{\frac{1}{2}}
\newcommand{\veps}{\varepsilon}
\newcommand{\vphi}{\phi}

\newcommand{\s}{\sigma}
\newcommand{\vel}{\uu{u}}

\newcommand{\M}{\mcal{M}}
\newcommand{\E}{\mcal{E}}
\newcommand{\Ds}{D_\sigma}
\newcommand{\Eds}{\bar{\mcal{E}}(D_\sigma)}
\newcommand{\Eint}{\mcal{E}_{\mathrm{int}}}
\newcommand{\Ek}{\mcal{E}(K)}
\newcommand{\Lm}{L_{\M}(\Omega)}

\newcommand{\Hez}{\uu{H}_{\E,0}(\Omega)}

\newcommand{\dt}{\delta t}
\newcommand{\chark}{\mcal{X}_K}

\newcommand{\fsigk}{F_{\sigma , K}}
\newcommand{\fesig}{F_{\epsilon , \sigma}}

\newcommand{\diam}{\mathrm{diam}}
\newcommand{\absq}[1]{\abs{#1}^2}
\newcommand{\intr}{\mathrm{int}}
\newcommand{\extr}{\mathrm{ext}}
\def\bold#1{\mbox{\boldmath $#1$}}

\makeindex             

\begin{document}

\title[An Energy Stable Well-balanced Scheme]{An Energy Stable 
  Well-balanced Scheme for the Barotropic Euler System with Gravity
  under the Anelastic Scaling}   
  
\author[Arun]{K.\ R.\ Arun}
\thanks{K.\ R.\ A.\ gratefully acknowledges the Core Research Grant -
  CRG/2021/004078 from the Science and Engineering Research Board,
  Department of Science \& Technology, Government of India.}  
\address{School of Mathematics, Indian Institute of Science Education
  and Research Thiruvananthapuram, Thiruvananthapuram 695551, India}   
\email{arun@iisertvm.ac.in, mainak17@iisertvm.ac.in}

\author[Kar]{Mainak Kar}

\date{\today}



\maketitle

\begin{abstract}
  We design and analyse an energy stable, structure preserving,
  well-balanced and asymptotic preserving (AP) scheme for the 
  barotropic Euler system with gravity in the anelastic limit. The key
  to energy stability is the introduction of appropriate velocity
  shifts in the convective fluxes of mass and momenta. The
  semi-implicit in time and finite volume in space fully-discrete
  scheme supports the positivity of density and yields the consistency
  with the weak solutions of the Euler system upon mesh
  refinement. The numerical scheme admits the discrete hydrostatic
  states as solutions and the stability of numerical solutions in
  terms of the relative energy leads to well-balancing. The AP
  property of the scheme, i.e.\ the boundedness of the mesh parameters
  with respect to the Mach/Froude numbers and the scheme's asymptotic
  consistency with the anelastic Euler system is rigorously shown on
  the basis of apriori energy estimates. The numerical scheme is
  resolved in two steps: by solving a non-linear elliptic problem for
  the density and a subsequent explicit computation of the
  velocity. Results from several benchmark case studies are presented
  to corroborate the proposed claims.   
\end{abstract}

\section{Introduction}
\label{sec:Intro}

The Euler equations of inviscid compressible fluid dynamics are
invariably one of the most commonly used set of equations to model
non-hydrostatic atmospheric processes. These equations have a 
complex wave-structure consisting of advection waves, gravity waves
and acoustic waves or sound waves. Though the sound waves are not of
much relevance in atmospheric flow calculations, they introduce a lot
of difficulties in the numerical integration of flow equations due to
their large characteristic speeds. Developing efficient methods for
handling fast acoustic waves is still an active area of research. In
the realm of atmospheric flow computations, the so-called
sound-proof models constitute a good alternative to using the full set
of compressible Euler equations. The sound-proof models are usually
derived via approximating the Euler equations by eliminating fast
acoustic waves, at the same time retaining meteorologically relevant
advection and internal gravity waves. For the Euler equations, the two 
characteristic parameters which quantify the effects due to
compressibility and gravitational stratification are, respectively,
the Mach and Froude numbers. It is relevant to derive sound-proof
models for atmospheric and astrophysical flows by considering the
scenario when the Mach and Froude numbers diminish to zero at the same
rate. The resulting regime, wherein the flow becomes incompressible
and stratified simultaneously at the same pace, can be represented by
a sound-proof equation system called the anelastic equations
\cite{MET+00}. In the meteorological literature, different types of
anelastic models which vary in the way they account for density or
temperature stratifications and internal gravity waves can be found
\cite{Kle09}. The forerunners among them are the anelastic model of
Ogura and Philips \cite{OP62}, the anelastic model of Bannon
\cite{Ban96} and the pseudo-incompressible model of Durran
\cite{Dur88}.      

From a mathematical point of view, the Ogura and Philips anelastic
approximation of the Euler system via the vanishing limit of the Mach 
and Froude numbers corresponds to performing a singular limit of the
governing equations whereby the Euler system changes its nature from
hyperbolic to mixed hyperbolic-elliptic; see, e.g.\
\cite{BF18a,FZ23,FN09,Mas07} for rigorous and detailed accounts on
various sound-proof limits. Consequently, a feasible platform to build
a numerical Euler solver for multiscale atmospheric flow applications
turns out to be the so-called `Asymptotic Preserving' (AP) methodology
\cite{Jin12}. The AP framework is a generic, yet powerful, tool for
approximating singularly perturbed problems. When applied to a
stratified atmospheric flow simulation, an AP scheme can automatically 
emulate the flow features in the singular (anelastic), non-singular
(compressible) as well as the intermediate (weakly compressible)
regions where the regime changes can occur. The working principle
behind the AP methodology is as follows. Suppose $P^\veps$ denote a
singularly perturbed problem, such as the scaled compressible Euler
equations containing small Mach and Froude numbers denoted by an
infinitesimal $\veps$. In the singular limit $\veps\to0$, suppose the
problem $P^\veps$ converges to a well-posed limit problem $P^0$; e.g.\
the anelastic model. An AP scheme respects the above convergence at
the discrete level in the following sense. 
\begin{enumerate}[label=(\roman*)]
\item Let $h$ denote the space-time discretisation parameter. For each
  $\veps>0$, let $P^\veps_h$ denote a consistent discretisation of the
  problem $P^\veps$, i.e.\ $P^\veps_h\to P^\veps$ as $h\to 0$.
\item For a fixed discretisation $h$, the scheme $P^\veps_h$ goes 
  to  $P^0_h$ as $\veps\to 0$, where $P^0_h$ denotes a consistent
  scheme for the limit problem $P^0$, i.e.\ $P^0_h\to P^0$ as $h\to
  0$.  
\item The stability constraints on the discretisation parameter $h$
  remain independent of $\veps$. 
\end{enumerate}

A rather common practice to develop AP schemes for time-dependent
partial differential equations, such as the Euler equations, is to employ
an implicit-explicit (IMEX) or semi-implicit time
discretisation \cite{BFR16}. Compared to their fully-implicit
counterparts, a semi-implicit formalism has the advantage that it
minimises the need to invert large and dense matrices.  However, the
challenge is to maintain the optimal level of implicitness so as to
overcome the stiffness imposed by the time-step restrictions and
achieve computational efficiency at the same time.  

Models of gravity driven flows, such as the Euler equations with
gravitational source terms, often describe a phenomenon where the
force due to the gravity is balanced by the internal force. This
balance of forces leads to an interesting stationary solution, known
as the hydrostatic steady state, which corresponds to a static fluid with
the corresponding pressure gradient exactly in balance with the weight
of the fluid. Most of the atmospheric flow problems of practical 
interest can be described as small perturbations of the hydrostatic
steady state \cite{Dur99}. Furthermore, the derivation of Ogura and
Philipps anelastic approximation of the Euler equations by a scale
analysis \cite{Dur99,OP62} or a zero Mach number asymptotic analysis
\cite{Kle00} reveals that the ambient flow or the leading order
solution is in hydrostatic balance. Consequently, it is much desirable
that numerical schemes for the Euler equations with gravity,
particularly in the anelastic regime, be capable of preserving the
steady states at a discrete level while accurately emulating the flow
features evolving from its perturbations. Unfortunately, classical
numerical schemes often fail to preserve the steady states for a large
time within an acceptable range of accuracy. A cure to this ailment is
achieved by introducing the so-called `well-balanced' schemes. A
well-balanced scheme exactly satisfy the discrete counterparts of the
steady states of the corresponding continuum equations. The
well-balancing methodology has since been adopted and developed for
formulating robust numerical schemes for several hydrodynamic models
with source terms in the context of explicit time-stepping schemes;
see, e.g.\ \cite{ABB+04,FMT11,Gos00,Xu02} for a few
references. However, establishing the well-balancing property of IMEX 
or semi-implicit time discretisation schemes is a challenging
task. Often, proving the balance for implicit time discretisations
involve solving nonlinear algebraic equations for the pressure or
density. The literature in this direction is rather sparse; see, e.g.\
\cite{BBK24,BAL+14,TPK20} and the references therein.     

The aim of the present work is to design and analyse a semi-implicit,
AP, well-balanced, energy stable and structure preserving finite
volume scheme for the barotropic Euler system with gravity in the
anelastic limit. An optimal semi-implicit time discretisation is
adopted to overcome the stiffness posed by the pressure gradient
and gravitational terms. Since the solutions of hyperbolic systems
are known to develop discontinuities in finite time, additional
stability requirements are to be imposed to single out physically
valid weak solutions. Numerical schemes that respect important
physical stability properties of the corresponding continuum
equations, such as preserving the positivity of mass density,
respecting the second law of thermodynamics (entropy stability) and so 
on, are sometimes referred to as invariant domain preserving or
structure preserving methods. Consequent to the semi-implicit 
treatment of the fluxes, the present scheme achieves the energy
stability apriori under a CFL-like time-step restriction. An
application of the tools from topological degree theory in finite
dimensions helps in proving the existence of a numerical solution and
the positivity of density. The scheme is weakly consistent with the
continuous equations as the mesh parameters go to zero {\`a} la
Lax-Wendroff. We employ the so-called relative energy, defined with
respect to a hydrostatic equilibrium state, as a measure to establish
the energy stability. The relative energy is a powerful tool that can
be adopted to obtain uniform estimates on the solutions and to perform
rigorous asymptotic limits of the continuous equations \cite{FN09} as
well as their numerical approximations
\cite{AGK23,FL18,HLS21}. Interestingly, in the present work, we 
show that using relative energy not only facilitate the entropy
stability but also gives another crucial benefit, namely the
provability of both the well-balancing and AP properties. Our careful
design to achieve these salient features involves two key
ingredients. The first is introducing a shift in the velocity in the
discrete convective fluxes of mass and momenta as in
\cite{DVB17,DVB20} to get the dissipation of mechanical energy. The
second is a novel treatment of the discrete density on the interfaces,
while discretising the mass flux and the stiff gravitational source
term. This particular interface discretisation, referred to as the
`$\gamma$-mean' with $\gamma$ being the adiabatic constant, helps in
emulating the entropy balance at the discrete level, and subsequently
leads to the well-balancing and AP properties. The stability of
solutions with respect to the relative energy ensures that any initial
steady state remains preserved during the subsequent updates. In other
words, the discrete energy stability renders the scheme
well-balanced. The stability also guarantees that numerical solutions
generated from initial conditions that are given in terms of
perturbations of the steady state remain numerically stable, i.e.\ no
spurious solutions are generated. The discrete relative energy
stability further leads to apriori estimates on the solution which
enables us to perform a rigorous asymptotic convergence analyses with
respect to the singular parameters, i.e.\ the Mach and Froude
numbers. Consequently, we establish the AP property of the scheme,
where we rigorously prove that the proposed scheme for the
compressible Euler system with gravity converges to a semi-implicit
scheme for the anelastic Euler equations as the Mach and Froude
numbers vanish. We wish to emphasise that this is the first time a
finite volume scheme has been developed for the barotropic Euler
system with gravitational source term, which is at the same time
provably energy stable, well-balanced, AP and positivity preserving,
to the best of our knowledge. 

The rest of this paper is organised as follows. In
Section\,\ref{sec:ewg_cont} we review a few results regarding the
scaled barotropic Euler system with gravity and its anelastic
limit. Section\,\ref{sec:MAC_disc_diff} provides an outline of the
domain discretisation along with the discrete differential
operators in a finite volume framework. The final scheme and the
description of the discrete hydrostatic steady state is described in
Section\,\ref{sec:FVApp}. Section\,\ref{sec:enstab_wb} is devoted to
prove the energy stability and the well-balancing property of the
proposed scheme. In Sections \ref{sec:weak_cons} and
\ref{sec:dis_anelastic_limit} we give rigorous proofs of the
consistency of the approximate solutions with respect to the weak
solutions and the anelastic limit model respectively. Finally, in
Section\,\ref{sec:num_exp}, we substantiate our claims regarding the
salient properties of the scheme with several numerical case studies.   

\section{Compressible Euler System with Gravity and the Anelastic
  Limit} 
\label{sec:ewg_cont}

We start with the following non-dimensional barotropic compressible
Euler system with gravity, parametrised by the Mach and the Froude
numbers, and posed for $(t, \uu{x}) \in Q_T := (0, T ) \times \Omega$:
\begin{align}
  \D_t\rho^\veps+\dive (\rho^\veps\uu{u}^\veps) &=
                                                  0, \label{eq:cons_mas}\\ 
  \D_t(\rho^\veps\uu{u}^\veps)+\bdive
  (\rho^\veps\uu{u}^\veps\otimes\uu{u}^\veps)+\frac{1}{\veps^2}\grd
  p^\veps &=-\frac{1}{\veps^2}\rho^\veps\grd\phi. \label{eq:cons_mom}  
\end{align}
Here, $T>0$ and $\Omega\subset\mbb{R}^d, \ d=1,2,3$, is a bounded open
connected subset and the dependent variables
$\rho^\veps=\rho^\veps(t,\uu{x})>0$ and
$\uu{u}^\veps=\uu{u}^\veps(t,\uu{x})\in\mbb{R}^d$ denote the fluid
density and the fluid velocity, respectively. The pressure
$p^\veps=\wp(\rho^\veps)$ is assumed to follow a barotropic equation
of state $\wp(\rho):=\rho^\gamma$, with $\gamma> 1$ being the ratio of
specific heats. The gravitational potential $\phi=\phi(\uu{x})$ is
assumed to be a known, continuous function. The system
\eqref{eq:cons_mas}-\eqref{eq:cons_mom} is supplied with the following
initial and boundary conditions:
\begin{equation}
  \label{eq:eq_ic_bc}
  \rho^\veps\vert_{t=0}=\rho^\veps_0, \quad
  \uu{u}^\veps\vert_{t=0}=\uu{u}^\veps_0, \quad
  \uu{u}^\veps\cdot\uu{\nu}\vert_{\D\Omega}=\uu{0},
\end{equation}
where $\uu{\nu}$ denotes the unit outward normal to the boundary 
$\D\Omega$. Throughout this work, we assume that the system
\eqref{eq:cons_mas}-\eqref{eq:cons_mom} is under the anelastic scaling
so that the pressure gradient and the gravitational force term are of
equal magnitude in the momentum equation \eqref{eq:cons_mom}.
In other words, the Mach and the Froude numbers are scaled by the same
power of the infinitesimal $\veps\in (0,1]$. Classical solutions to the
initial boundary value problem \eqref{eq:cons_mas}-\eqref{eq:eq_ic_bc}
are known to develop discontinuities in finite time, which necessitates
the definition of weak solutions. A weak solution of the system
\eqref{eq:cons_mas}-\eqref{eq:cons_mom} in the conservative variables
$\rho^\veps$ and $\uu{m}^\veps=\rho^\veps\uu{u}^\veps$ is defined
as follows.
\begin{definition}\label{defn:weak_soln}
  The pair $(\rho^\veps,\uu{m}^\veps)$ is a weak solution of
  \eqref{eq:cons_mas}-\eqref{eq:cons_mom} if
  \begin{enumerate}[label=(\roman*)]
  \item $\rho^\veps > 0$ a.e.\ in $[0,T) \times \Omega$, and the identity
    \begin{equation}
      \label{eq:cont_weak_soln_mas}
      \int_0^T\int_\Omega\left(\rho^\veps\D_t\varphi+
        \uu{m}^\veps\cdot\grd\varphi\right)\dd\uu{x}\dd t
      =-\int_\Omega\rho^\veps_0\phi(0,\cdot)\dd\uu{x}
    \end{equation}
    holds for all $\varphi\in C_c^\infty([0,T)\times\bar{\Omega})$.
  \item $\uu{m}^\veps=\uu{0}$ whenever $\rho^\veps=0$, and the identity
    \begin{equation}
      \label{eq:cont_weak_soln_mom}
      \begin{aligned}
        \int_0^T\int_\Omega\left(\uu{m}^\veps\cdot\D_t\uu{\varphi}+
          \left(\frac{\uu{m}^\veps\otimes\uu{m}^\veps}{\rho^\veps}\right)
          \colon\bgrd\uu{\varphi}+\frac{1}{\veps^2}p^\veps 
        \,
        \dive\uu{\varphi}\right)\dd\uu{x}\dd t&=
      -\int_\Omega\uu{m}^\veps_0\cdot\uu{\varphi}(0,\cdot)\dd\uu{x}
      \\
      &\quad+\frac{1}{\veps^2}\int_0^T\int_\Omega\rho^\veps
      \grd\phi\cdot\uu{\varphi}\,\dd\uu{x}\dd t
    \end{aligned}
    \end{equation}
    holds for all $\uu{\varphi}\in C_c^\infty([0,T)\times\bar{\Omega})^d$.
  \end{enumerate}
\end{definition}
The well-posedness of the Euler system
\eqref{eq:cons_mas}-\eqref{eq:cons_mom} in the class of weak
solutions as defined above is not well understood for general initial
data; see \cite{BF18a,BF18b,FKK+20} for more discussions. Throughout
the present work, however, we assume the existence of a weak solution
of \eqref{eq:cons_mas}-\eqref{eq:cons_mom} in order to carry out the
numerical analysis.    

\subsection{Hydrostatic Steady States}
\label{sec:hydstat}
In the numerical modelling of the Euler system with gravity, the
hydrostatic steady state wherein the velocities vanish and the
pressure gradient exactly balances the gravitational source term plays
a vital role. For the scaled Euler system
\eqref{eq:cons_mas}-\eqref{eq:cons_mom}, the hydrostatic steady state
is given by $\rho^\veps=\tilde{\rho}, \ \uu{u}^\veps=\uu{0}$, where 
the steady state density $\tilde{\rho}$ is independent of $\veps$ and
can be obtained from the equilibrium condition 
\begin{equation}
  \label{eq:hydrostat}
  \grd\tilde{p}=-\tilde{\rho}\grd\vphi.
\end{equation}
Here, $\tilde{p}=\wp(\tilde{\rho})$ denotes the hydrostatic
pressure. We can integrate \eqref{eq:hydrostat} to readily yield  
\begin{equation}
  \label{eq:int_pot}
  h_{\gamma}^{\prime}(\tilde{\rho})+\phi=C(m_0),
\end{equation}
where the so-called Helmholtz function
$h_{\gamma}\colon\mbb{R}^{+}\rightarrow \mbb{R}$ is defined by
\begin{equation}
  \label{eq:psi_gamma_a}
  h_{\gamma}(\rho):=\frac{\rho^\gamma}{\gamma -1},\;\rho\in\mbb{R}^{+},
\end{equation} 
and $C(m_0)$ is a constant of integration which can be determined by
the total mass $m_0=\int_{\Omega}\tilde{\rho}(\uu{x})\dd\uu{x}$ at
equilibrium. Note that $h_\gamma$ is a convex function and it satisfies the
following relation for all $\rho>0$:
\begin{equation}
  \label{eq:psi_gm_prop}
  \rho h_\gamma^{\prime}(\rho)-h_\gamma(\rho)=\wp(\rho).
\end{equation}
\begin{remark}
  For convenience, we assume that $\gamma>1$ throughout the present 
  work. The case $\gamma=1$ can be handled in a similar fashion. 
\end{remark}

\subsection{Apriori Energy Estimates}
\label{sec:engy_est}
We can derive the following apriori estimates for classical solutions
of the Euler system.    
\begin{proposition}
  \label{prop:cont_en_id}
  The following identities are satisfied by classical solutions of
  \eqref{eq:cons_mas}-\eqref{eq:cons_mom}.
  \begin{enumerate}[label=(\roman*)]
  \item A renormalisation identity:
    \begin{equation}
      \label{eq:cont_renorm}
      \Dt h_\gamma(\rho^\veps)
      +\dive\left(h_\gamma(\rho^\veps)\uu{u}^\veps\right)
      +p^\veps\dive\uu{u}^\veps=0.
    \end{equation}
  \item A positive renormalisation identity:
    \begin{equation}
      \label{eq:cont_porenorm}
      \Dt\Pi_\gm(\rho^\veps|\tilde{\rho})
      +\dive\left(h_\gamma(\rho^\veps)\uu{u}^\veps\right)
      +p^\veps\dive\uu{u}^\veps=h_\gamma^\prime(\tilde{\rho})\dive
      (\rho^\veps\uu{u}^\veps),
    \end{equation}
    where the relative internal energy
    $\Pi_\gm(\rho|\tilde{\rho}):=
    h_\gamma(\rho)-h_\gamma(\tilde{\rho})-h_{\gamma}^{\prime}
    (\tilde{\rho})(\rho-\tilde{\rho})$
    is an affine approximation of $h_\gamma$ with respect to the
    hydrostatic density $\tilde{\rho}$.
    
  \item The kinetic energy identity:
    \begin{equation}
      \label{eq:cont_kinbal}
      \Dt\Big(\frac{1}{2}\rho^\veps{\abs{\uu{u}^\veps}}^2\Big)
      +\dive\Big(\frac{1}{2}\rho^\veps{\abs{\uu{u}^\veps}}^2\uu{u}^\veps\Big)
      +\frac{1}{\veps^2}\grd p^\veps\cdot\uu{u}^\veps
      =-\frac{1}{\veps^2}\rho^{\veps}\grd \phi\cdot\uu{u}^\veps.
    \end{equation}
  \item The total energy identity:
    \begin{equation}
      \label{eq:cont_totbal}
      \Dt\Big(\frac{1}{2}\rho^\veps{\abs{\uu{u}^\veps}}^2+\frac{1}
      {\veps^2}\Pi_\gamma(\rho^\veps|\tilde{\rho})\Big)
      +\dive\bigg(\Big(\frac{1}{2}\rho^\veps{\abs{\uu{u}^\veps}}^2
      +\frac{1}{\veps^2}h_\gamma(\rho^\veps)
      +\frac{1}{\veps^2}p^\veps-\frac{1}{\veps^2}
      h_\gamma^{\prime}(\tilde{\rho})\rho^\veps\Big)\uu{u}^\veps\bigg)=0.
    \end{equation}
  \end{enumerate}  
\end{proposition}
\begin{proof}
  Proofs of \eqref{eq:cont_renorm}-\eqref{eq:cont_kinbal} are
  straightforward; see, e.g.\ \cite{HLS21} for details in the case of the
  Navier-Stokes system without source terms. We can obtain
  \eqref{eq:cont_totbal} upon multiplying \eqref{eq:cont_porenorm} by
  $\frac{1}{\veps^2}$, adding the resulting equation to
  \eqref{eq:cont_kinbal} and making use of the relation
  \eqref{eq:int_pot}.  
\end{proof}
\begin{remark}
  \label{rem:ent_wk_soln}
  In the case of weak solutions, the identity \eqref{eq:cont_totbal}
  remains as an inequality which represents the decay of mechanical
  energy. Consequently, throughout the rest of this paper, we will
  assume that the system \eqref{eq:cons_mas}-\eqref{eq:eq_ic_bc}
  admits energy stable weak solutions, i.e.\ weak solutions that
  satisfy the estimate 
  \begin{equation}
    \label{eq:cont_tot_entr_est}
    \frac{1}{2}\int_{\Omega}\rho^\veps\absq{\uu{u}^\veps}
    (t,\cdot)\dd\uu{x}
    +\frac{1}{\veps^2}\int_\Omega\Pi_\gamma(\rho^\veps\vert\tilde\rho)
    (t,\cdot)\dd\uu{x}
    \leq
    \frac{1}{2}\int_{\Omega}\rho^\veps_0\absq{\uu{u}^\veps_0}\dd\uu{x}
    +\frac{1}{\veps^2}\int_\Omega\Pi_\gamma(\rho^\veps_0\vert\tilde\rho)
    \dd\uu{x}, \; t\in(0,T) \, \mathrm{a.e.}.
  \end{equation}
\end{remark}
\begin{remark}
  \label{rem:rel_ent}
  The functional
  \begin{equation*}
    \mcal{E}(\rho^\veps,\uu{m}^\veps \vert
    \rho,\uu{u}):=\frac{1}{2}\int_\Omega\rho^\veps\Big\vert\frac{\uu{m}^\veps}{\rho^\veps}
        -\uu{u}\Big\vert^2\dd\uu{x}
    +\frac{1}{\veps^2}\int_\Omega\Pi_\gamma(\rho^\veps|\rho)\dd\uu{x}
  \end{equation*}
  is called the relative energy or entropy which measures the distance
  between a pair of solutions $(\rho^\veps,\uu{u}^\veps)$ and
  $(\rho,\uu{u})$. The relative energy plays a pivotal role in the
  subsequent analyses carried out.   
\end{remark}
  
\subsection{Well-Prepared Initial Data and the Anelastic Limit}
\label{sec:anelastic_limit}
The anelastic limit of the Euler system
\eqref{eq:cons_mas}-\eqref{eq:cons_mom} is obtained by taking limit
$\veps\to0$ of a sequence of weak solutions
$(\rho^\veps,\uu{u}^\veps)_{\veps>0}$. For the sake of simplicity, we
assume that the given initial data are well-prepared in the following
sense. 
\begin{definition}[Well-prepared initial data]
    \label{def:wp_id}
 An initial datum $(\rho^\veps_0, \uu{u}^\veps_0)\in
 L^\infty(Q_T)^{1+d}$ of the Euler system
 \eqref{eq:cons_mas}-\eqref{eq:cons_mom} is called well-prepared if
 \begin{enumerate}[label=(\roman*)]
 \item $\rho^\veps_0>0$ a.e.\ in $Q_T$;
 \item $ \rho^\veps_0=\tilde\rho + \veps r^\veps_0$, where
   $\{r^\veps_0\}_{\veps>0}\subset L^\infty(\Omega)$ is uniformly
   bounded and $r^\veps_0\to 0$ in $L^2(\Omega)$ as $\veps\to 0$;
 \item $\uu{u}^\veps_0\to \uu{U}_0$ in $L^2(\Omega)^d$ as $\veps\to 0$, with
 $\dive(\tilde\rho\uu{U}_0)=0$.
 \end{enumerate} 
\end{definition}
There are many rigorous studies on the anelastic approximation as a
singular limit of hydrodynamic models in the literature; see, e.g.\
\cite{BGL05, BF18a, BF18b, Cha22, FMN+08, FN09, Mas07}. Note that the
right hand side of the energy estimate \eqref{eq:cont_tot_entr_est}
remains uniformly bounded, given the initial data
$(\rho^\veps_0,\uu{u}^\veps_0)$ is well-prepared. Hence, if
$(\rho^\veps,\uu{u}^\veps)$ is an energy stable weak solution of
\eqref{eq:cons_mas}-\eqref{eq:eq_ic_bc} with respect to a
well-prepared initial datum, we have the uniform boundedness of the
relative energy from the estimate \eqref{eq:cont_tot_entr_est}. The
control of the relative internal energy functional $\Pi_\gamma$ leads
to a strong convergence of  the density
$\rho^\veps\to\tilde{\rho}$. Consequently, the uniform bound on the
kinetic energy yields a weak-* convergence of the momentum as
$\veps\to 0$. We summarise the essence of anelastic limit in the
following theorem and refer the reader to the above cited references
for an elaborate treatment.      
\begin{theorem}
  \label{thm:cont_anelastic_lim}
  Let $(\rho^\veps,\uu{m}^\veps)_{\veps>0}$ be a sequence of energy
  stable weak solutions of the Euler system
  \eqref{eq:cons_mas}-\eqref{eq:eq_ic_bc} with respect to a sequence
  of well-prepared initial data
  $(\rho^\veps_0,\uu{m}^\veps_0)_{\veps>0}$. Then the following holds.  
  \begin{enumerate}[label=(\roman*)]
  \item The relative energy $\E(\rho^\veps,\uu{m}^\veps\vert
    \tilde\rho, \uu{U})(t)$ satisfies
    \begin{equation}
      \label{eq:rel_ent_lim}
      \lim_{\veps\to 0}\sup_{t\in(0,T)}\E(\rho^\veps,\uu{m}^\veps\vert
      \tilde\rho, \uu{U})(t)=0;
    \end{equation}
  \item $\rho^\veps\to \tilde\rho$ in
  $L^\infty(0,T;(L^2+L^{\gamma^\prime})(\Omega))$ as $\veps \to 0$, with 
  $\gamma^\prime=\min\{2,\gamma\}$;
  \item  $\frac{\uu{m}^\veps}{\sqrt{\rho^\veps}}\to \sqrt{\tilde{\rho}}\uu{U}$ in
  $L^\infty(0,T;L^2(\Omega)^d)$ and $\uu{m}^\veps\to\tilde\rho\uu{U}$ weak-* in
  $L^\infty(0,T;(L^2+L^{\bar{\gamma}})(\Omega)^d)$ as $\veps\to 0$,
  with $\bar{\gamma}=\min\{\frac{4}{3},\frac{2\gamma}{\gamma+1}\}$,
  \end{enumerate}
  where $\uu{U}$ is the solution of the following anelastic Euler
  system on $Q_T$: 
  \begin{gather}
    \label{eq:anelastic_mom}
    \dive(\tilde{\rho}\uu{U})=0,  \\
    \label{eq:anelastic_vel}
    \Dt(\tilde{\rho}\uu{U})+\bdive(\tilde{\rho}\uu{U}\otimes\uu{U})+\tilde{\rho}\grd
    \pi =\uu{0},\\
    \uu{U}(0,.) = \uu{U}_0,\;\dive{(\tilde\rho\uu{U}_0)}=0.
    \label{eq:anelastic_id}
  \end{gather}
\end{theorem}
For the subsequent calculations, we note the following result which
shows that the stiff terms in the momentum balance \eqref{eq:cons_mom}
indeed converges to a term $\tilde\rho\grd\pi$ as $\veps\to 0$ in the
sense of distributions. A discrete counterpart of the same is
essential to establish the AP property of the numerical scheme.    
\begin{corollary}
  \label{cor:cont_stiff_lim}
  Let $(\rho^\veps,\uu{m}^\veps)_{\veps>0}$ be a sequence of energy
  stable weak solutions of the Euler system
  \eqref{eq:cons_mas}-\eqref{eq:eq_ic_bc} with respect to a sequence
  of well-prepared initial data $(\rho^\veps_0,\uu{m}^\veps_0)$. Then,
  as $\veps\to0$,
  \begin{equation*}
    \frac{1}{\veps^2}\iint\limits_{Q_T}(\grd
    p^\veps+\rho^\veps\grd\phi)\cdot\uu{\psi}\dd\uu{x}\dd t \to 0
    \ \mbox{for any} \ \uu{\psi}\in C^\infty_c([0,T]\times\bar\Omega) \
    \mbox{with} \ \dive(\tilde{\rho}\uu{\psi})=0. 
  \end{equation*}
\end{corollary}
\begin{proof}
  Let, $\uu{\psi}\in
    C^\infty_c([0,T]\times\bar\Omega)^d$ such that
    $\dive(\tilde\rho\uu{\psi})=0$. Adding and subtracting the
    hydrostatic pressure gradient $\grd\tilde{p}$ and subsequently
    recalling the steady state relation \eqref{eq:hydrostat}, we obtain 
    \begin{equation}
      \label{eq:cont_stiff_lim}
      \begin{aligned}
      \frac{1}{\veps^2}\iint\limits_{Q_T}(\grd
      p^\veps+\rho^\veps\grd\phi)\cdot\uu\psi\dd\uu{x}\dd t
      &= \frac{1}{\veps^2}\iint\limits_{Q_T}\grd(
      p^\veps-\tilde{p})\cdot\uu\psi\dd\uu{x}\dd t 
      +\frac{1}{\veps^2}\iint\limits_{Q_T}(\rho^\veps-\tilde{\rho})
      \grd\phi\cdot\uu\psi\dd\uu{x}\dd t\\
      &=
      -\frac{1}{\veps^2}\iint\limits_{Q_T}(p^\veps-\tilde p)\dive\uu{\psi}\dd\uu{x}\dd t
      +\frac{1}{\veps^2}\iint\limits_{Q_T}(\rho^\veps-\tilde\rho)
      \grd\phi\cdot\uu{\psi}\dd\uu{x}\dd t \\
      &=
      -\frac{\gamma-1}{\veps^2}\iint\limits_{Q_T}
      \Pi_\gamma(\rho^\veps\vert\tilde\rho)\dive\uu{\psi}\dd\uu{x}\dd t
      -\frac{1}{\veps^2}\iint\limits_{Q_T}\wp^\prime(\tilde{\rho})(\rho^\veps-\tilde{\rho})
      \dive\uu{\psi}\dd\uu{x}\dd t\\
      &\quad-\frac{1}{\veps^2}\iint\limits_{Q_T}(\rho^\veps-\tilde\rho)
      \frac{\grd\tilde{p}}{\tilde{\rho}}\cdot\uu{\psi}\dd\uu{x}\dd t\\
      &=-\frac{\gamma-1}{\veps^2}\iint\limits_{Q_T}
      \Pi_\gamma(\rho^\veps\vert\tilde\rho)\dive\uu{\psi}
      \dd\uu{x}\dd t
      -\frac{1}{\veps^2}\iint\limits_{Q_T}(\rho^\veps-\tilde{\rho})
      \frac{\wp^\prime(\tilde{\rho})}{\tilde{\rho}}\dive(\tilde{\rho}\uu{\psi})
      \dd\uu{x}\dd t.
     \end{aligned}
   \end{equation}
   Note that the first term on right hand side of the above relation
   converges to $0$ as $\veps\to 0$, following
   \eqref{eq:rel_ent_lim}. The second term on the right hand side
   vanishes as $\dive(\tilde\rho\uu{\psi})=0$.  
\end{proof}

\subsection{Velocity Stabilisation and Apriori Energy Estimates}
\label{subsec:apr_est}
The present work is aimed to design and analyse a finite volume
scheme that is energy stable, well-balanced for the hydrostatic
steady state and AP for the anelastic limit. The key to achieve these
desirable properties is to add a stabilising term to the convective
fluxes appearing in
\eqref{eq:cons_mas}-\eqref{eq:cons_mom}. Accordingly, we consider the 
modified system  
\begin{align}
  \D_t\rho^\veps+\dive (\rho^\veps(\uu{u}^\veps - \delta\uu{u}^\veps)) &=
                                                  0, \label{eq:cont_mas_stab}\\ 
  \D_t(\rho^\veps\uu{u}^\veps)+\bdive
  (\rho^\veps\uu{u}^\veps\otimes(\uu{u}^\veps - \delta\uu{u}^\veps))
  +\frac{1}{\veps^2}\grd
  p^\veps &=-\frac{1}{\veps^2}\rho^\veps\grd\phi. \label{eq:cont_mom_stab}  
\end{align}
We supplement the above system with the same initial and boundary
conditions \eqref{eq:eq_ic_bc}.
The expression for the stabilisation term $\delta\uu{u}^\veps$ will be 
derived using the total energy balance in such a way that the
stability or the decay of mechanical energy of the solutions is
ensured; see \cite{AGK23,AGK24,DVB17,DVB20} for analogous
considerations for the Euler systems with and without source
terms. The implications of the energy stability in terms of the
relative energy are multifold. Since the relative energy is defined
with respect to the hydrostatic steady state, an adaptation of the
same for a finite volume scheme first leads to the well-balancing
property. Second, it plays a crucial role in establishing the
asymptotic limit $\veps \to 0$ by providing uniform estimates on the
relative internal energy $\Pi_\gamma$ and the velocity components with
respect to $\veps$. In the following, we establish that the modified
system \eqref{eq:cont_mas_stab}-\eqref{eq:cont_mom_stab} is also
energy stable and that it admits the same hydrostatic steady
states. These observations form the basis for our analysis carried out
in Sections \ref{sec:enstab_wb} and
\ref{sec:dis_anelastic_limit}. Similar calculations as in
Proposition~\ref{prop:cont_en_id} readily yields the energy identities
satisfied by the velocity stabilised Euler system
\eqref{eq:cont_mas_stab}-\eqref{eq:cont_mom_stab}.  
\begin{proposition}
  \label{prop:engy_balance}
 The following identities are satisfied by classical solutions of
 \eqref{eq:cont_mas_stab}-\eqref{eq:cont_mom_stab}.
  \begin{enumerate}[label=(\roman*)]
  \item A renormalisation identity:
    \begin{equation}
      \label{eq:renorm}
      \Dt h_\gamma(\rho^\veps)
      +\dive\left(h_\gamma(\rho^\veps)(\uu{u}^\veps
        -\delta\uu{u}^\veps)\right)+p^\veps\dive(\uu{u}^\veps-\delta\uu{u}^\veps)=0.
    \end{equation}
  \item A positive renormalisation identity:
    \begin{equation}
      \label{eq:positive_renorm}
      \Dt\Pi_\gamma(\rho^\veps|\tilde{\rho})+\dive\big(h_\gamma(\rho^\veps)
      (\uu{u}^\veps-\delta\uu{u}^\veps)\big)
      +p^\veps\dive(\uu{u^\veps}-\delta\uu{u}^\veps)
      =h^\prime_\gamma(\tilde{\rho})\dive
      \left(\rho^\veps(\uu{u}^\veps-\delta\uu{u}^\veps)\right).
      \end{equation}
  \item A kinetic energy identity:
    \begin{equation}
      \label{eq:kinbal}
      \begin{aligned}
      \Dt\Big(\frac{1}{2}\rho^\veps{\abs{\uu{u}^\veps}}^2\Big)
      +\dive\Big(\frac{1}{2}\rho^\veps{\abs{\uu{u}^\veps}}^2
      (\uu{u}^\veps-\delta\uu{u}^\veps)\Big)
      +\frac{1}{\veps^2}\grd p^\veps\cdot(\uu{u}^\veps
      -\delta\uu{u}^\veps) &=-\frac{1}{\veps^2}
      \rho^\veps\grd\phi\cdot(\uu{u}^\veps-\delta\uu{u}^\veps)
      \\
      &-\frac{1}{\veps^2}\delta\uu{u}^\veps
      \cdot(\grd p^\veps + \rho^{\veps}\grd\phi).
      \end{aligned}
      \end{equation}
      \item The total energy identity:
    \begin{equation}
  \begin{aligned}
    \label{eq:cont_energy_balance}
    &\Dt\Big(\frac{1}{\veps^2}\Pi_\gm(\rho^\veps\vert\tilde\rho)
    +\frac{1}{2}\rho^\veps{\abs{\uu{u}^\veps}}^2\Big)
   \\
   &+\dive\bigg(\Big(\half\rho^\veps{\abs{\uu{u}^\veps}}^2
   +\frac{1}{\veps^2}h_\gamma(\rho^\veps)
   +\frac{1}{\veps^2}p^\veps-\frac{1}{\veps^2}h_\gamma^\prime(\tilde{\rho})
   \rho^\veps\Big)(\uu{u}^\veps-\delta\uu{u}^\veps)\bigg)
   =-\frac{1}{\veps^2}\delta\uu{u}^\veps\cdot(\grd p^\veps + \rho^{\veps}\grd\phi).
   \end{aligned}
  \end{equation}
    \end{enumerate}
\end{proposition}
\begin{remark}
  Thus, at the continuous level, we immediately see that if we formally take 
  \begin{equation}
    \label{eq:delta_u_eps}
    \delta\uu{u}^\veps=\frac{\eta^\veps}{\veps^2}(\grd p^\veps +
    \rho^\veps\grd\phi), \ \eta^\veps>0,
  \end{equation}   
  then we get the energy stability inequality:
  \begin{equation}
    \label{eq:r_eng_id_stab}
    \Dt\Big(\frac{1}{\veps^2}\Pi_\gm(\rho^\veps\vert\tilde\rho)
    +\frac{1}{2}\rho^\veps{\abs{\uu{u}^\veps}}^2\Big)
    +\dive\left(\Big(\half\rho^\veps{\abs{\uu{u}^\veps}}^2
      +\frac{1}{\veps^2}h_\gamma(\rho^\veps)+\frac{1}{\veps^2}p^\veps
      -\frac{1}{\veps^2}h_\gamma^\prime(\tilde{\rho})\rho^\veps\Big)
      (\uu{u}^\veps-\delta\uu{u}^\veps)\right)\leq 0.
\end{equation}
\end{remark}

\subsection{Hydrostatic Steady States of the Velocity Stabilised
  System} 
The steady states of the velocity stabilised system
\eqref{eq:cont_mas_stab}-\eqref{eq:cont_mom_stab} are given by
\begin{align}
  \dive (\rho^\veps(\uu{u}^\veps - \delta\uu{u}^\veps)) &=0,
                                                          \label{eq:cont_mas_steady}\\ 
  \bdive (\rho^\veps\uu{u}^\veps\otimes(\uu{u}^\veps -
  \delta\uu{u}^\veps))+\frac{1}{\veps^2}\grd p^\veps
&=-\frac{1}{\veps^2}\rho^\veps\grd\phi. \label{eq:cont_mom_steady}
\end{align}
Since we are interested in hydrostatic steady states, we put
$\uu{u^\veps}=\uu{0}$ in the above and readily infer from the momentum 
equation \eqref{eq:cont_mom_steady} that 
\begin{equation}
     \label{eq:hydstat_stab}
     \grd p^\veps = -\rho^\veps\grd\phi.
\end{equation}
Consequently, we get $\delta\uu{u}^\veps=\uu{0}$ using
\eqref{eq:delta_u_eps}. Hence, the mass equation
\eqref{eq:cont_mas_steady} implies that the hydrostatic steady state
solution $\rho^\veps=\tilde{\rho}$ and $\uu{u}^\veps=\uu{0}$ of the
Euler system \eqref{eq:cons_mas}-\eqref{eq:cons_mom} is also a
stationary solution of the modified Euler system
\eqref{eq:cont_mas_stab}-\eqref{eq:cont_mom_stab}. Note that the
velocity stabilisation term $\delta\uu{u}^\veps$ is consistent with
the steady state in the sense that it vanishes for steady state
solutions. Furthermore, to obtain the energy stability of the
solutions of the modified system
\eqref{eq:cont_mas_stab}-\eqref{eq:cont_mom_stab}, we enforce the
steady state state condition \eqref{eq:hydstat_stab} also on the
boundary of the domain, i.e.\
\begin{equation}
     \label{eq:hydstat_stab_bdry}
     \grd p^\veps +\rho^\veps\grd\phi=\uu{0} \ \mbox{on} \ \D\Omega.
\end{equation}
\begin{remark}
  Integrating \eqref{eq:r_eng_id_stab} on $Q_T$ and using the above
  boundary conditions yields the following energy stability for any $t
  \in (0,T)$: 
  \begin{equation}
    \label{eq:cont_toten_int}
    \half\int_{\Omega}\rho^\veps\absq{\uu{u}^{\veps}}(t,\cdot)\dd\uu{x}
    +\frac{1}{\veps^2}\int_{\Omega}\Pi_\gamma
    (\rho^\veps(t,\cdot)\vert\tilde{\rho})\dd\uu{x}\leq
    \half\int_{\Omega}\rho^\veps_0\absq{\uu{u}^{\veps}_0}\dd\uu{x}
    +\frac{1}{\veps^2}\int_{\Omega}\Pi_\gamma(\rho^\veps_0\vert\tilde{\rho})\dd\uu{x}.
  \end{equation}
  Suppose the initial data is in hydrostatic state, i.e.\ $\rho^\veps_0=\tilde{\rho}, \,
  \uu{u}^\veps_0=\uu{0}$. Since $\Pi_\gamma$ is positive and
  $\Pi_\gamma(a\vert b)=0$ if and only if $a=b$ for any
  $a,b\in\mbb{R}$, the right hand side of \eqref{eq:cont_toten_int}
  vanishes. Since the terms on the left hand side are non-negative, we
  immediately get  $\rho^\veps(t,\cdot)=\tilde{\rho}$ and
  $\uu{u}^\veps(t,\cdot)=\uu{0}$ for all $t>0$. In other words, the
  stabilised system preserves the hydrostatic steady state, provided
  the initial data are hydrostatic. The well-balancing property of our
  finite volume scheme follows an adaptation of this result that is
  ensured by the discrete energy stability.
 \end{remark}

\section{Space Discretisation and Discrete Differential Operators}
\label{sec:MAC_disc_diff}

In order to approximate the velocity stabilised Euler system
\eqref{eq:cont_mas_stab}-\eqref{eq:cont_mom_stab} in a finite volume
framework, we take a computational space-domain $\Omega\subseteq
\mbb{R}^d$ such that the closure of $\Omega$ is the union of closed
rectangles ($d=2$) or closed orthogonal parallelepipeds ($d=3$). In
this section, we briefly introduce a discretisation of the domain
$\Omega$ using a marker and cell (MAC) grid and the corresponding
discrete function spaces; see \cite{GHL+18, GHM+16, MN21} for more
details.  
\subsection{Mesh and Unknowns}
\label{subsec:msh_unkn}
A MAC grid is a pair $\mcal{T}=(\mcal{M},\mcal{E})$, where $\mcal{M}$
is called the primal mesh which is a partition of $\bar{\Omega}$
consisting of possibly non-uniform closed rectangles ($d=2$) or
parallelepipeds ($d=3$) and $\mcal{E}$ is the collection of all edges
of the primal mesh cells. For each $\sigma\in\mcal{E}$, we construct a
dual cell $\Ds$ which is the union of half-portions of the primal
cells $K$ and $L$, where $\s=\bar{K}\cap\bar{L}$. Furthermore, we
decompose $\E$ as $\mcal{E}=\cup_{i=1}^d\mcal{E}^{(i)}$, where
$\mcal{E}^{(i)}=\mcal{E}_\intr^{(i)}\cup\mcal{E}_\extr^{(i)}$. Here,
$\mcal{E}_\intr^{(i)}$ and $\mcal{E}_\extr^{(i)}$ are, respectively,
the collection of $(d-1)$-dimensional internal and external edges that
are orthogonal to the $i$-th unit vector $e^{(i)}$ of the canonical
basis of $\mbb{R}^d$. We denote by $\mcal{E}(K)$, the collection of
all edges of $K\in\mcal{M}$ and $\tilde{\mcal{E}}(D_\sigma)$, the
collection of all edges of the dual cell $D_\sigma$.

Now, we define a discrete function space $L_{\mcal{M}}(\Omega) \subset
L^{\infty}(\Omega)$, consisting of scalar valued functions which are
piecewise constant on each primal cell $K\in\mcal{M}$. Analogously, we
denote by $\uu{H}_{\mcal{E}}(\Omega)=\prod_{i=1}^{d}
H^{(i)}_{\mcal{E}}(\Omega)$, the set of vector valued (in $\mbb{R}^d$)
functions which are constant on each dual cell $D_\sigma$ for each
$\sigma\in\mcal{E}^{(i)}$, $i=1,2,\dots,d$. The space of vector
valued functions vanishing on the external edges is denoted as
$\uu{H}_{\mcal{E},0}(\Omega)=\prod_{i=1}^d
H^{(i)}_{\mcal{E},0}(\Omega)$, where  $H^{(i)}_{\mcal{E},0}(\Omega)$
contains those elements of $H^{(i)}_{\mcal{E}}(\Omega)$ which vanish
on the external edges. For a primal grid function $q\in
L_{\mcal{M}}(\Omega)$, such that $q =\sum_{K\in\mcal{M}}q_K\chark$,
and for each $\sigma = K|L \in\cup_{i=1}^d\mcal{E}^{(i)}_\intr$, the
dual average $q_{D_\sigma}$ of $q$ over $D_\sigma$ is defined via the
relation
\begin{equation}
  \label{eq:mass_dual}
  \abs{D_\sigma}q_{D_\sigma}=\abs{D_{\sigma,K}}q_K+\abs{D_{\sigma,L}}q_L.
\end{equation}
\begin{remark}
  All along the subsequent exposition we have assumed that the dual
  variables along with the gradients defined on the dual cells vanish on
  the boundary for the sake of simplicity of performing the
  analysis. However, we have implemented several other relevant
  boundary conditions while carrying out various numerical test cases
  in Section\,\ref{sec:num_tst}.

\end{remark}

\subsection{Discrete Convection Fluxes and Differential Operators}
\label{sec:dic_convect}

In this subsection we introduce the discrete convection fluxes and
discrete differential operators on the function spaces described
above. Let us assume that a discretisation of $\Omega$ is done using a
MAC grid $\mcal{T}=(\mcal{M},\mcal{E})$. Construction of the mass flux 
requires the following technical result which is stated below for
convenience and referred to \cite[Lemma 2]{GH+21} for proof.  
\begin{lemma}
  \label{lem:rho_sig}
  Let $\psi$ be a strictly convex and continuously differentiable
  function over an open interval $I$ of $\mathbb{R}$. If $\rho_K,
  \rho_L\in I$, then there exists a unique $\rho_{KL}\in\llbracket
  \rho_K,\rho_L\rrbracket$ such that
  \begin{equation}
    \psi(\rho_K) + \psi^\prime(\rho_K)(\rho_{KL}-\rho_L) = \psi(\rho_L)
    + \psi^\prime(\rho_L)(\rho_{KL}-\rho_L),\; \mbox{if}\;
    \rho_K\neq\rho_L,
  \end{equation}
  and $\rho_{KL}=\rho_K=\rho_L$, otherwise. In particular, for
  $\wp(\rho)=\rho^\gamma$, and for each $K,L\in\mcal{M}$, there exists
  a unique $\rho_{KL}\in \llbracket \rho_K,\rho_L\rrbracket$ such
  that 
  \begin{equation}
    \label{eq:rhosigchoice}
    \begin{aligned}
      \rho_K^\gamma - \rho_L^\gamma &=
      \rho_{KL}[h^\prime_\gamma(\rho_K)-h^\prime_\gamma(\rho_L)], \,
      \mbox{if } \rho_K\neq\rho_L, \\
      \rho_{KL}=\rho_K &=\rho_L, \, \mbox{otherwise}.
    \end{aligned}
  \end{equation}
\end{lemma}
\begin{remark}
  In the above lemma and throughout the rest of this paper, we denote
  by $\llbracket a,b\rrbracket$, the interval
  $[\min\{a,b\},\max\{a,b\}]$ for two real numbers $a$ and $b$.  
\end{remark}
\begin{definition}
  \label{def:disc_conv_flux}
  The mass and momentum convection fluxes are defined
  as follows.   
  \begin{itemize}
  \item For each $K \in \mcal{M}$, and $\sigma\in \mcal{E}(K)$, the
    mass convection flux $F_{\sigma,K} \colon L_{\mcal{M}}(\Omega)
    \times \uu{H}_{\mcal{E},0}(\Omega) \to \mbb{R}$ is defined by
    \begin{equation}
      \label{eq:mass_up}
      F_{\sigma,K}(\rho,\uu{u}):=\abs{\sigma}\rho_\sigma(u_{\sigma,
        K}-\delta u_{\sigma, K}), \ (\rho,\uu{u})\in
      L_{\mcal{M}}(\Omega) \times \uu{H}_{\mcal{E},0}(\Omega). 
    \end{equation}
    Here, $u_{\sigma , K}=u_{\sigma}
    \uu{e}^{(i)}\cdot\uu{\nu}_{\sigma, K}$ with $\uu{\nu}_{\sigma, K}$
    being the unit vector normal to the edge
    $\sigma=K|L\in\mcal{E}^{(i)}_{\mathrm{int}}$ in the direction
    outward to the cell $K$. The approximation of the density on each 
    interface $\sigma=K|L\in\mcal{E}^{(i)}_{\mathrm{int}}$,
    $i=1,2,\dots,d$, is obtained using Lemma\,\ref{lem:rho_sig} as
    $\rho_\sigma = \rho_{KL}$, where $\rho_{KL}$ is given by
    \eqref{eq:rhosigchoice}. The velocity stabilisation term
    $\delta\uu{u}\in\uu{H}_{\E,0}(\Omega)$ which dependents on
    $\rho$ and the discretised potential $\phi$ will be specified
    later.  
  \item For a fixed $i=1,2,\dots,d$, for each
    $\sigma\in\mcal{E}^{(i)}, \epsilon\in\bar{\E}_{D_\sigma}$, and
    $(\rho,\uu{u},v)\in \Lm\times \uu{H}_{\mcal{E},0}\times
    H^{(i)}_{\mcal{E},0} $, the upwind momentum convection flux is
    given by the expression 
    \begin{align}
      \label{eq:mom_flux_up} 
      \sum_{\epsilon\in\Eds}\fesig(\rho,\uu{u})v_{\epsilon, \mathrm{up}}.
    \end{align}
    Here, $\fesig(\rho,\uu{u})$ is the mass flux across the edge
    $\epsilon$ of the dual cell $\Ds$ which is taken to be $0$ if
    $\epsilon\in\mcal{E}_{\mathrm{ext}}$; otherwise, it is a
    suitable linear combination of the primal mass convection fluxes
    at the neighbouring edges with constant coefficients; see, e.g.\
    \cite{GHM+16} for more details.    
  \item For any dual face $\epsilon=\Ds|D_{\sigma^{\prime}}$ we have
    $F_{\epsilon,\sigma}(\rho,\uu{u})=-F_{\epsilon,\sigma^{\prime}}(\rho,\uu{u})$. Note
    that same is true for the primary mass fluxes, i.e.\ for any
    $\sigma=K|L$, we have
    $\fsigk(\rho,\uu{u})=-F_{\sigma^{\prime},K}(\rho,\uu{u})$. 
  \item In \eqref{eq:mom_flux_up}, $v_{\epsilon,\mathrm{up}}$ is
    determined by the following upwind choice:
    \begin{equation}
      \label{eq:mom_up}
      v_{\epsilon,\mathrm{up}}:=\begin{cases}
        v_{\sigma}, &\fesig(\rho,\uu{u})\geq 0,
        \\
        v_{\sigma^{\prime}}, &\mathrm{otherwise}.
    \end{cases}
  \end{equation}
\end{itemize}
\end{definition}
\begin{definition}[Discrete gradient, divergence and Laplacian operators]
  \label{def:disc_grad_div}
  The discrete gradient operator
  $\grd_{\mcal{E}}:L_{\mcal{M}}(\Omega)\rightarrow\uu{H}_{\mcal{E},0}(\Omega)$
  is defined by the map $q \mapsto
  \grd_{\mcal{E}}q=\Big(\D^{(1)}_{\mcal{E}}q,\D^{(2)}_{\mcal{E}}q,
  \dots,\D^{(d)}_{\mcal{E}}q\Big)$,
  where for each $i=1,2,\dots,d$, $\partial^{(i)}_{\mcal {E}}q$
  denotes
  \begin{equation}
    \label{eq:dis_grd}
    \partial^{(i)}_{\mcal {E}}q:=\sum_{\sigma\in
      \mcal{E}^{(i)}_\intr}(\partial^{(i)}_{\mcal{E}}q)_{\sigma}\mcal{X}_{D_{\sigma}},
    \ \mbox{with} \  (\partial^{(i)}_{\mcal{E}}q)_{\sigma}:=
    \frac{\abs{\sigma}}{\abs{D_\sigma}}(q_L-q_K)\uu{e}^{(i)}\cdot
    \uu{\nu}_{\sigma,K}, \; \sigma=K|L\in \mcal{E}^{(i)}_\intr.
  \end{equation}
  We set the gradient to zero on the external faces in all the
  subsequent analysis for the sake of ease. Note that the approximation
  of the interface values provided by Lemma \ref{lem:rho_sig} and the
  definition of discrete gradient \eqref{eq:dis_grd} further implies
  the following identity for any $q\in L_\M(\Omega)$ and $\gamma>1$: 
  \begin{equation}
    \label{eq:dis_pr_grd}
    (\D^{(i)}_\E q^\gamma)_\s = q_\s(\D^{(i)}_\E
    h^\prime_\gamma(q))_\s,\;
    \s\in\E^{(i)}_\mathrm{int},\;i=1,2,\dots,d.
  \end{equation}
  The discrete divergence operator
  $\dive_\M\colon L_\M(\Omega)\times\uu{H}_{\E,0}(\Omega)\rightarrow
  L_{\mcal{M}}(\Omega)$ is defined as $(q,\uu{v}) \mapsto \dive_\M
  (q\uu{v})=\sum_{K\in\M}(\dive_{\mcal{M}} (q\uu{v}))_K \mcal{X}_{K}$,
  where for each $K\in\M$, $(\dive_{\mcal{M}} \uu{v})_K $ denotes
  \begin{equation}
    \label{eq:dis_div_gen}
    (\dive_\M (q\uu{v}))_K
    :=\frac{1}{\abs{K}}\sum_{\sigma\in\mcal{E}(K)}\abs{\sigma} q_\s
    v_{\sigma,K},
  \end{equation}
  with $q_\s$ being defined as in Lemma \ref{lem:rho_sig}. In
  particular, for any $\uu{v}\in\uu{H}_{\E,0}(\Omega)$, we have
  $\dive_\M\uu{v}=\sum_{K\in\M}(\dive_{\mcal{M}}\uu{v})_K
  \mcal{X}_{K}\in L_\M(\Omega)$, where
  \begin{equation}
    \label{eq:dis_div}
    (\dive_\M \uu{v})_K
    :=\frac{1}{\abs{K}}\sum_{\sigma\in\mcal{E}(K)}\abs{\sigma}
    v_{\sigma,K}.
  \end{equation}
  Finally, the discrete Laplacian operator $\Delta_{\M}\colon L_{\M}(\Omega)\to
  L_{\M}(\Omega)$ is defined via the map $q\mapsto
  \Delta_{\M}q=\sum_{K\in\M}(\Delta_{\M} q)_{K}\mcal{X}_K\in L_{\M}(\Omega)$, where
  \begin{equation}
    \label{eq:dis_lap}
    (\Delta_{\M} q)_{K}:=(\dive_{\M}(\grd_{\E}q))_{K},\;\forall K\in\M.
  \end{equation}
\end{definition}

\begin{proposition}[Div-grad duality]
  For any $(p, q,\uu{v})\in L_\M(\Omega)\times
  L_\M(\Omega)\times\uu{H}_{\E,0}(\Omega)$, the following
  `gradient-divergence duality' holds for the discrete gradient and
  divergence operators: 
  \begin{equation}
    \label{eq:weighted_disc_dual}
    \sum_{K\in\M}\abs{K}p_K(\dive_\M(q\uu{v}))_K +
    \sum_{i=1}^d\sum_{\s\in\E^{(i)}_\mathrm{int}}\abs{D_\sigma}q_\s(\D^{(i)}_\E
    p)_\s v_\s=0.
  \end{equation}
  In particular for $q\equiv 1$, and any $(p,\uu{v})\in
  L_\M(\Omega)\times\uu{H}_{\E,0}(\Omega)$, we have
  \begin{equation}
    \label{eq:disc_dual}
    \sum_{K\in\M}\abs{K}p_K(\dive_\M\uu{v})_K +
    \sum_{i=1}^d\sum_{\s\in\E^{(i)}_\mathrm{int}}\abs{D_\sigma}(\D^{(i)}_\E p)_\s
    v_\s=0.
  \end{equation}
\end{proposition}

\section{The Finite Volume Scheme}
\label{sec:FVApp}

\subsection{A Semi-Implicit Scheme}
\label{subsec:semi_impl_schm}

Based on the finite volume discretisation and the corresponding
discrete operators defined in the previous section we now introduce a
semi-implicit scheme for the system
\eqref{eq:cont_mas_stab}-\eqref{eq:cont_mom_stab}.  

For the given potential $\phi$, the corresponding discrete
hydrostatic density $\tilde{\rho}$ is obtained as the cell average on
each control volume in the collection of
primal cells $\M$. Accordingly, by a slight abuse of notation,
$\tilde{\rho}=\sum_{K\in\M}
\tilde{\rho}_K\mcal{X}_K\in L_\M(\Omega)$ denotes the discrete
hydrostatic density, where
\begin{equation}
\label{eq:hyd_dens_dis}
    \tilde{\rho}_K :=
    \frac{1}{\abs{K}}\int_K\tilde{\rho}(\uu{x})\dd\uu{x}
    =\frac{1}{\abs{K}}\int_K\left(\frac{\gamma - 1}{\gamma} (C(m_0)-
      \phi(\uu{x}))\right)^{\frac{1}{\gamma -1}}\dd\uu{x},\; \forall K\in\M,
  \end{equation}
following the relations \eqref{eq:int_pot} and
\eqref{eq:psi_gamma_a}. Similarly,
$\phi=\sum_{K\in\M}\phi_K\mcal{X}_K\in L_\M(\Omega)$ denotes the
discrete approximation of the smooth potential $\phi$ given by
\begin{equation}
\label{eq:phi_dis}
    \phi_K := C(m_0)-h^\prime_\gamma(\tilde{\rho}_K),\; \forall K\in\M,
\end{equation}
following \eqref{eq:int_pot}.

Let $\delta t> 0$ be a constant time-step. An approximate solution
$(\rho^{n},\uu{u}^n)\in L_{\mcal{M}}(\Omega)\times
\uu{H}_{\mcal{E},0}(\Omega)$ to the stabilised Euler system
\eqref{eq:cont_mas_stab}-\eqref{eq:cont_mom_stab} at time $t_n =
n\dt$, for $1\leq n\leq N = \lfloor\frac{T}{\delta t}\rfloor$, is computed
inductively by the following semi-implicit scheme: 
\begin{align}
\label{eq:disc_mas_updt}
  &\frac{1}{\delta
    t}(\rho^{n+1}_{K}-\rho^{n}_{K})+\frac{1}{\abs{K}}
    \sum_{\sigma\in{\mcal{E}(K)}}F_{\sigma,K}(\rho^{n+1},\uu{u}^n)
    =0, \ \forall K  \in\M, \\
  \label{eq:disc_mom_updt}
  &\frac{1}{\delta
    t}\left(\rho^{n+1}_{D_\sigma}u^{n+1}_{\sigma}-\rho^n_{D_\sigma}
    u^{n}_{\sigma}\right)+\frac{1}{\abs{D_\sigma}}\sum_{\epsilon\in
    {\bar{\mcal{E}}(D_{\sigma})}}F_{\epsilon,\sigma}(\rho^{n+1}
    ,\uu{u}^n)u^{n}_{\epsilon,\mathrm{up}}  \\
  &\hspace*{3cm}+\frac{1}{\veps^2}(\D^{(i)}_{\mcal{E}}p^{n+1})_{\sigma}=
    -\frac{1}{\veps^2}\rho^{n+1}_\sigma (\partial^{(i)}_{\mcal{E}}\phi)_{\sigma},
    \ \forall\sigma \in\E^{(i)}_\mathrm{int}, \ i=1,2,\dots,d.\nonumber
\end{align}
We now define the implicit momentum stabilisation term present in the
mass flux, cf.\ Section\,\ref{sec:MAC_disc_diff}, in accordance with the
energy stability analysis as a discrete analogue of
\eqref{eq:delta_u_eps}, i.e.\
\begin{equation}
  \label{eq:imp_vel_stab}
  \delta u^{n+1}_\sigma := \frac{\eta\delta
    t}{\veps^2}\Big[(\partial^{(i)}_{\E}p^{n+1})_\sigma +
  \rho^{n+1}_\sigma(\partial^{(i)}_{\E}\phi)_\sigma\Big].
\end{equation}
The implicit density term appearing on the right hand side of
\eqref{eq:disc_mom_updt} for each
$\sigma\in\mcal{E}^{(i)}_\mathrm{int}$, $\sigma = K|L$, is also
determined using Lemma\,\ref{lem:rho_sig}.
\begin{remark}
  \label{rem:interface_den}
  Note that the implicit density term on the interfaces appearing in
  the source term is consistent with the mass and momentum fluxes. As
  we shall see in the subsequent analyses, this feature enables us to
  obtain the discrete counterpart of the apriori energy estimates in
  Proposition\,\ref{prop:engy_balance}, which yields the scheme's
  stability and finally leads to a consistent discretisation of the
  anelastic Euler system in the asymptotic limit. In addition, the
  same interface approximation of the density ensures that the scheme
  exactly satisfies a discrete analogue of the hydrostatic steady state 
\eqref{eq:hydrostat}. 
\end{remark}
\begin{remark}
  \label{rem:order}
  The proposed velocity stabilised scheme
  \eqref{eq:disc_mas_updt}-\eqref{eq:disc_mom_updt} is closely related
  to the so-called non-incremental projection schemes
  \cite{Cho68}. Consequently, the scheme remains first-order
  accurate. Nonetheless, we validate by a numerical experiment in
  Subsection\,\ref{subsec:stat_vort} that the order of accuracy is
  uniform in $\veps$. 
\end{remark}

Averaging the mass balance \eqref{eq:disc_mas_updt} over dual cells,
we obtain the following mass balance on each dual cell $D_\s$ using
the relation between the primal and dual convection fluxes as
described in Definition \ref{def:disc_conv_flux}:  
\begin{equation} 
    \label{eq:mas_dual}
    \frac{1}{\delta
      t}(\rho^{n+1}_{D_\sigma}-\rho^{n}_{D_\sigma})
    +\frac{1}{\abs{D_\sigma}}\sum_{\epsilon\in\bar{\mcal{E}}(D_\sigma)}
    F_{\epsilon,\sigma}(\rho^{n+1},\uu{u}^n)=0.
  \end{equation}
The discrete momentum update \eqref{eq:disc_mom_updt} and the dual
mass balance \eqref{eq:mas_dual} together yields the following update
of the velocity components:
\begin{equation}
  \label{eq:dis_vel_dual}
  \frac{u_\s^{n+1}-u_\s^n}{\dt}-\frac{1}{\abs{\Ds}}
  \sum_{\epsilon\in\tilde\E(\Ds)}F_{\epsilon,\s}(\rho^{n+1},\uu{u}^{n})^-
  \frac{u_{\s^{\prime}}^{n}-u_\s^n}{\rho_{\Ds}^{n+1}}
  +\frac{1}{\veps^2}\frac{1}{\rho_{\Ds}^{n+1}}\left((\partial_{\E}^{(i)}p^{n+1})_\sigma
  +\rho^{n+1}_\sigma(\partial_{\E}^{(i)}\phi)_\sigma\right)=0.
\end{equation}
Here, and throughout rest of the paper, 
$x^\pm=\half(\abs{x} \pm x)$ denotes the positive and negative parts of a real
number $x$ so that $x=x^+-x^-$ with $x^\pm\geq 0$. 
\subsection{Initialisation of the scheme} 
The initial approximations at $t_0=0$ is given by the cell averages of
the initial data, i.e.\  
\begin{align}
  (\rho^\veps)^{0}_{K} &=\frac{1}{\abs{K}}\int_{K}\rho_{0}^{\veps}(\uu{x})\dd\uu{x},
                 \ \forall K\in\M, \label{eq:disc_mas_id} \\
  (u^\veps)^{0}_\sigma &=
                 \frac{1}{\abs{D_\sigma}}\int_{D_\sigma}u_{0}^{\veps,(i)}(\uu{x})\dd\uu{x},
                         \ 
                 \forall \sigma\in \mcal{E}^{(i)}_{\mathrm{int}}. \label{eq:disc_vel_id} 
\end{align}
The following estimates are direct consequences of the well-prepared 
initial datum and its discretisation. 
\begin{lemma}
  \label{lem:wp_est}
  Let the initial data $(\rho^\veps_0, \uu{u}^\veps_0)$ be
  well-prepared in the sense of Definition\,\ref{def:wp_id}. Then, 
  \begin{enumerate}[label=(\roman*)]
  \item $(\rho^\veps)^0\in L_{\M}(\Omega)$
    satisfies the estimate
    \begin{equation}
      \label{eq:dis_wp_est}
      \frac{1}{\veps}\max_{K\in\mcal{M}}\abs{\rho^0_K -
        \tilde{\rho}_K}\leq M(\veps),
    \end{equation}
    where $M(\veps)\to 0$ as $\veps\to 0$;
  \item $\{(\uu{u}^\veps)^0\}_{\veps>0}\subset\uu{H}_{\E,0}(\Omega)$ is
    uniformly bounded.
  \end{enumerate}
\end{lemma}
\begin{remark}
  For the sake of brevity, we shall omit $\veps$ and denote the
  discretised initial data as $(\rho^0,\uu{u}^0)\in
  L_{\M}(\Omega)\times\uu{H}_{\E,0}(\Omega) $ in the subsequent
  analysis carried out in this section.
\end{remark}

\subsection{Existence of a Numerical Solution}
Having defined the finite volume scheme
\eqref{eq:disc_mas_updt}-\eqref{eq:disc_mom_updt}, we now proceed to 
establish its key features. Our first result is regarding the
existence and positivity of the density given by the semi-implicit
mass update \eqref{eq:disc_mas_updt} for a given data $(\rho^n,
\uu{u}^n)\in L_\M(\Omega)\times\uu{H}_{\E,0}(\Omega)$ with $\rho^n>0$
on $\Omega$. Since \eqref{eq:disc_mas_updt} is leads to a nonlinear
equation for $\rho^{n+1}$, we make use of a few results from
topological degree theory in finite dimensions \cite{Dei85} to
establish the existence of a positive solution.  
\begin{theorem}
  \label{thm:existence}
  Let $(\rho^n,\uu{u}^n)\in
  L_{\mcal{M}}(\Omega)\times\uu{H}_{\mcal{E},0}(\Omega)$ be such
  that $\rho^{n}>0$ on $\Omega$. Then there exists a time-step
  $\dt^{n}_{\mcal{T},\Omega}>0$
  depending only upon $\rho^n$, the space discretisation $\mcal{T}$,
  and the domain $\Omega$, such that
  for any $0<\dt\leq\dt^n_{\mcal{T},\Omega}$, the updates 
  \eqref{eq:disc_mas_updt}-\eqref{eq:disc_mom_updt} admits a solution
  $(\rho^{n+1},\uu{u}^{n+1})\in L_{\mcal{M}}(\Omega)\times
  \uu{H}_{\E,0}(\Omega)$ such that $\rho^{n+1}>0$. 
\end{theorem}
\begin{proof}
  Let $C^{n}_{\M,\Omega}$ be a positive constant defined as
  \begin{equation}
    \label{eq:n_upbd}
    C^{n}_{\M,\Omega} :=
    \abs{\Omega}\frac{\max_{K\in\M}\{\rho^n_K\}}{\min_{K\in\M}\{\abs{K}\}}. 
  \end{equation}
  We further define $\Lambda^n_{\M,\Omega}(\tilde{\rho})>0$ as
  \begin{equation}
    \label{eq:stab_upbd}
    \Lambda^n_{\M,\Omega}(\tilde{\rho}) := (C^{n}_{\M,\Omega})^{\gamma} +
    C^{n}_{\M,\Omega}\max_{K\in\M}\abs{h_\gamma^\prime(\tilde{\rho}_K)}. 
  \end{equation}
  Now, we choose a time-step $\dt^{n}_{\mcal{T},\Omega}>0$ so that
  \begin{equation}
\label{eq:suff_tstep}
\dt^{n}_{\mcal{T},\Omega}\max_{K\in\M}\{\partial K\}\Big(
C_{\M}\max_{\s\in\E_{\mathrm{int}}}\abs{u^{n}_\s} +
\frac{1}{\veps}\sqrt{2\eta
  C_{\mcal{T}}\Lambda^n_{\M,\Omega}(\tilde{\rho})}\Big)<
\min\{1,\zeta^{n}\}, 
\end{equation}
where $C_{\M} = \max_{K\in\M}\{\frac{1}{\abs{K}}\}$, $C_\mcal{T} =
C_{\M}\max_{\s\in\E_{\mathrm{int}}}\{\frac{1}{\abs{D_\s}}\}$ and
$\zeta^n = \frac{\min_{K\in\M}\{\rho^n_K\}}{3C^{n}_{\M,\Omega}}$. The
constant $\eta>0$ is taken as $\eta =
\eta_1\max_{\s\in\E_{\mathrm{int}}}\{\frac{1}{\rho^n_{D_\sigma}}\}$
with $\eta_1>3/2$. Let $C>C^{n}_{\M,\Omega}$ and let us define an open
bounded set $V\subset L_{\M}(\Omega)$ via  
\begin{equation}
\label{eq:setdef_exdom}
V :=\{\rho\in L_{\M}(\Omega)\colon 0<\rho_K<C,\; \forall K\in\M\}.
\end{equation}
For any $0<\dt\leq \dt^{n}_{\mcal{T}, \Omega}$, let us introduce a
function $H\colon [0,1]\times L_{\M}(\Omega)\rightarrow
L_{\M}(\Omega)$ such that 
\begin{align*}
  (\lambda, \rho) &\mapsto\sum_{K\in\M}H(\lambda,\rho)_K\mcal{X}_K,\;
                    \forall (\lambda, \rho)\in[0,1]\times L_\M(\Omega), 
  \\
  H(\lambda,\rho)_K &:= (\rho_K - \rho^n_K) +
                      \frac{\lambda\dt}{\abs{K}}\sum_{\s\in\E(K)}F_{\s,
                      K}(\rho, \uu{u}^n),\; \forall K\in\M. 
\end{align*}
Note that  $H(\lambda, \cdot)$ is a homotopy connecting $H(0, \cdot)$
and $H(1, \cdot)$. We first show that $H(\lambda, \cdot)\neq 0$ on
$\partial V$, for any $\lambda\in[0,1]$, which in turn would imply
that $H(\lambda, \cdot)$ has a non-zero topological degree with
respect to $V$. On the contrary, let us assume that $H(\lambda,
\rho^\lambda)=0$ for some $\lambda\in[0,1]$ and
$\rho^\lambda\in\partial V$. So we have that
\begin{equation}
\label{eq:lambda_soln}
(\rho^\lambda_K - \rho^n_K) +
\frac{\lambda\dt}{\abs{K}}\sum_{\s\in\E(K)}F_{\s, K}(\rho^\lambda,
\uu{u}^n) = 0, \; \forall K\in\M.
\end{equation}
Since $\rho^\lambda\in \partial V\subset \bar{V}$, taking sum over all
$K\in\M$ in $\eqref{eq:lambda_soln}$, we have, 
\begin{equation}
\label{eq:lambda_soln_ub}
0\leq \rho^\lambda_K\leq C^{n}_{\M,\Omega}<C,\; \forall K\in\M.
\end{equation}
Now, for a time-step $0<\dt\leq\dt^{n}_{\mcal{T}, \Omega}$, we further have from \eqref{eq:suff_tstep} and  \eqref{eq:lambda_soln_ub} that for each $K\in\M$
\begin{equation}
\label{eq:rho_n_strict}
\frac{1}{3}\rho^n_K >
\frac{\lambda\dt}{\abs{K}}\sum_{\sigma\in\Ek}\abs{F_{\s,K}(\rho^\lambda,u^n)},
\end{equation}
from which, taking \eqref{eq:lambda_soln} into account we have,
\begin{equation}
\label{eq:rho_lambda_strict}
\frac{1}{2}\rho^{\lambda}_K -
\frac{\lambda\dt}{\abs{K}}\sum_{\sigma\in\Ek}\abs{F_{\s,K}(\rho^\lambda,
  u^n)} > 0,\; \forall K\in\M. 
\end{equation}
From \eqref{eq:lambda_soln_ub} and \eqref{eq:rho_lambda_strict} we have
\begin{equation}
0<\rho^\lambda_K<C,\; \forall K\in\M,
\end{equation}
which contradicts the assumption that $\rho^\lambda\in \partial
V$. Hence, we have that for any $\lambda\in[0,1]$,
$H(\lambda,\cdot)\neq 0$ on
$\D V$ which implies that $\deg H(1,\cdot) = \deg H(0,\cdot)$ on
$V$. Since $\deg H(0,\cdot)\neq 0$, we conclude that $H(1,\cdot)$ has
a zero in $V$. In other words, there exists a solution $\rho^{n+1}\in
_\M(\Omega)$ to the mass update \eqref{eq:disc_mas_updt} satisfying
$\rho^{n+1}>0$. Subsequently, the updated velocity can be
explicitly determined from $\eqref{eq:disc_mom_updt}$.
\end{proof}

\section{Energy Stability and Well-Balancing}
\label{sec:enstab_wb}

In this section we establish the nonlinear energy stability of the
semi-implicit scheme developed in
Section\,\ref{subsec:semi_impl_schm}. To this end, we first prove the
discrete counter parts of the energy balances stated in
Proportion\,\ref{prop:cont_en_id}. 

\subsection{Internal Energy Balance}
\label{subsec:disc_int_pot}
The following lemma gives the discrete internal energy balance for the
scheme \eqref{eq:disc_mas_updt}-\eqref{eq:disc_mom_updt}. 
\label{lem:disc_int} 
\begin{lemma}[Renormalisation identify]
  Let $(\rho^{n+1},\uu{u}^{n+1})\in L_{\mcal{M}}(\Omega)\times
  \uu{H}_{\mcal{E},0}(\Omega)$ be a solution to the semi-implicit scheme
  (\ref{eq:disc_mas_updt})-(\ref{eq:disc_mom_updt}). Then for all
  $K\in\M$, the following holds:
  \begin{equation}
    \label{eq:disc_int}
    \frac{\abs{K}}{\delta t}\big
    (h_\gamma(\rho^{n+1}_K)-h_\gamma(\rho^{n}_K)\big)+
    h^\prime_\gamma(\rho^{n+1}_K)\sum_{\sigma\in\mcal{E}(K)}F_{\sigma,K}(\rho^{n+1},
    \uu{u}^n) + \mcal{R}_{K,\delta t}^{n+1,n}=0, 
  \end{equation}
  where $\mcal{R}_{K,\delta t}^{n+1,n}\geq 0$.
\end{lemma}
\begin{proof}
  Multiplying the discrete mass update \eqref{eq:disc_mas_updt} by
  $\abs{K}h_{\gamma}^{\prime}(\rho_{K}^{n+1})$ and using a
  second-order Taylor expansion yields 
  \begin{align}
    0 &=\frac{\abs{K}}{\delta
        t}h_{\gamma}^{\prime}(\rho_{K}^{n+1})(\rho_K^{n+1}-\rho_K^n)+h_{\gamma}^{\prime}(\rho_{K}^{n+1})\sum_{\sigma\in\mcal{E}(K)}F_{\sigma,K}(\rho^{n+1},\uu{u}^n) 
    \\ 
      &=\frac{\abs{K}}{\delta
        t}\big(h_{\gamma}^{\prime}(\rho_K^{n+1})-h_{\gamma}^{\prime}(\rho_K^n)\big)+h^\prime_\gamma(\rho^{n+1}_K)\sum_{\sigma\in\mcal{E}(K)}F_{\sigma,K}(\rho^{n+1},
        \uu{u}^n) \nonumber \\
      &\quad+\half\frac{\abs{K}}{\delta
        t}h_{\gamma}^{\prime\prime}(\bar{\rho}_K^{n+\half})(\rho_K^{n+1}-\rho_K^n)^2,\nonumber 
  \end{align}
  where
  $\bar{\rho}_K^{n+\half}\in\llbracket
  \rho_K^{n+1},\rho_K^n\rrbracket$. The non-negativity of the
  remainder term $\mcal{R}_{K,\dt}^{n+1,n}$ follows from the convexity
  of $h_\gamma$.  
\end{proof} 
The following discrete relative internal energy balance
follows as an immediate corollary to Lemma\,\ref{lem:disc_int}. 
\begin{corollary}[Positive renormalisation identify]
  \label{cor:disc_int_pot_entr}
  Let $(\rho^{n+1},\vel^{n+1})\in\Lm\times\Hez$ be a solution to the
  semi-implicit scheme
  \eqref{eq:disc_mas_updt}-\eqref{eq:disc_mom_updt}. Then for all
  $K\in\M$, the following holds:
  \begin{equation}
    \label{eq:disc_int_pot_entr}
    \begin{aligned}
      &\frac{\abs{K}}{\veps^2\delta
        t}\big(\Pi_{\gamma}(\rho^{n+1}_{K}|\tilde{\rho}_K)-\Pi_{\gamma}(\rho^{n}_{K}|\tilde{\rho}_K)\big)+\frac{1}{\veps^2}h_{\gamma}^{\prime}(\rho^{n+1}_K)\sum_{\sigma\in\mcal{E}(K)}F_{\sigma,K}(\rho^{n+1},\uu{u}^n) + \mcal{R}_{K,\delta t}^{n+1,n}\\
      &=\frac{1}{\veps^2}h_{\gamma}^{\prime}(\tilde{\rho}_K)\sum_{\sigma\in\mcal{E}(K)}F_{\sigma,K}(\rho^{n+1},\uu{u}^n),
    \end{aligned}
  \end{equation}
  where $\mcal{R}_{K,\delta t}^{n+1,n}\geq 0$.
\end{corollary}  
\subsection{Kinetic Energy Estimate}
\label{subsec:kin}
In the following lemma we establish a discrete kinetic energy estimate
for the numerical scheme.
\begin{lemma}[Discrete kinetic energy identity]Any solution to the
  system \eqref{eq:disc_mas_updt}-\eqref{eq:disc_mom_updt} satisfies
  the following equality for $1\leq i\leq d,\;
  \s\in\E_\mathrm{int}^{(i)}$ and $0\leq n\leq{N-1}$:
  \begin{equation}
    \label{eq:dis_kinbal}
    \begin{aligned}
      & \frac{1}{2}\frac{\abs{\Ds}}{\dt}\big(\rho_{\Ds}^{n+1}(u_\s^{n+1})^2-\rho_{\Ds}^{n}(u_\s^{n})^2\big)+\sum_{\epsilon\in\tilde\E(\Ds)}F_{\epsilon,\s}(\rho^{n+1},\uu{u}^{n})\frac{|u_\epsilon^{n}|^2}{2}+\frac{1}{\veps^2}\abs{\Ds}(\D_\E^{(i)}p^{n+1})_\sigma u_\s^n
      \\
      &+\mcal{R}_{\s,\dt}^{n+1}=-\frac{1}{\veps^2}\abs{\Ds}\rho^{n+1}_\sigma (\D_\E^{(i)}\phi)_\s u_\s^n,
    \end{aligned}
  \end{equation}
\end{lemma}
where the remainder term $\mcal{R}_{\s,\dt}^{n+1}$ is defined by
\begin{equation}
  \label{eq:Rsn+1}
  \mcal{R}_{\sigma,\dt}^{n+1}= -\frac{\abs{D_\sigma}}{2\dt}\rho_{D_\sigma}^{n+1}(u_\sigma^{n+1}-u_\sigma^n)^2 +\sum_{\substack{\epsilon\in\tilde\E(\Ds)\\\epsilon=\Ds|D_{\s^{\prime}}}}F_{\epsilon,\s}(\rho^{n+1},\uu{u}^{n})^{-}\frac{(u_\s^n-u_{\s^{\prime}}^n)^2}{2}.
\end{equation}
\begin{proof}
  Multiplying the momentum balance equation \eqref{eq:disc_mom_updt} by
  $\abs{\Ds}u_\s^n$ and using the dual mass balance
  \eqref{eq:mas_dual} yields the identity. 
\end{proof}
\subsection{Discrete Total Energy Estimate}
Following the discrete relative internal energy balance
\eqref{eq:disc_int_pot_entr} and the discrete kinetic energy balance
\eqref{eq:dis_kinbal}, we now proceed to prove a discrete counterpart
of the energy stability \eqref{eq:cont_toten_int} in the following
theorem. 
\begin{theorem}[Total energy balance]
  \label{thm:dis_totbal}
  Any solution to the system
  \eqref{eq:disc_mas_updt}-\eqref{eq:disc_mom_updt} satisfies the
  following energy inequality:
  \begin{equation}
    \label{eq:dis_totbal}
    \frac{1}{\veps^2}\sum_{K\in\M}|K|\Pi_\gamma(\rho_{K}^{n+1}|\tilde{\rho}_K)
    +\sum_{\s\in\E_\mathrm{int}}\abs{\Ds}\frac{1}{2}
    \rho_{\Ds}^{n+1}(u_{\s}^{n+1})^2 \leq 
    \frac{1}{\veps^2}\sum_{K\in\M}|K|\Pi_\gamma(\rho_{K}^{n}|\tilde{\rho}_K)
    +\sum_{\s\in\E_\mathrm{int}}|\Ds|\frac{1}{2}\rho_{\Ds}^{n}(u_{\s}^{n})^2,
  \end{equation}
  under the time-step restrictions $\forall\s \in\E_\mathrm{int}^{(i)}, \ 1 \leq i \leq d$:
  \begin{equation}
    \label{eq:TotEn_tstep}
    \displaystyle
    \frac{\dt}{|\Ds|}\sum_{\epsilon\in\tilde\E(\Ds)}\frac{F_{\epsilon,\s}
      (\rho^{n+1},\uu{u}^n)^-}{\rho_{\Ds}^{n+1}}\leq\frac{1}{2},
  \end{equation}
  and the condition on $\eta$ such that $\forall\s
  \in\E_\mathrm{int}^{(i)}, \ 1 \leq i \leq d$: 
  \begin{equation}
\label{eq:TotEn_eta}
\eta\geq \frac{1}{\rho^{n+1}_{D_\sigma}}.
\end{equation}
\end{theorem}
\begin{proof} 
  We take sums over $K\in\M$ in \eqref{eq:disc_int_pot_entr} and over
  $\s\in\E_\mathrm{int}$ in \eqref{eq:dis_kinbal}. Upon adding the
  resulting expressions and recalling \eqref{eq:dis_pr_grd},
  \eqref{eq:dis_grad_psigm} and the discrete duality
  \eqref{eq:weighted_disc_dual} we obtain
  \begin{equation}
    \label{eq:thm_TE_eq3}
    \begin{aligned}
      &\frac{1}{\veps^2\dt}\sum_{K\in\M}|K|\Big(\Pi_\gamma(\rho_{K}^{n+1}|\tilde{\rho}_K)-\Pi_\gamma(\rho_{K}^{n}|\tilde{\rho}_K)\Big)+\sum_{\s\in\E_\mathrm{int}}\frac{1}{2}\frac{|\Ds|}{\dt}\Big(\rho_{\Ds}^{n+1}(u_\s^{n+1})^2-\rho_{\Ds}^{n}(u_\s^n)^2\Big)\\
      &\quad+\mcal{R}_{\M,\dt}^{n+1}+\mcal{R}_{\E,\dt}^{n+1}=-\frac{\eta\dt}{\veps^4}\sum_{i=1}^d\sum_{\s\in\E_\mathrm{int}^{(i)}}|\Ds|\left((\D_\E^{(i)}p^{n+1})_\s+\rho^{n+1}_\sigma (\D_\E^{(i)}\phi)_\sigma\right)^2.
    \end{aligned}
  \end{equation}
  Here, the global remainder terms $\mcal{R}^{n+1}_{\M,\dt}$ and
  $\mcal{R}^{n+1}_{\E,\dt}$ are 
  obtained by summing the local remainders $\mcal{R}^{n+1}_{K,\dt}$ and
  $\mcal{R}^{n+1}_{\s,\dt}$, respectively. Since $\mcal{R}^{n+1}_{\M,\dt}$ is non-negative
  unconditionally, the required stability estimate
  \eqref{eq:dis_totbal} follows from \eqref{eq:thm_TE_eq3} once we
  enforce the non-negativity of $\mcal{R}^{n+1}_{\E,\dt}$. To this end, we use the
  elementary inequality $(a+b)^2\leq 2a^2 + 2b^2,\, a,b\in\mbb{R}$, 
  in the velocity update \eqref{eq:dis_vel_dual} to yield   
  \begin{equation}
    \label{eq:thm_TE_eq1}
    \begin{aligned}
      \frac{1}{2}\rho_{\Ds}^{n+1}|u_\s^{n+1}-u_\s^n|^2&\leq\frac{(\dt)^2}{|\Ds|^2\rho_{\Ds}^{n+1}}\Bigg|-\sum_{\epsilon\in\tilde\E(\Ds)}(u_{\s^{\prime}}^n-u_\s^n)F_{\epsilon,\s}(\rho^{n+1},\uu{u}^n)^-\Bigg|^2\\
      &\quad+\frac{(\dt)^2}{\veps^4\rho^{n+1}_{\Ds}}\left((\D_\E^{(i)}p^{n+1})_\s+\rho^{n+1}_\sigma(\D_\E^{(i)}\phi)_\sigma\right)^2.
    \end{aligned}
  \end{equation}
  Upon an application of the Cauchy-Schwarz inequality, the first term
  on the right hand side of \eqref{eq:thm_TE_eq1} yields
  \begin{equation}
    \label{eq:rem_TE_CS}
    \begin{aligned}
    \Bigg|-\sum_{\epsilon\in\tilde\E(\Ds)}(u_{\s^{\prime}}^n-u_\s^n)
    F_{\epsilon,\s}(\rho^{n+1},\uu{u}^n)^-\Bigg|^2 &\leq
    \left(\sum_{\epsilon\in\E(\Ds)}\frac{F_{\epsilon,\s}(\rho^{n+1},\uu{u}^{n})^{-}}
      {\rho_{\Ds}^{n+1}}\right)\\
    &\times\left(\sum_{\substack{\epsilon\in\bar\E(\Ds)\\\epsilon=\Ds|D_{\s^{\prime}}}}
      F_{\epsilon,\sigma}(\rho^{n+1},\uu{u}^{n})^-(u_\s^{n}-u_{\s^\prime}^{n})^2\right).
    \end{aligned}
  \end{equation}
  Using the above inequality and grouping the like-terms on the right
  hand sides of \eqref{eq:thm_TE_eq1}-\eqref{eq:rem_TE_CS}, we finally
  obtain from \eqref{eq:thm_TE_eq3}, cf.\ also \cite{AGK23, DVB17}, the
  following inequality: 
  \begin{equation}
    \label{eq:dis_totenbal_press}
    \begin{aligned}
      &\frac{1}{\veps^2\dt}\sum_{K\in\M}|K|\Big(\Pi_\gamma(\rho_{K}^{n+1})-\Pi_\gamma(\rho_{K}^{n})\Big)+\sum_{\s\in\E_\intr}\frac{1}{2}\frac{|\Ds|}{\dt}\Big(\rho_{\Ds}^{n+1}(u_\s^{n+1})^2-\rho_{\Ds}^{n}(u_\s^n)^2\Big)\\
      &+\frac{\dt}{\veps^4}\sum_{i=1}^d\sum_{\s\in\E_\mathrm{int}^{(i)}}|\Ds|\left(\eta-\frac{1}{\rho^{n+1}_{\Ds}}\right)\left((\D_\E^{(i)}p^{n+1})_\s+\rho^{n+1}_\sigma (\D_\E^{(i)}\phi)_\sigma\right)^2\\
      &\leq\sum_{\s\in\E_\intr}\left(\frac{\dt}{|\Ds|}\sum_{\epsilon\in\E(\Ds)}\frac{F_{\epsilon,\s}(\rho^{n+1},\uu{u}^{n})^{-}}{\rho_{\Ds}^{n+1}}-\half\right)\left(\sum_{\substack{\epsilon\in\bar\E(\Ds)\\\epsilon=\Ds|D_{\s^{\prime}}}}F_{\epsilon,\sigma}(\rho^{n+1},\uu{u}^{n})^-(u_\s^{n}-u_{\s^\prime}^{n})^2\right).
    \end{aligned}
  \end{equation}
  The third term on the left hand side of the above estimate remains
  non-negative given that the parameter $\eta$ satisfies the condition
  \eqref{eq:TotEn_eta}, whereas the right hand side remains non-positive
  under the time-step restriction \eqref{eq:TotEn_tstep}.
\end{proof}
\begin{remark}
  For analogous treatments on momentum and velocity stabilised schemes
  in the context of the shallow water and barotropic Euler systems
  without source terms we refer to \cite{AGK23, DVB17, DVB20}.  
\end{remark}
\begin{remark}
  Note that the stability conditions \eqref{eq:TotEn_tstep} and
  \eqref{eq:TotEn_eta} are implicit in nature as they involve
  $\rho^{n+1}$. Nevertheless, proceeding as in \cite{AGK23, DVB20}, we
  can obtain a sufficient condition on the time-step which satisfies the
  requirement \eqref{eq:TotEn_tstep} and in addition yields  
  \begin{equation}
    \label{eq:rho_3/2}
    \frac{\rho_{\Ds}^n }{\rho_{\Ds}^{n+1}}\leq \frac{3}{2},
  \end{equation}
  from the dual mass balance \eqref{eq:mass_dual}. Note that
  \eqref{eq:rho_3/2} indeed enables us to make an explicit choice of
  $\eta$ as $\eta=\eta_1/\rho^n_{D_\s}$ at each interface $\s$, where
  $\eta_1>3/2$.   
\end{remark}

\subsection{Well-Balancing Property of the Semi-implicit Scheme}
In this section we discuss in detail the well-balancing property of
the semi-implicit scheme
\eqref{eq:disc_mas_updt}-\eqref{eq:disc_mom_updt}. Using the discrete
relative energy estimate \eqref{eq:dis_totbal} we show that the scheme
exactly satisfies and preserves a discrete form of the hydrostatic steady state. 

Upon direct verification we can immediately observe that the scheme
\eqref{eq:disc_mas_updt}-\eqref{eq:disc_mom_updt} admits the discrete
equilibrium solution   
\begin{equation}
  \begin{aligned}
    \rho^{n}_{K}&=\tilde{\rho}_{K},\;\forall K\in\mcal{M}, \;  0\leq n
    \leq N, \\
    u^{n}_{\sigma}&=0,\;\forall \sigma\in \mcal{E},\;  0\leq n \leq N, 
  \end{aligned}\label{eq:disc_stat_soln}
\end{equation} 
with $\tilde{\rho}\in L_{\mcal{M}}(\Omega)$ satisfying the
following discrete analogue of the hydrostatic steady state
\eqref{eq:hydrostat}:
\begin{equation}
\label{eq:disc_hydrostat}
(\partial^{(i)}_{\mcal{E}}\tilde{p})_{\sigma}=-\tilde{\rho}_{\sigma}
(\partial^{(i)}_{\mcal{E}}\phi)_{\sigma},\;\forall\sigma\in\mcal{E}^{(i)}_{\mathrm{int}},\;  
i=1,2,\dots,d.  
\end{equation}
In the above, $\tilde{p}=\sum_{K\in\mcal{M}}\tilde{p}_{K}\mcal{X}_K \in
L_{\mcal{M}}(\Omega)$ is a discretisation of the hydrostatic pressure,
i.e.\ $\tilde{p}_{K}=\wp(\tilde{\rho}_K),\; \forall K\in\mcal{M}$, and for
each $\sigma=K|L\in\mcal{E}^{(i)}_{\mathrm{int}},\;i=1,2,\dots,d$, the
interface hydrostatic density is given by $\tilde{\rho}_\s
= \tilde{\rho}_{KL}$, as in Lemma\,\ref{lem:rho_sig}. Recalling the
identity \eqref{eq:dis_pr_grd} and upon dividing by $\tilde{\rho}_\s$,
the relation \eqref{eq:disc_hydrostat} further yields
\begin{equation}
    \label{eq:dis_grad_psigm}
    (\D^{(i)}_{\E}(h^\prime_\gamma(\tilde{\rho}) + \phi))_\s =
    0,\;\forall\s\in\E^{(i)}_\mathrm{int},\;i=1,2,\dots,d,  
\end{equation}
which serves as a discrete counter part of the identity
\eqref{eq:int_pot}. Note that the discrete stabilisation parameter is 
consistent with the steady state in the sense that
$\delta\uu{u}^{n+1}=0$ whenever $\rho^{n+1}=\tilde{\rho}$. The
relative energy estimate \eqref{eq:dis_totbal} with respect to the
above discrete hydrostatic steady state gives a distance between the
successive numerical solutions and the stationary counterparts. Given
the initial condition as the steady state solution, the estimate
\eqref{eq:dis_totbal} is satisfied if and only if the discrete steady
state remains preserved in the subsequent updates.
\begin{theorem}
  \label{prop:well_balancing}
  The semi-implicit scheme has the following properties.
  \begin{enumerate}[label=(\roman*)]
  \item Suppose $(\rho^{n},\uu{u}^n)\in L_{\mcal{M}}(\Omega)\times
    \uu{H}_{\mcal{E},0}(\Omega)$, where $\rho^{n}>0$, and 
    $(\rho^{n+1},\uu{u}^{n+1})\in L_{\mcal{M}}(\Omega)\times
    \uu{H}_{\mcal{E},0}(\Omega)$ solves the semi-implicit scheme
    \eqref{eq:disc_mas_updt}-\eqref{eq:disc_mom_updt}. Then
    $\rho^{n+1}_{K}=\tilde{\rho}_{K},\,\forall K\in\mcal{M}$, and
    $u^{n+1}_{\sigma}=0,\,\forall \sigma\in\mcal{E}^{(i)},\, i = 1,
    2,\dots, d$,
    if and only if
    \begin{equation}
      \label{eq:zero_energy}
      \frac{1}{\veps^2}\sum_{K\in\M}|K|\Pi_\gamma(\rho_{K}^{n+1}|\tilde{\rho}_K)
      +\sum_{\s\in\E_\intr}\abs{\Ds}\half\rho_{\Ds}^{n+1}(u_{\s}^{n+1})^2
      = 0.
    \end{equation}
  \item The semi implicit scheme is well-balanced in the following sense:
    given the initial data $\rho^{0}=\tilde{\varrho}$,
    $\uu{u}^{0}_{\veps}=\uu{0}$, the scheme admits the unique
    stationary solution $(\tilde{\rho}, \uu{0})$.  
  \end{enumerate}
\end{theorem}
\begin{proof}
  The result in (i) readily follows from Theorem\,\ref{thm:dis_totbal}
  on energy stability along with the fact that
  $\Pi_\gamma(\rho_1|\rho_2)\geq 0$ for any $\rho_1, \rho_2\in\mbb{R}$
  and $\Pi_\gamma(\rho_1|\rho_2)=0$ if and only if
  $\rho_1=\rho_2$. Upon applying induction on the time-steps we obtain
  (ii) as a direct consequence of (i).  
\end{proof}
\section{Weak Consistency of the Scheme}
\label{sec:weak_cons}

Goal of this section is to show a Lax-Wendroff-type weak consistency
of the scheme \eqref{eq:disc_mas_updt}-\eqref{eq:disc_mom_updt}. We
consider a sequence of numerical solutions generated upon successive
mesh refinements, and satisfying certain boundedness assumptions. If
the sequence of solutions converges strongly to a limit, then the
limit must be a weak solution of the Euler equations. Considering a
bounded initial data $(\rho^{\veps}_0,\uu{u}^{\veps}_0)\in
L^\infty(\Omega)^{1+d}$ and recalling the definition of a weak
solution, cf.\ Definition~\ref{defn:weak_soln}, we have the following
Lax-Wendroff-type consistency formulation for the semi-implicit scheme
\eqref{eq:cons_mas}-\eqref{eq:cons_mom}. In the theorem, we take the
liberty to suppress $\veps$ wherever necessary for the sake of
simplicity of writing. 
\begin{theorem}
\label{thm:weak_cons}
Let $\Omega$ be an open bounded set of $\mbb{R}^d$. Assume that
$\big(\mcal{T}^{(m)},\delta t^{(m)}\big)_{m\in\mbb{N}}$ is a sequence
of discretisations such that both $\lim_{m\rightarrow +\infty}\delta
t^{(m)}$ and $\lim_{m\rightarrow +\infty}h^{(m)}$ are $0$. Let
$\big(\rho^{(m)},\uu{u}^{(m)}\big)_{m\in\mbb{N}}$ be the corresponding
sequence of discrete solutions with respect to an initial data
$(\rho^\veps_0,\uu{u}^\veps_0)\in L^\infty(\Omega)^{1+d}$. We assume
that $(\rho^{(m)},\uu{u}^{(m)})_{m\in\mbb{N}}$ satisfies the
following.
\begin{enumerate}[label=(\roman*)]
\item $\big(\rho^{(m)},\uu{u}^{(m)}\big)_{m\in\mbb{N}}$ is uniformly
  bounded in $L^\infty(Q)^{1+d}$, i.e.\ 
\begin{align}
  \underbar{C}<(\rho^{(m)})^n_K &\leq \bar{C},
  \ \forall K\in\mcal{M}^{(m)}, \ 0\leq n\leq N^{(m)},
  \ \forall m\in\mbb{N}\label{eq:dens_abs_bound}, \\
|(u^{(m)})^n_\sigma| &\leq C, \ \forall \sigma\in\mcal{E}^{(m)}, \
 0\leq n\leq N^{(m)}, \ \forall
 m\in\mbb{N},\label{eq:u_abs_bound}
\end{align}
where $\underbar{C}, \bar{C}, C>0$ are constants independent of the
discretisations.  
\item $\big(\rho^{(m)},\uu{u}^{(m)}\big)_{m\in\mbb{N}}$ converges to
  $(\rho^\veps,\uu{u}^\veps)\in L^\infty(0, T;
  L^\infty(\Omega)^{1+d})$ in $L^r(Q_T)^{1+d}$ for $1\leq r<\infty$.
  
\end{enumerate}
We also assume that the sequence of discretisations
$\big(\mcal{T}^{(m)},\delta t^{(m)}\big)_{m\in\mbb{N}}$ satisfies the
following mesh regularity conditions:
\begin{equation}
\label{eq:bound_meshpar}
\frac{\delta t^{(m)}}{\min_{K\in\mcal{M}^{(m)}}\abs{K}}\leq\theta,\
\max_{K \in \mcal{M}^{(m)}}
\frac{\diam(K)^2}{\abs{K}}\leq\theta,\;\forall m\in\mbb{N}, 
\end{equation}
where $\theta>0$ is independent of the discretisations. 
Then, $(\rho^\veps,\uu{u}^\veps)$ satisfies the weak formulation
\eqref{eq:cont_weak_soln_mas}-\eqref{eq:cont_weak_soln_mom}.
\end{theorem}

\begin{proof}
  First, we consider the limits of the discrete convection operators in
  \eqref{eq:disc_mas_updt}-\eqref{eq:disc_mom_updt}, borrowing the
  results obtained in \cite{GHL19, GHL22, HLN+23}. Since our treatment
  closely follows the one in the above references, we skip most of
  the detailed calculations and focus only on those terms arising from
  the stabilisation terms. We consider a smooth test function $\psi\in
  C^\infty_c([0,T)\times\bar{\Omega})$ and for each $m\in\mbb{N}$, 
  let $\psi^n_K$ denote the average of $\psi$ on $(t^n,t^{n+1})\times
  K$ for each $n=0,\dots,N^{(m)}-1,\; K\in\M^{(m)}$. We multiply the
  mass update \eqref{eq:disc_mas_updt} by $\dt\abs{K}\psi^n_K$ and
  take sum over all $n=0,\dots,N^{(m)}-1,\; K\in\M^{(m)}$ to obtain 
  \begin{equation}
    \label{eq:dis_mass_lw}
    \sum_{n=0}^{N^{(m)}-1}\dt\sum_{K\in\M^{(m)}}\abs{K}\mcal{C}_{\mcal{T}^{(m)}}
    (\rho^{(m)},\uu{u}^{(m)})^n_K
    \psi^n_K + R^{(m)}_1=0,
  \end{equation}
  where
  $\mcal{C}_{\mcal{T}^{(m)}}(\rho^{(m)},\uu{u}^{(m)}):Q_T\to\mbb{R}$
  denotes the discrete convection operator
  \begin{equation}
    \label{eq:mass_disconv1}
    \mcal{C}_{\mcal{T}^{(m)}}(\rho^{(m)},\uu{u}^{(m)})(t,\uu{x})
    =\sum_{n=0}^{N^{(m)}-1}\sum_{K\in\M^{(m)}}\mcal{C}_{\mcal{T}^{(m)}}(\rho^{(m)},\uu{u}^{(m)})^n_K
    \mcal{X}_{(t^n,t^{n+1})}(t)\mcal{X}_K(\uu{x}).
  \end{equation}
  In the above, we have denoted 
  \begin{equation}
    \label{eq:mass_disconv2}
    \mcal{C}_{\mcal{T}^{(m)}}(\rho^{(m)},\uu{u}^{(m)})^n_K =
    \frac{1}{\dt}(\rho^{n+1}_K-\rho^n_K)
    +\frac{1}{\abs{K}}\sum_{\substack{\sigma\in\E(K)\\\sigma=K\vert
        L}}
    \abs{\sigma}\rho^{n+1}_\sigma u^n_{\s,K},
    \;n=0,\dots,N^{(m)}-1,\;K\in\M^{(m)}.
  \end{equation}
  In \eqref{eq:dis_mass_lw}, $R^{(m)}_1$ is the remainder term
  due to the velocity stabilisation given by 
 \begin{equation}
   R^{(m)}_1=\sum_{n=0}^{N^{(m)}-1}\dt^n\sum_{K\in\M^{(m)}}
   \diam(K)\sum_{\s\in\E(K)}\abs{\s}\abs{\delta
     u^{n+1}_\sigma}.
 \end{equation}
 Proceeding as in \cite[Lemma 4.1]{GHL22} with the discrete convection
 operator \eqref{eq:mass_disconv1}-\eqref{eq:mass_disconv2} shows that
 the first term on the left 
 hand side of \eqref{eq:dis_mass_lw} yields the weak formulation
 \eqref{eq:cont_weak_soln_mas} upon taking the limit
 $m\to\infty$. Therefore, it remains to show that the remainder term
 $R^{(m)}_1\to 0$ as $m\to\infty$. To this end, we further decompose
 $R^{(m)}_1$ as 
 \begin{equation}
   R^{(m)}_1\leq R_{1,1}^{(m)}+R_{1,2}^{(m)},
 \end{equation}
 where
 \begin{align}
   \label{eq:pr_grd_trunc}
   R_{1,1}^{(m)}&=C\sum_{n=0}^{N^{(m)}-1}\dt\sum_{K\in\M^{(m)}}
                  \diam(K)\sum_{\s\in\E(K)}\abs{\s}\eta\dt
                  \frac{\abs{\s}}{\abs{\Ds}}\abs{\rho^n_{L}-\rho^n_{K}},\\ 
   \label{eq:source_trunc}
   R_{1,2}^{(m)}&=C\sum_{n=0}^{N^{(m)}-1}\dt\sum_{K\in\M^{(m)}}
                  \diam(K)\sum_{\s\in\E(K)}|\s|\eta\dt
                  \frac{\abs{\s}}{\abs{\Ds}}\abs{\phi_{L}-\phi_{K}},
 \end{align}
 with a constant $C>0$ depending only on $\psi$ and $\bar{C}$.
 Applying the results on the convergence of discrete space translates,
 cf.\ \cite[Lemma 4.1]{GHL19}, we can easily see that the right hand
 sides of \eqref{eq:pr_grd_trunc}-\eqref{eq:source_trunc} tends to $0$
 as $m\rightarrow \infty$ under the given assumptions. Consequently,
 the weak consistency of the mass update \eqref{eq:disc_mas_updt}
 follows. 

 In order to prove the weak consistency of the momentum balance, we
 similarly start with a compactly supported smooth vector valued test
 function $\uu{\psi}=(\psi_i)_{i=1}^d\in
 C^\infty_c([0,T)\times\bar{\Omega})^d$. For each $m\in\mbb{N}$ and
 $i=1,\dots,d$, denote by $(\psi^{(m)}_i)^n_{\sigma}$, the average of
 $\psi_i$ on $(t^n,t^{n+1})\times D_{\sigma}$ for
 $n=0,\dots,N^{(m)}-1,\; \sigma\in(\E^{(m)})^{(i)}_{\mathrm{int}}$. We
 multiply the momentum update \eqref{eq:disc_mom_updt} by
 $(\psi^{(m)}_i)_\s^n$ for $i=1,\dots,d$ and take sum over all
 $n=0,\dots,N^{(m)}-1,\;
 \sigma\in(\E^{(m)})^{(i)}_{\mathrm{int}}$. The weak consistency of
 the corresponding discrete convection operator and the pressure 
 gradient can be obtained using analogous results as in
 \cite[Theorem 4.1]{HLN+23}. Whereas, the remainder terms due to the
 stabilisation terms in the momentum flux vanishes as $m\to\infty$
 upon recalling similar arguments as in the case of the mass update.
 
 Finally, the summation over all $\s\in\E^{(i)}_\intr$ and all
 $n\in\{0,1,\dots,N^{(m)-1}\}$, the source term of
 \eqref{eq:disc_mom_updt} yields the expression 
 \begin{equation}
   \label{eq:src_cvrg}
   T^{(m)}_{i,\mathrm{source}}
   =\sum_{n=0}^{N^{(m)}-1}\dt^n\sum_{\s\in\E^{(i)}_\intr}\abs{\Ds}
   \rho_\s^n\psi_\s^n(\partial^{(i)}_\E
   \phi)_\s.
 \end{equation}
 Since, $\phi\in L_{\M^{(m)}}(\Omega)$ is obtained upon taking
 cell-averages of a smooth function, an application of \cite[Lemma
 3.1]{GHL19} and the assumptions on the sequence of discrete solutions yield, 
 \begin{equation}
   T^{(m)}_{i,\mathrm{source}} \to \iint\limits_{Q_T}\rho
   \D_{x_i}{\phi}\psi_i\dd\uu{x}\dd t,\;\mathrm{as}\;m\to\infty,
 \end{equation}
 and accordingly we have the weak consistency of the momentum update 
 \eqref{eq:disc_mom_updt}. 
\end{proof}

\section{Consistency with the Anelastic Limit}
\label{sec:dis_anelastic_limit}

With the aid of the estimates obtained in
Section\,\ref{sec:enstab_wb}, we now proceed to obtain the asymptotic
limit of the semi-implicit scheme. As a first step, we derive the
following uniform estimates with respect to $\veps$ as an immediate
consequence of total energy bound obtained
Theorem\,\ref{thm:dis_totbal}. Consequently, we show that, for a given
discretisation of the domain, the velocity stabilised semi-implicit
scheme \eqref{eq:disc_mas_updt}-\eqref{eq:disc_mom_updt} converges to a
velocity stabilised semi-implicit scheme for the anelastic Euler
equations \eqref{eq:anelastic_mom}-\eqref{eq:anelastic_vel} as
$\veps\to 0$. Note that, in the subsequent analysis, we shall take
advantage of the corresponding discrete function spaces being
finite-dimensional for a fixed mesh. This feature, along with the
uniform bounds, finally enables us to obtain the limit of the stiff
terms as a solution of a discrete elliptic equation.
\begin{proposition}[Global discrete energy inequality]
  \label{prop:dis_energy_uni_bd}
  Let $(\rho^{n+1}, \uu{u}^{n+1})\in
  L_\M(\Omega)\times\uu{H}_{\E,0}(\Omega)$ be a solution of the semi
  implicit scheme \eqref{eq:disc_mas_updt}-\eqref{eq:disc_mom_updt}
  with respect to a well-prepared initial data
  $(\rho^\veps_0,\uu{u}^\veps_0)\in L^\infty(\Omega)^{1+d}$. Then
  $(\rho^{n+1}, \uu{u}^{n+1})$ satisfies the following estimate for
  sufficiently small $\veps>0$.    
  \begin{equation}
    \begin{aligned}
      \label{eq:tot_en_unibd}
      \frac{1}{\veps^2}\sum_{K\in\M}|K|\Pi_\gamma(\rho_{K}^{n+1}|\tilde{\rho}_K)
      &+\sum_{i=1}^d\sum_{\s\in\E^{(i)}_\mathrm{int}}
      \abs{\Ds}\frac{1}{2}\rho_{\Ds}^{n+1}(u_{\s}^{n+1})^2 \\
      &+\frac{\dt^2}{\veps^4}\sum_{k=0}^n
      \sum_{i=1}^d\sum_{\s\in\E_\mathrm{int}^{(i)}}|\Ds|
      \left(\eta-\frac{1}{\rho^{k+1}_{\Ds}}\right)
      \left((\D_\E^{(i)}p^{k+1})_\s+\rho^{k+1}_\sigma
        (\D_\E^{(i)}\phi)_\sigma\right)^2 \leq C,
    \end{aligned}
  \end{equation}
  where $C>0$ is a constant independent of $\veps$.
\end{proposition}
\begin{proof}
  Upon multiplying \eqref{eq:dis_totenbal_press} by $\dt$ and taking
  sum over time-steps from $0$ to $n-1$ we obtain that, any solution
  $(\rho^{n+1},\uu{u}^{n+1})\in L_\M(\Omega)\times\uu{H}_{\E,0}(\Omega)$
  of \eqref{eq:disc_mas_updt}-\eqref{eq:disc_mom_updt} satisfies
  \begin{equation}
    \begin{aligned}
      \label{eq:tot_en_unibd_ini}
      &\frac{1}{\veps^2}\sum_{K\in\M}|K|\Pi_\gamma(\rho_{K}^{n+1}|\tilde{\rho}_K)
      +\sum_{i=1}^d\sum_{\s\in\E^{(i)}_\mathrm{int}}
      \abs{\Ds}\frac{1}{2}\rho_{\Ds}^{n+1}(u_{\s}^{n+1})^2 \\
      &+\frac{\dt^2}{\veps^4}\sum_{k=0}^n\sum_{i=1}^d
      \sum_{\s\in\E_\mathrm{int}^{(i)}}|\Ds|
      \left(\eta-\frac{1}{\rho^{k+1}_{\Ds}}\right)\left((\D_\E^{(i)}p^{k+1})_\s
        +\rho^{k+1}_\sigma (\D_\E^{(i)}\phi)_\sigma\right)^2 \\
      &\leq 
      \frac{1}{\veps^2}\sum_{K\in\M}|K|\Pi_\gamma(\rho_{K}^0|\tilde{\rho}_K)
      +\sum_{\s\in\E_\mathrm{int}}\abs{\Ds}\frac{1}{2}\rho_{\Ds}^0(u_{\s}^0)^2.
    \end{aligned}
  \end{equation}
  The uniform bound of the right hand side of
  \eqref{eq:tot_en_unibd_ini} with respect to $\veps$ follows from
  Lemma\,\ref{lem:wp_est}. 
\end{proof}
Carrying out a treatment analogous to the one performed for the
continuous case in Section\,\ref{sec:ewg_cont}, we obtain the following
estimates for the scheme
\eqref{eq:disc_mas_updt}-\eqref{eq:disc_mom_updt} as $\veps\to 0$. 
\begin{lemma}
  \label{lem:dis_var_conv_bd}
  Let $\mcal{T}=(\mcal{M},\mcal{E})$ be a fixed primal-dual mesh pair
  that gives a MAC grid discretisation of $\bar{\Omega}$ and let $\delta
  t>0$ be such that $\{0=t^0<t^1<\dots<t^N=T\}$ is a discretisation of
  the time domain $[0,T]$ with  $\dt =  t^{n}-t^{n-1}$ for $1\leq
  n\leq N$. For each $\veps>0$, let us denote by
  $(\rho^\veps,\uu{u}^\veps)\in
  L^\infty(0,T;L_{\mcal{M}}(\Omega)\times\uu{H}_{\mcal{E},0}(\Omega))$
  the approximate solutions to the semi-implicit scheme
  \eqref{eq:disc_mas_updt}-\eqref{eq:disc_mom_updt} with respect to
  a well-prepared initial data $(\rho^\veps_0, \uu{u}^\veps_0)\in
  L^\infty(\Omega)^{1+d}$. Then we have the following.
  \begin{enumerate}[label=(\roman*)]
  \item $\rho^\veps\rightarrow \tilde{\rho}$ as $\veps\rightarrow 0$ in
    $L^\infty(0,T;L^r(\Omega))$, for any $r\in(1,\min\{2,\gamma\}]$,
    where $\tilde{\rho}=\sum_{K\in\M}\tilde{\rho}_K\mcal{X}_K\in
    L_\M(\Omega)$ denotes the discrete approximation of the hydrostatic
    density.
  \item The sequence of approximate solutions for the velocity
    component $\{\uu{u}^\veps\}_{\veps>0}\subset\uu{H}_{\E,0}(\Omega)$
    is uniformly bounded with respect to $\veps$, i.e.\ for all $\veps>0$
    sufficiently small, we have
    \begin{equation}
      \label{eq:dis_vel_uni_ub}
      \norm{\uu{u}^\veps}_{L^\infty(0,T;L^2(\Omega)^d)}\leq C,
    \end{equation}
    where $C>0$ is a constant independent of $\veps$.
  \end{enumerate}
\end{lemma}
\begin{proof}
  Let $\underline{\rho}>0,\;\bar\rho>0$ be such that
  \begin{equation}
    \label{eq:tilderho_up_lo}
    \underline{\rho}\coloneqq\inf_{x\in\Omega}\tilde{\rho}(x),\;\bar\rho
    \coloneqq\sup_{x\in\Omega}\tilde{\rho}(x).
  \end{equation}
  From \eqref{eq:tot_en_unibd} and the fact that $\Pi_\gamma(a|b)>0$ for
  any $a,b\in\mbb{R}$, we obtain that a discrete solution $\rho^\veps\in
  L_\M(\Omega)$ of \eqref{eq:disc_mas_updt}-\eqref{eq:disc_mom_updt}
  satisfies
  \begin{equation}
    \label{eq:dis_pigam_est}
    \abs{K}\Pi_\gamma\big((\rho^\veps)^{n+1}_K|\tilde{\rho}_K\big)
    \leq C\veps^2,\;\forall K\in\M,\; n=0,\dots, N-1,
  \end{equation}
  where $C>0$ is a constant independent of $\veps$. Analogously as in
  the the proof of \cite[Lemma 6.2]{AGK23}, for any $\gamma>2$, we get
  \begin{equation}
    \label{eq:dis_rho_est_gamma_big}
    \abs{K}\absq{(\rho^\veps)^{n+1}_K-\tilde{\rho}_K}\leq C\veps^2,
    \;\forall K\in\M, \; n = 0,\dots,N-1,
  \end{equation}
  and for any $\gamma\in(1,2)$, and each $n\in\{0,1,\dots,N-1\}$, we
  have 
  \begin{align}
    \abs{K}\absq{(\rho^\veps)^{n+1}_K-\tilde{\rho}_K}
    &\leq C\veps^2,\;\forall K\in\M^\veps_\mathrm{ess},
      \label{eq:dis_rho_est_gamma_ess} \\
    \abs{K}\abs{(\rho^\veps)^{n+1}_K-\tilde{\rho}_K}^\gamma
    &\leq C\veps^2,\;\forall K\in\M^\veps_\mathrm{res}.
      \label{eq:dis_rho_est_gamma_res}
  \end{align}
  Here, the essential and residual parts are defined by
  $\M^\veps_\mathrm{ess}:=\{K\in\M\colon (\rho^\veps)^{n+1}_K<
  2\bar{\rho}\}$ and $\M^\veps_\mathrm{res}:=\{K\in\M\colon
  (\rho^\veps)^{n+1}_K\geq 2\bar{\rho}\}$. In all the estimates 
  \eqref{eq:dis_rho_est_gamma_big}-\eqref{eq:dis_rho_est_gamma_res} 
  above, $C>0$ denotes a constant independent of $\veps$ which in
  turn yields the convergence of $\rho^\veps$ as stated in (i). From
  the estimates
  \eqref{eq:dis_rho_est_gamma_big}-\eqref{eq:dis_rho_est_gamma_res}
  we further obtain that for a fixed space-time discretisation,
  $(\rho^\veps)^n_K\to \tilde{\rho}_K$ as $\veps\to 0$ for each
  $K\in\M$ and $1\leq n\leq N$. Subsequently, for $\veps>0$,
  sufficiently small, we have  
  \begin{equation}
    \label{eq:dis_rho_bd}
    0<\underline{\rho}\leq (\rho^\veps)^n_K\leq \bar{\rho},\;
    \forall K\in\M,\; 1\leq n\leq N.
  \end{equation}
  The estimate \eqref{eq:dis_vel_uni_ub} in (ii) readily follows from
  the kinetic energy component of the total energy estimate
  \eqref{eq:tot_en_unibd} and the uniform lower bound of
  $\{\rho^\veps\}_{\veps>0}$ obtained in \eqref{eq:dis_rho_bd} for
  sufficiently small values of $\veps$.
\end{proof}
In order to study the convergence of the stiff terms as the
Mach/Froude numbers vanish, we first obtain the following 
discrete $L^2$-control that comes as a consequence of the uniform
bound obtained in Proposition\,\ref{prop:dis_energy_uni_bd}.  
\begin{lemma}
  \label{lem:stiff_bd}
  Let  $\mcal{T}=(\M,\E)$ and $\dt>0$ be as in
  Lemma\,\ref{lem:dis_var_conv_bd}. Let $\veps>0$ be fixed and let
  $(\rho^n,\uu{u}^n)_{1\leq n\leq N}\subset
  L_\M(\Omega)\times\uu{H}_{\E,0}(\Omega)$ be a solution of the scheme
  \eqref{eq:disc_mas_updt}-\eqref{eq:disc_mom_updt} with respect to a
  well-prepared initial data $(\rho^\veps_0, \uu{u}^\veps_0)\in
  L^\infty(\Omega)^{1+d}$. Then there exists a constant $C>0$,
  independent of $\veps$, such that for each $1\leq n\leq N$, 
  \begin{equation}
    \label{eq:stiff_term_ub}
    \frac{1}{\veps^4}\sum_{i=1}^d\sum_{\sigma\in\E^{(i)}_\mathrm{int}}\abs{D_\sigma}
    \left((\D_\E^{(i)}p^n)_\sigma+\rho^n_\sigma(\D_{\E}^{(i)}\phi)_\sigma\right)^2<
    C,
  \end{equation}
  which further yields,
  \begin{equation}
    \label{eq:grad_L2_bd}
    \frac{1}{\veps^4}\sum_{i=1}^d\sum_{\sigma\in\E^{(i)}_\mathrm{int}}\abs{D_\sigma}
    \left((\D_\E^{(i)}(h_\gamma^\prime(\rho^n)+\phi))_\sigma\right)^2< C.
  \end{equation}
\end{lemma}
\begin{proof}
  From the energy estimate \eqref{eq:tot_en_unibd} we note that for
  $\veps>0$, small enough, and for each $1\leq n\leq N$, the following
  holds: 
  \begin{equation}
    \label{eq:dis_pressure_source_est}
    \frac{1}{\veps^4}\sum_{i=1}^d\sum_{\s\in\E_\mathrm{int}^{(i)}}
    \abs{\Ds}\left(\eta-\frac{1}{\rho^{n}_{\Ds}}\right)
    \left((\D_\E^{(i)}p^{n})_\s
      +\rho^{n}_\sigma (\D_\E^{(i)}\phi)_\sigma\right)^2 \leq C,
  \end{equation}
  where the constant $C>0$ is independent of $\veps$, but depends on
  $\underline{\rho}, \bar{\rho}$, the domain $Q_T$, and the fixed
  discretisation $(\dt, \mcal{T})$. The choice of the constant $\eta$ in
  accordance with the inequality \eqref{eq:rho_3/2}, imposed by the
  time-step restriction, yields \eqref{eq:stiff_term_ub}. An
  application of the identity \eqref{eq:dis_pr_grd} related to the
  choice $\rho_\s$ and the uniform bounds \eqref{eq:dis_rho_bd} on
  $\rho^\veps$ yield \eqref{eq:grad_L2_bd} from
  \eqref{eq:stiff_term_ub}.    
\end{proof}
We have the following corollary on the limit of the stiff terms as an
immediate consequence of the uniform estimate
\eqref{eq:stiff_term_ub}. 
\begin{corollary}
  \label{cor:stiff_wk_lim}
  For each $n=1,\dots,N$, there exists
  $\uu{\theta}^{n}\in\uu{H}_{\E,0}(\Omega)$ such that
  $\{\frac{1}{\veps^2}(\grd_{\E}p^n+\rho^n\grd_{\E}\phi)\}_{\veps>0}$
  weakly converges to $\uu{\theta}^{n}$ in $L^2(\Omega)^d$.
\end{corollary}
\begin{proof}
  From Lemma\,\ref{lem:stiff_bd} we note that
  $\{\frac{1}{\veps^2}(\grd_{\E}p^n+\rho^n\grd_{\E}\phi)\}_{\veps>0}\subset 
  \uu{H}_{\E,0}(\Omega)$ is uniformly bounded in $L^2(\Omega)^d$ for
  each $n=1,\dots,N$ and hence admits a weak limit $\uu{\theta}^n\in
  L^2(\Omega)^d$ as $\veps\to0$. Since $\uu{H}_{\E,0}(\Omega)$
  is weakly closed in $L^2(\Omega)^d$ as it is a finite dimensional
  subspace of $L^2(\Omega)^d$, we conclude that
  $\uu{\theta}^n\in \uu{H}_{\E,0}(\Omega)$ for each $n=1,\dots,N$.
\end{proof}
The following lemma provides a description of the weak limit
$\uu{\theta}^n$ of the stiff terms; cf.\
Corollary\,\ref{cor:cont_stiff_lim} for the analogous result in the
continuous case.
\begin{lemma}
  \label{lem:stiff_lim}
  For each $n=1,\dots,N$, there exists $\pi^n\in L_{\M}(\Omega)$ such
  that
  \begin{equation}
    \label{eq:tilde_rho_grad}
    \theta^n_{\sigma}=\tilde\rho_{\sigma}(\D_{\E}^{(i)}\pi^n)_{\sigma},
    \;\sigma\in\E^{(i)}_{\mathrm{int}},\;
    i=1,2,\dots,d.
  \end{equation}
\end{lemma}
\begin{proof}
  First, we claim that the following identity holds for any
  $\uu{\psi}\in\uu{H}_{\E,0}(\Omega)$ with
  $\dive_{\M}(\tilde\rho\uu{\psi})=0$: 
  \begin{equation}
    \label{eq:wt_div_free_lim}
    \sum_{i=1}^d\sum_{\sigma\in\E^{(i)}_{\mathrm{int}}}
    \abs{D_\sigma}\theta^n_\sigma\psi_\sigma
    = 0.
  \end{equation}
  Note that, for any $\uu{\psi}\in\uu{H}_{\E,0}(\Omega)$ satisfying
  $\dive_{\M}(\tilde\rho\uu{\psi})=0$, we have 
  \begin{equation}
    \label{eq:wt_div_free_lim1}
    \begin{aligned}
      &\frac{1}{\veps^2}\sum_{i=1}^d\sum_{\sigma\in\E^{(i)}_{\mathrm{int}}}
      \abs{D_\sigma}\left((\D^{(i)}_{\E}p^n)_{\sigma}
        +\rho^n_\s (\D^{(i)}_{\E}\phi)_{\sigma}\right)\psi_\sigma\\
      &\quad =\frac{1}{\veps^2}\sum_{i=1}^d\sum_{\sigma\in\E^{(i)}_{\mathrm{int}}}
      \abs{D_\sigma}\rho^n_\sigma
      \left((\D^{(i)}_{\E}(h_\gamma^\prime(\rho^n+\phi)))_{\sigma}\right)\psi_\sigma\\
      &\quad=\frac{1}{\veps^2}\sum_{i=1}^d\sum_{\sigma\in\E^{(i)}_{\mathrm{int}}}
      \abs{D_\sigma}(\rho^n_\sigma-\tilde\rho_\sigma)
      \left((\D^{(i)}_{\E}(h_\gamma^\prime(\rho^n+\phi)))_{\sigma}\right)\psi_\sigma\\
      &\qquad+ \frac{1}{\veps^2}\sum_{i=1}^d\sum_{\sigma\in\E^{(i)}_{\mathrm{int}}}
      \abs{D_\sigma}\tilde\rho_\sigma
      \left((\D^{(i)}_{\E}(h_\gamma^\prime(\rho^n+\phi)))_{\sigma}\right)\psi_\sigma.
    \end{aligned}
  \end{equation}
  Applying the Cauchy-Schwarz inequality, the first term on the
  right hand side of \eqref{eq:wt_div_free_lim1} yields
  \begin{equation}
    \label{eq:wt_div_free_lim2}
    \begin{aligned}
      &\frac{1}{\veps^2}\bigg|\sum_{i=1}^d\sum_{\sigma\in\E^{(i)}_{\mathrm{int}}}
      \abs{D_\sigma}(\rho^n_\sigma-\tilde\rho_\sigma)
      \left((\D^{(i)}_{\E}(h_\gamma^\prime(\rho^n+\phi)))_{\sigma}\right)\psi_\sigma\bigg|\\  
      &\leq C_{\uu{\psi}} \Big(\sum_{i=1}^d\sum_{\sigma\in\E^{(i)}_{\mathrm{int}}}
      \abs{D_\sigma}\absq{\rho^n_\sigma-\tilde\rho_\sigma}\Big)^{\frac{1}{2}}
      \Big(\frac{1}{\veps^4}\sum_{i=1}^d\sum_{\sigma\in\E^{(i)}_{\mathrm{int}}}
      \abs{D_\sigma}\absq{(\D^{(i)}_{\E}
        (h_\gamma^\prime(\rho^n+\phi)))_{\sigma}}\Big)^{\frac{1}{2}},
    \end{aligned}
  \end{equation}
  where $C_{\uu{\psi}}>0$ is a constant that depends only on
  $\uu{\psi}$. The uniform bound \eqref{eq:grad_L2_bd} and the 
  convergence of $\rho^n\in L_{\M}(\Omega)$, detailed in
  Lemma\,\ref{lem:dis_var_conv_bd}, along with a finite dimensional
  argument shows that the right hand side of \eqref{eq:wt_div_free_lim2}
  goes to $0$ as $\veps\to 0$. The second term on the right hand side
  of \eqref{eq:wt_div_free_lim1} vanishes after using the div-grad
  duality \eqref{eq:weighted_disc_dual} and noting that
  $\dive_{\M}(\tilde\rho\uu{\psi})=0$. Therefore, the claim
  \eqref{eq:wt_div_free_lim} is proved upon passing to the limit
  $\veps\to 0$ in \eqref{eq:wt_div_free_lim2}, using
  Corollary\,\ref{cor:stiff_wk_lim}. 
  
  Next, we introduce a function $\pi^n\in L_{\M}(\Omega)$, for each
  $n=1,\dots,N$, via the solution of the discrete Neumann elliptic
  boundary value problem  
  \begin{equation}
    \label{eq:el_bdary_val}
    \begin{aligned}
      \Delta_{\M}\pi^n&=\dive_{\M}\big(\frac{1}{\tilde\rho}\uu{\theta}^n\big), \\
      (\grd_{\E}\pi^n)_{\sigma}&=0,\;\forall\sigma\in\E^{(i)}_{\mathrm{ext}},\;i=1,\dots,d. 
    \end{aligned}
  \end{equation}
  The above problem gives rise to the following $\#\M\times \#\M$
  system of linear equations for the unknowns $(\pi^n_K)_{K\in\M}$: 
  \begin{equation}
    \label{eq:lin_prob_pi}
    \sum_{\substack{\sigma\in\E(K)\\\sigma=K\vert
        L}}\frac{\absq{\sigma}}{\abs{D_\sigma}}(\pi^n_L-\pi^n_K) =
    \abs{K}\Big(\dive_{\M}\big(\frac{1}{\tilde\rho}\uu{\theta}^n\big)\Big)_K,\;\forall
    K\in\M. 
  \end{equation}
  Note that the right hand side of \eqref{eq:lin_prob_pi} satisfies the
  compatibility condition
  \begin{equation}
    \label{eq:compatibility}
    \sum_{K\in\M}\abs{K}(\dive_{\M}(\frac{1}{\tilde\rho}\uu{\theta}^n))_K=
    \sum_{K\in\M}\abs{K}\sum_{\sigma\in\E(K)}\abs{\sigma}
    \frac{1}{\tilde\rho_\sigma}\uu{\theta}^n_{\sigma,K}=0.
  \end{equation}
  Hence, the system of linear equations \eqref{eq:lin_prob_pi} admits
  a solution $(\pi^n_K)_{K\in\M}\equiv\pi^n\in L_{\M}(\Omega)$, unique
  upto an additive constant; see \cite[Lemma 3.6]{EGH00}. Finally, we
  define $\uu{\xi}=(\xi^{(i)})_{i=1}^d\in\uu{H}_{\E,0}(\Omega)$, where
  $\xi^{(i)}=\sum_{\sigma\in\E^{(i)}_\mathrm{int}}\xi_{\sigma}\mcal{X}_{D_\sigma},\;
  i=1,\dots,d$, is given by 
  \begin{equation}
    \label{eq:div_free_tf}
    \xi_\sigma\coloneqq (\D_{\E}^{(i)}\pi^n)_\sigma -
    \frac{1}{\tilde\rho_\sigma}\theta^n_\sigma.
  \end{equation}
  Since $\pi^n$ solves \eqref{eq:el_bdary_val}, we notice that
  $\dive_{\M}\uu{\xi}=0$ and the identity \eqref{eq:wt_div_free_lim}
  further implies that 
  \begin{equation}
    \label{eq:stiff_lim1}
    \sum_{i=1}^d\sum_{\sigma\in\E^{(i)}_{\mathrm{int}}}\abs{D_\sigma}
    \theta^n_\sigma\frac{\xi_\sigma}{\tilde\rho_\sigma} = 0.
  \end{equation}
  Now, for each $\sigma\in\E^{(i)}_{\mathrm{int}}$, $i=1,\dots,d$,
  replacing $\theta^n_\sigma$ using \eqref{eq:div_free_tf} in
  \eqref{eq:stiff_lim1}, we have
  \begin{equation}
    \label{eq:stiff_lim2}
    \sum_{i=1}^d\sum_{\sigma\in\E^{(i)}_{\mathrm{int}}}\abs{D_\sigma}
    ((\D_{\E}^{(i)}\pi^n)_\sigma-\xi_\sigma)\xi_\sigma = 0.
  \end{equation}
  An application of the div-grad duality \eqref{eq:disc_dual} and
  the fact that $\dive_{\M}\uu{\xi}=0$, the above expression yields
  \begin{equation}
    \label{eq:stiff_lim3}
    \sum_{i=1}^d\sum_{\sigma\in\E^{(i)}_{\mathrm{int}}}
    \abs{D_\sigma}\absq{\xi_\sigma}=0,
  \end{equation}
  and hence we have \eqref{eq:tilde_rho_grad}.
\end{proof}
Based on these convergences and the uniform boundedness obtained in
Lemmas \ref{lem:dis_var_conv_bd} and \ref{lem:stiff_lim},  we
deduce the following theorem, where we show that the velocity
stabilised semi-implicit scheme 
\eqref{eq:disc_mas_updt}-\eqref{eq:disc_mom_updt} converges to a
velocity stabilised semi-implicit scheme for the anelastic Euler
system as $\veps\to 0$.

\begin{theorem}
  \label{thm:dis_anelastic_limit}
  Let $(\veps^{(k)})_{k\in \mathbb{N}}$ be a sequence of positive
  numbers such that $\veps^{(k)}\to 0$ as $k\to\infty$.
  Let $\{(\rho^{(k)},\uu{u}^{(k)})\}_{k\in\mathbb{N}}$ be the corresponding
  sequence of numerical solutions obtained from the scheme
  \eqref{eq:disc_mas_updt}-\eqref{eq:disc_mom_updt} with respect to a
  sequence of well-prepared initial data
  $\{(\rho^{(k)}_{0},\uu{u}^{(k)}_{0})\}_{k\in\mbb{N}}\subset
  L^\infty(\Omega)^{1+d}$. Then $\{\rho^{(k)}\}_{k\in\mathbb{N}}$
  converges to $\tilde{\rho}\in L_\M(\Omega)$ in
  $L^{\infty}(0,T;L^r(\Omega))$ for any $r\in(1, \min\{2,\gamma\}]$ and
  so in any discrete norm on $L_\M(\Omega)$. Furthermore, there exists
  $(\pi, \uu{U})\in L^\infty(0,T;
  L_\M(\Omega)\times\uu{H}_{\E,0}(\Omega))$ such that
  $\{\uu{u}^{(k)}\}_{k\in\mathbb{N}}$ converges to $\uu{U}\in
  L^\infty(0,T; \uu{H}_{\E,0}(\Omega))$ in any discrete norm as $k\to
  \infty$ and the piecewise constant function $\{(\pi^n,
  \uu{U}^n)\}_{1\leq n\leq N}\subset L_\M(\Omega)\times 
  \uu{H}_{\E,0}(\Omega)$ is defined as follows.
  
  Given $(\pi^n, \uu{U}^n)\in L_\M(\Omega)\times
  \uu{H}_{\E,0}(\Omega)$ at time $t^n$, $(\pi^{n+1}, \uu{U}^{n+1})\in
  L_\M(\Omega)\times \uu{H}_{\E,0}(\Omega)$ solves the semi-implicit
  scheme
  \begin{gather}
    (\dive_\M(\tilde{\rho}(\uu{U}^n-\delta\uu{U}^{n+1}))_{K}=0,
    \; \forall K \in \mcal{M}, \label{eq:dis_anelastic_mas}\\
    \frac{1}{\dt}\tilde{\rho}_{\Ds}\big(U_\s^{n+1}-U_\s^{n}\big)
    +\frac{1}{\abs{\Ds}}\sum_{\epsilon\in\tilde{\E}(\Ds)}
    F_{\epsilon,\sigma}(\tilde{\rho},\uu{U}^{n}
    -\delta\uu{U}^{n+1})U_{\epsilon,\mathrm{up}}^{n}
    +\tilde{\rho}_{\sigma}(\partial^{(i)}_{\E}\pi^{n+1})_{\s}=0,
    \ 1\leq i\leq d, \ \forall \s\in\E_\intr^{(i)}, \label{eq:dis_anelastic_mom}
  \end{gather}
  with the correction term $(\delta
  U^{n+1})_\s=\eta\dt(\D_\E^{(i)}\pi^{n+1})_\s,\;
  \eta>\frac{1}{\tilde{\rho}_{\Ds}}$ for $\s\in\Eint^{(i)}$,
  $i=1,2,\dots d$.
\end{theorem}

\section{Numerical Case Studies}
\label{sec:num_exp}

In this section, we present the results of several numerical tests
performed with the proposed scheme
\eqref{eq:disc_mas_updt}-\eqref{eq:disc_mom_updt}. The stability
analysis performed in Section~\ref{sec:enstab_wb}, cf.\
Theorem~\ref{thm:dis_totbal}, requires that the time-steps $\dt$ are
to be chosen according to the condition  
\begin{equation}
\label{eq:time-step_res}
     \frac{\dt}{|\Ds|}\sum_{\epsilon\in\tilde\E(\Ds)}\frac{F_{\epsilon,\s}(\rho^{n+1},\uu{u}^n)^-}{\rho_{\Ds}^{n+1}}\leq\frac{1}{2},\;\forall
     \s\in\E^{(i)},\; i=1,\dots,d. 
\end{equation}
Since the above stability condition is implicit, carrying out the
analogous calculations as done in \cite{CDV17, DVB17,DVB20}, we obtain
the following sufficient condition on the time-step which is easier to
implement in practice. Nonetheless, the sufficient condition is still
implicit and we impose it in an explicit fashion.    
\begin{proposition}
Suppose $\dt>0$ be such that for each $\s\in\E^{(i)},\;
i\in\{1,\dots,d\},\; \s = K|L$, the following holds: 
\begin{equation}
\label{eq:time-step_suff}
    \dt\max\Bigg\{\frac{\abs{\D K}}{\abs{K}},\frac{\abs{\D L}}{\abs{L}}\Bigg\}\left(\abs{u^n_\s} + \frac{\sqrt{\eta}}{\veps}\Big|(p^{n+1}_L - p^{n+1}_K) + \rho^{n+1}_\sigma(\phi_L - \phi_K)\Big|^{\half}\right)\leq \min\Big\{1,\; \frac{1}{3}\mu^{n,n+1}_{K,L}\Big\},
\end{equation}
where $\abs{\D K}=\sum_{\s\in\E(K)}\abs{\s}$ and $\mu^{n,n+1}_{K,L} = \displaystyle\frac{\min\{\rho^n_K,\rho^n_L\}}{\max\{\rho^{n+1}_K,\rho^{n+1}_L\}}$. Then $\dt$ satisfies the inequality \eqref{eq:time-step_res}.
\end{proposition}
\begin{proof}
We omit the proof and refer to \cite[Proposition 3.2]{CDV17} for 
calculations done in the case of an explicit scheme; see also
\cite{AGK23} for an analogous result derived for a velocity stabilised 
semi-implicit scheme applied to the barotropic Euler system without
source terms.   
\end{proof}

\begin{remark}
  \label{rem:implementation}
  The numerical implementation of the scheme
  \eqref{eq:disc_mas_updt}-\eqref{eq:disc_mom_updt} is done as
  follows. First, the mass conservation equation \eqref{eq:disc_mas_updt}
  is solved to get the updated density $\rho^{n+1}$. A Newton iteration
  has to be performed due to the presence of nonlinear stabilisation
  terms. In our experiments, we note that the iterations converge in
  2-3 steps. Once $\rho^{n+1}$ is calculated, the momentum update
  \eqref{eq:disc_mom_updt} is evaluated explicitly to get the velocity
  $\uu{u}^{n+1}$.   
\end{remark}

\label{sec:num_tst}
\subsection{Well-Balancing Test}
\label{subsec:wb_test}
In this test case we show that a hydrostatic steady state is exactly
preserved by the scheme uniformly with respect to $\veps$. We consider
the following one-dimensional (1D) steady state initial data:   
\begin{equation*}
    \rho(0, x) = \tilde{\rho}(x) = \Big(1 - \frac{\gm -
                 1}{\gm}\phi(x)\Big)^{\frac{1}{\gm - 1}}, \  u(0, x) = 0,
\end{equation*}
in the domain $[0,1]$ under three gravitational potentials defined by
the choices $\phi(x) = x, \half x^2, \sin (2\pi x)$. The value of the
adiabatic constant is $\gm = 1.4$. The domain is discretised into
$100$ grid points and the $L^1$-errors of the density and momentum
with respect to the above exact solution are calculated until a final
time $T=2$. In Table\,\ref{tab:isen_hydst_1d} we present the errors
for different values of $\veps$ corresponding to both the compressible
and stiff anelastic regimes. We observe that the velocity stabilised
semi-implicit scheme approximates the steady state with high
precision, irrespective of the choices of $\phi$ and irrespective of
the asymptotic regime. 

\begin{table}[ht]
\centering
\begin{adjustbox}{width=0.4\textwidth}
\small
\begin{tabular}{rlrr}
  \hline
   $\phi(x)$ & $\veps$ & $L^1$-error in $\rho$ & $L^1$-error in $\rho u$ \\ 
  \hline
            &$10^{-1}$	&5.7732E-17	&1.0495E-13 \\
   $x$      &$10^{-2}$	 &7.0777E-17 &3.8677E-13 \\
            &$10^{-3}$	&6.8001E-17	&1.0013E-13 \\
   \hline
                            &$10^{-1}$	&3.7192E-17	  &1.0722E-13 \\
   $\half x^2$              &$10^{-2}$	 &3.9968E-17   &1.7715E-13 \\
                            &$10^{-3}$	&3.9413E-17	  &4.7424E-14 \\
    \hline
                            &$10^{-1}$	&2.0983E-16	&2.4883E-13 \\
   $\sin (2\pi x)$          &$10^{-2}$	 &2.2260E-16 &4.3341E-13 \\
                            &$10^{-3}$	&2.1122E-16	&1.6502E-13 \\
    \hline
\end{tabular}
\end{adjustbox}
\caption{$L^1$-errors with respect to a 1D hydrostatic equilibrium.} 
\label{tab:isen_hydst_1d}
\end{table} 

\subsection{Strong Rarefaction}
\label{subsec:str_rrf}
This case study is to showcase the scheme's capability of preserving
the positivity of density under a strong rarefaction. In that order we
solve the following extreme Riemann problem from \cite{TPK20}. The
initial density is given by the hydrostatic equilibrium  
\begin{equation*}
  \rho(0, x, y) = \tilde{\rho}(x,y)=
  \Big(1 - \frac{\gm - 1}{\gm}\phi(x, y)\Big)^{\frac{1}{\gm - 1}},
\end{equation*}
under the gravitational potential $\phi(x, y) = \frac{1}{2}((x-0.5)^2
+ (y-0.5)^2)$. The initial velocity is given by
\begin{equation}
    u(0, x, y) =\begin{cases}
     -5, &\mathrm{if}\;x\leq 0.5,
      \\
      5, &\mathrm{if}\;x>0.5,
    \end{cases} \quad v(0,x, y) = 0.
\end{equation}
The computational domain $[0,1]\times[0,1]$ is discretised into
$100\times100$ mesh points. Extrapolation boundary conditions are
applied on all sides of the domain. The adiabatic constant is chosen
as $\gm = 2$ and the simulations are run till the time $T=0.1$. Due to
the symmetry, we plot only the cross-sections of the density and
velocity profiles along the $x$-direction in
Figure\,\ref{fig:str_rrf}. The results clearly indicate the formation
of vacuum in the middle of the domain and the scheme's excellent
capability to maintain the positivity of density. 

\begin{figure}[htbp]
  \centering
  \includegraphics[height=0.25\textheight]{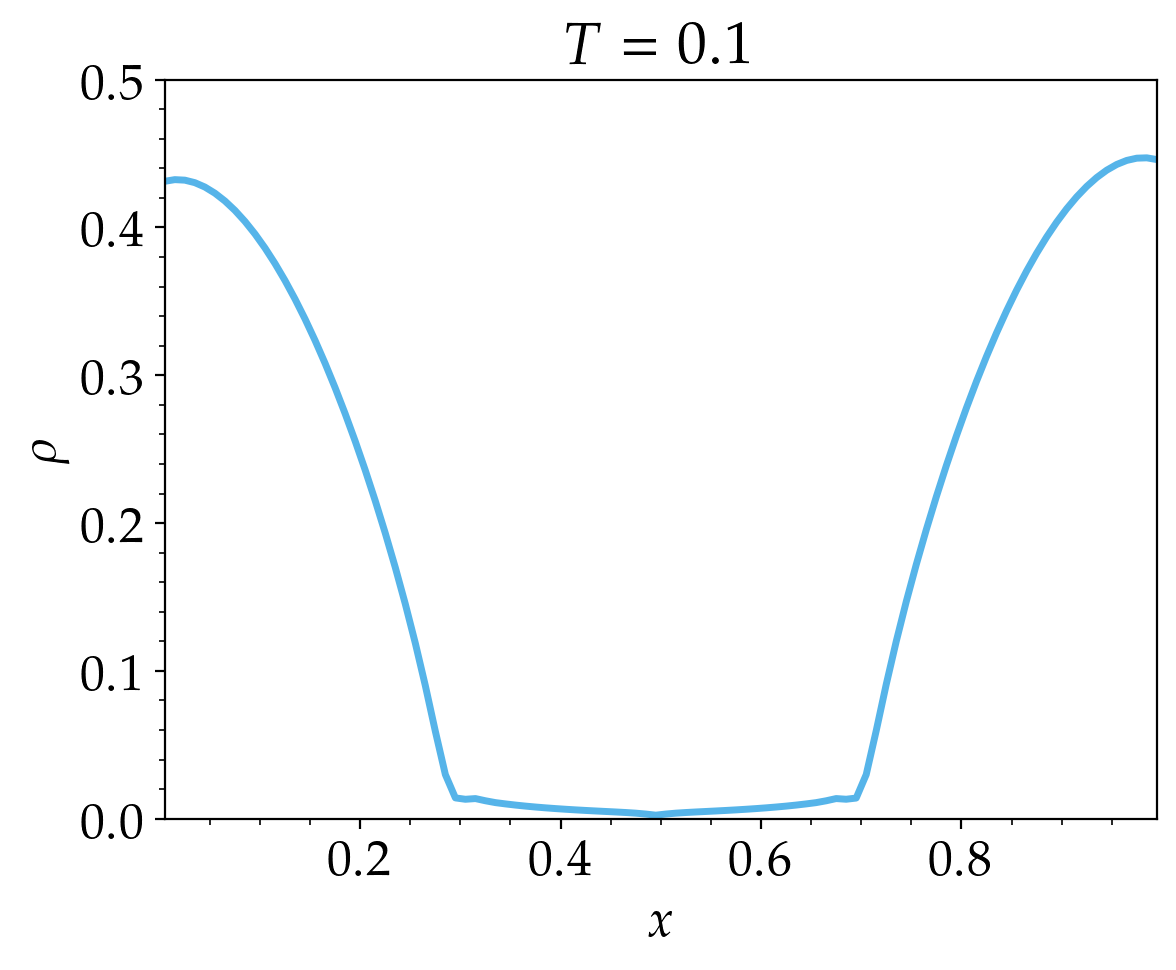}
  \includegraphics[height=0.25\textheight]{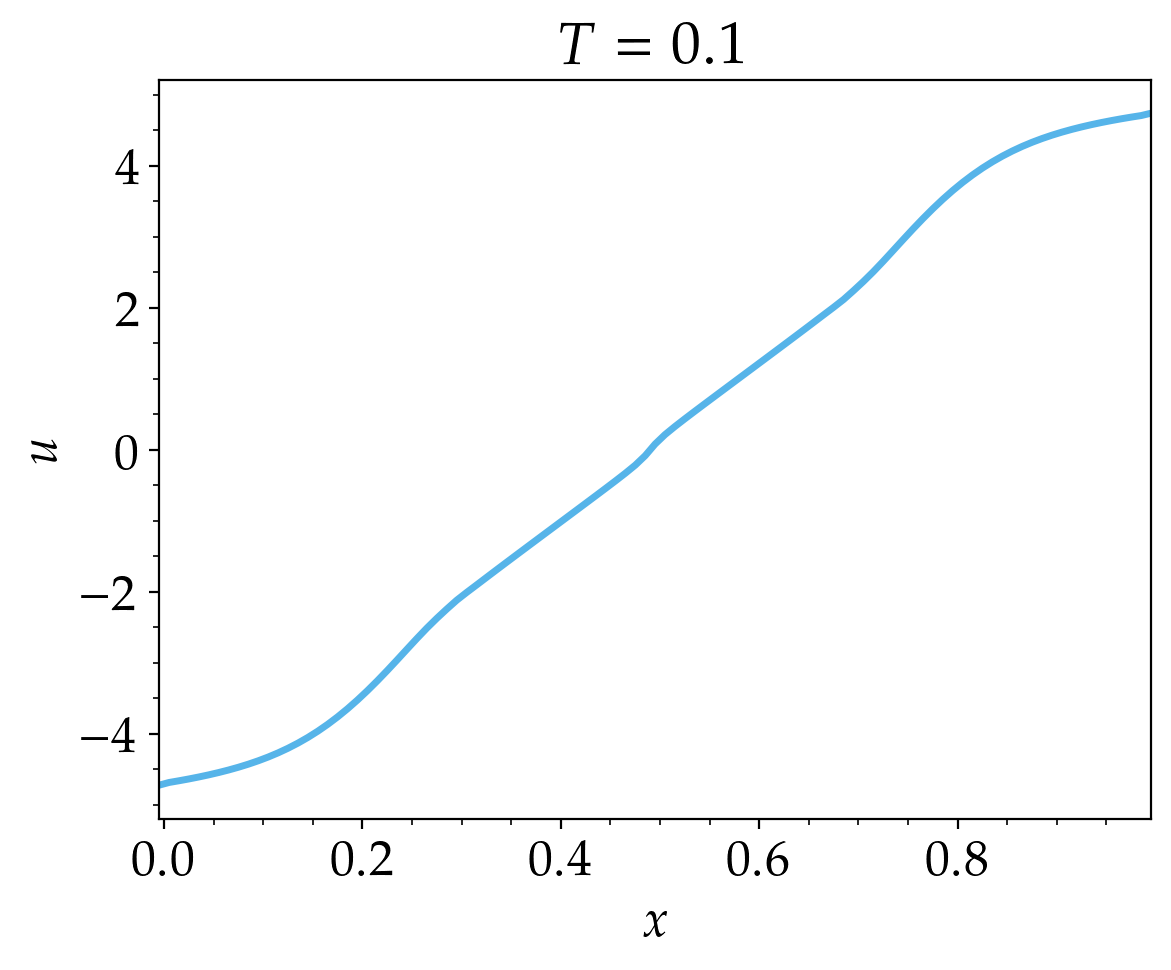}
  \caption{Cross-sections of density and velocity profiles at time
    $T=0.1$ for the strong rarefaction problem.}
  \label{fig:str_rrf}
\end{figure}

\subsection{Sod Problem with Gravitational Potential}
\label{subsec:sod}
This case study is the classical 1D Sod shock tube problem. Our goal
is to demonstrate the scheme's capability of resolving shock
discontinuities in the compressible regime. The initial conditions are
given by  
\begin{equation}
  \rho(0, x) =\begin{cases}
    1, &\mathrm{if}\;x< 0.5, \\
    0.125, &\mathrm{if}\;x\geq 0.5,
  \end{cases} \quad
  u(0,x)=0.
\end{equation}
The computational domain is $[0,1]$ and a linear gravitational
potential $\phi(x)=x$ is taken. The domain is discretised with $200$
grid points. Solid wall boundary conditions are applied on both the
ends. In order to simulate the compressible regime, we choose
$\veps=1$ and the adiabatic constant is taken as $\gm = 1.4$. The
simulations are run till the time $T=0.2$. We compare the computed 
density and velocity profiles to a reference solution obtained using a 
classical explicit Rusanov scheme on a fine mesh of 2000
points. The results are depicted in Figure\,\ref{fig:sod_1d} which
clearly shows a good agreement between the two solutions. It can be
noted that the scheme well captures the shock front propagating
towards the right followed by an expansion and that it emulates the
speed of the shock propagation correctly.   
\begin{figure}[htbp]
    \centering
    \includegraphics[height=0.25\textheight]{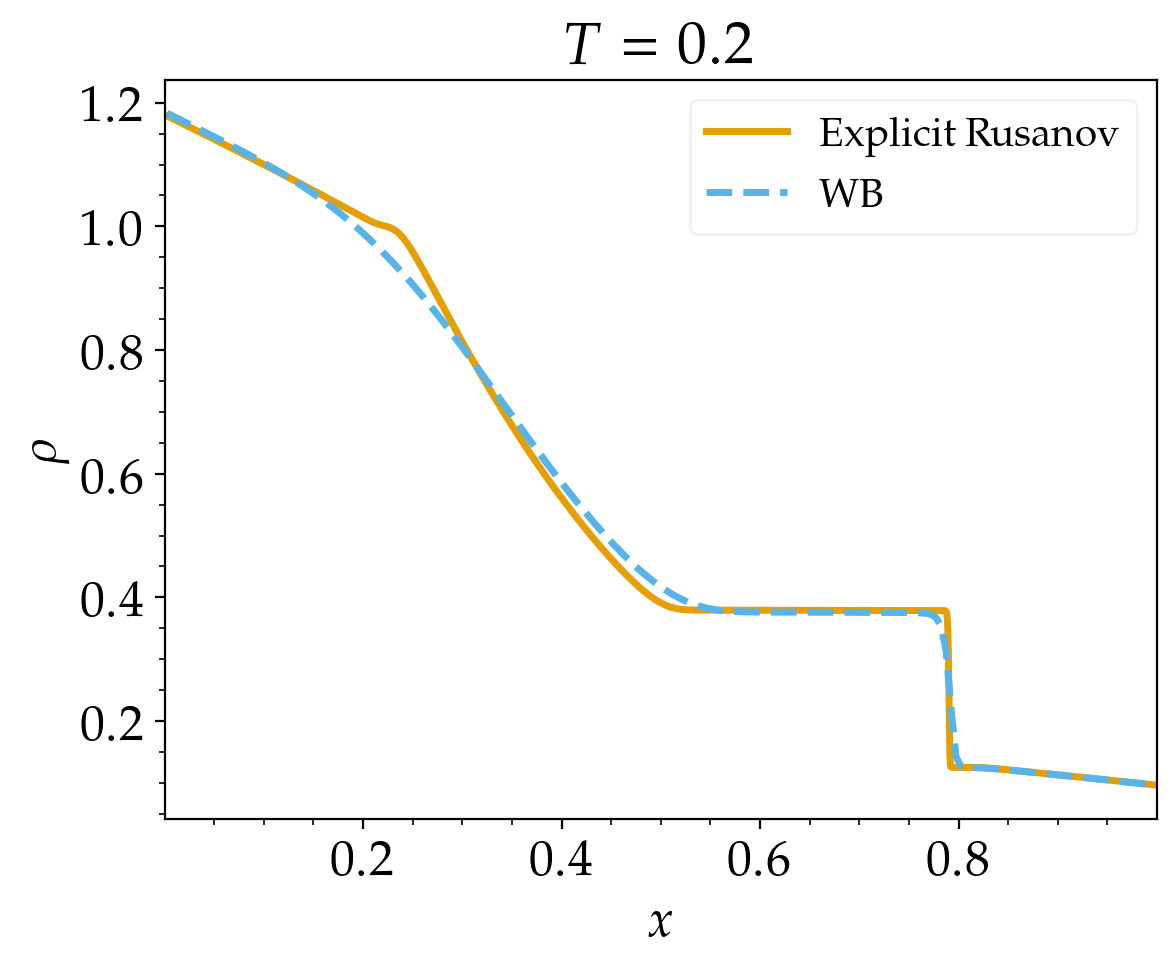}
    \includegraphics[height=0.25\textheight]{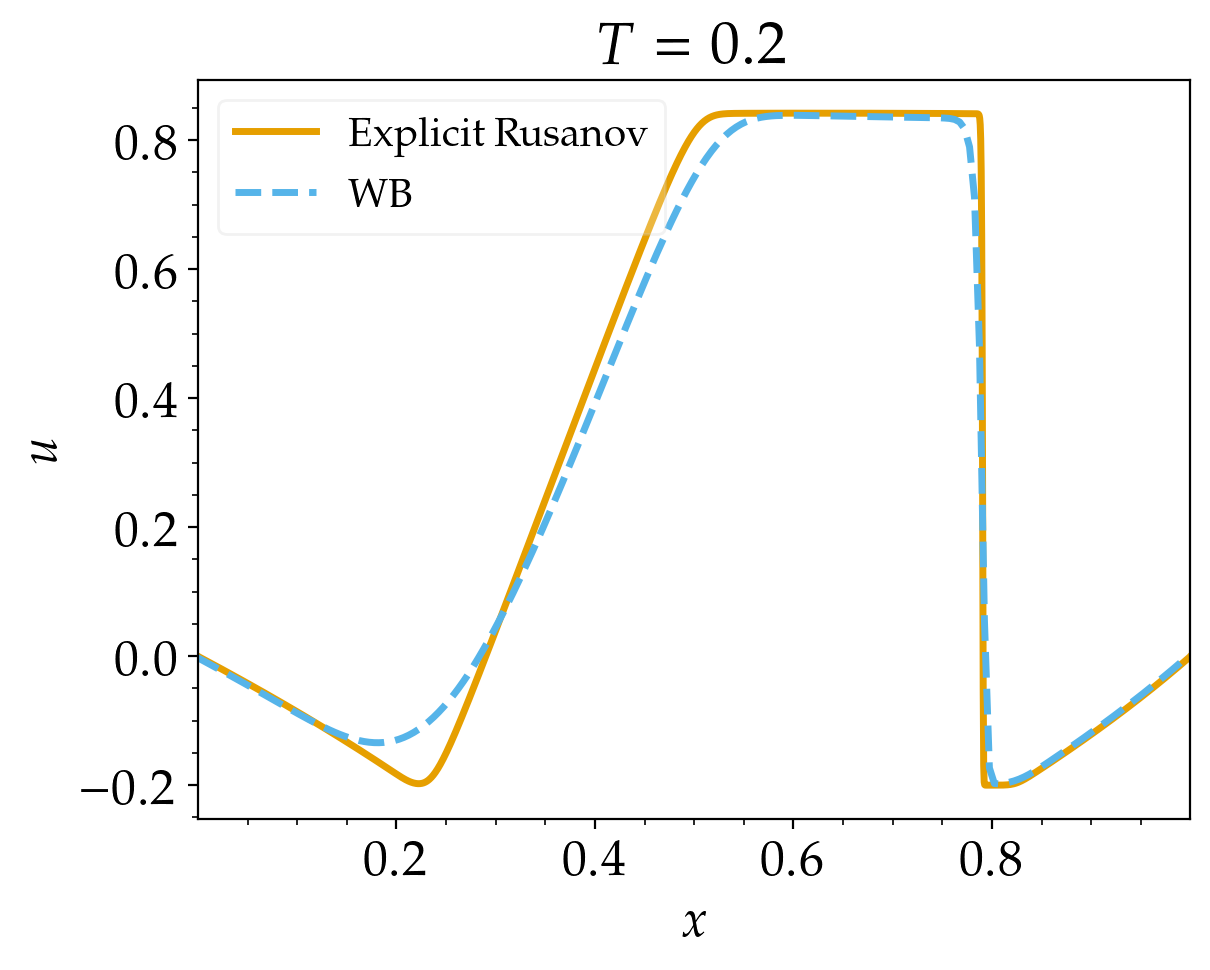}
    \caption{Density and velocity profiles for the Sod problem at time
      $T=0.2$.} 
    \label{fig:sod_1d}
\end{figure}
\subsection{1D Perturbation of Hydrostatic Steady State}
\label{subsec:1d_pert_hyd}
We test the performance of the proposed well-balanced AP scheme in
simulating the perturbation of a steady state against a non
well-balanced scheme for a range of values of $\veps$ from the
non-stiff to stiff regimes. The setup consists of adding a small
perturbation to a hydrostatic density profile and setting the
velocities to zero. The computational domain is $[0,1]$ and
the initial condition reads     
\begin{align}
    \rho(0,x) &=  \Big(1 - \frac{\gm -
                1}{\gm}\phi(x)\Big)^{\frac{1}{\gm - 1}} + \zeta
                e^{-100(x-0.5)^2}, \label{eq:1d_pert_den}\\
    u(0,x) &= 0.
\end{align}
We choose the gravitational potential function $\phi(x) = x$, the gas
constant $\gm = 1.4$, and the final time $T=0.25$. The domain is
discretised using $100$ grid points. The boundary conditions are
implemented by interpolating the steady state values for all the
variables in the ghost cells. 

In Figure\,\ref{fig:pert_hydst_1d_nstiff} we plot the density
perturbations obtained using $\veps = 1$ for the well-balanced and a
non well-balanced Rusanov scheme. The amplitudes of the perturbation
are taken as $\zeta = 10^{-3}$ and
$10^{-5}$. Figure\,\ref{fig:pert_hydst_1d_nstiff}(A) indicates that 
when $\zeta = 10^{-3}$, both the schemes diminish the perturbations
even though the well-balanced scheme shows a superior
performance. However, when the steady state is approached, i.e.\ when
$\zeta = 10^{-5}$, the non well-balanced scheme fails to capture the
perturbations as evident from Figures \ref{fig:pert_hydst_1d_nstiff}(B) and 
\ref{fig:pert_hydst_1d_nstiff}(C).      
 
\begin{figure}[htbp]
  \centering
  \includegraphics[height=0.18\textheight]{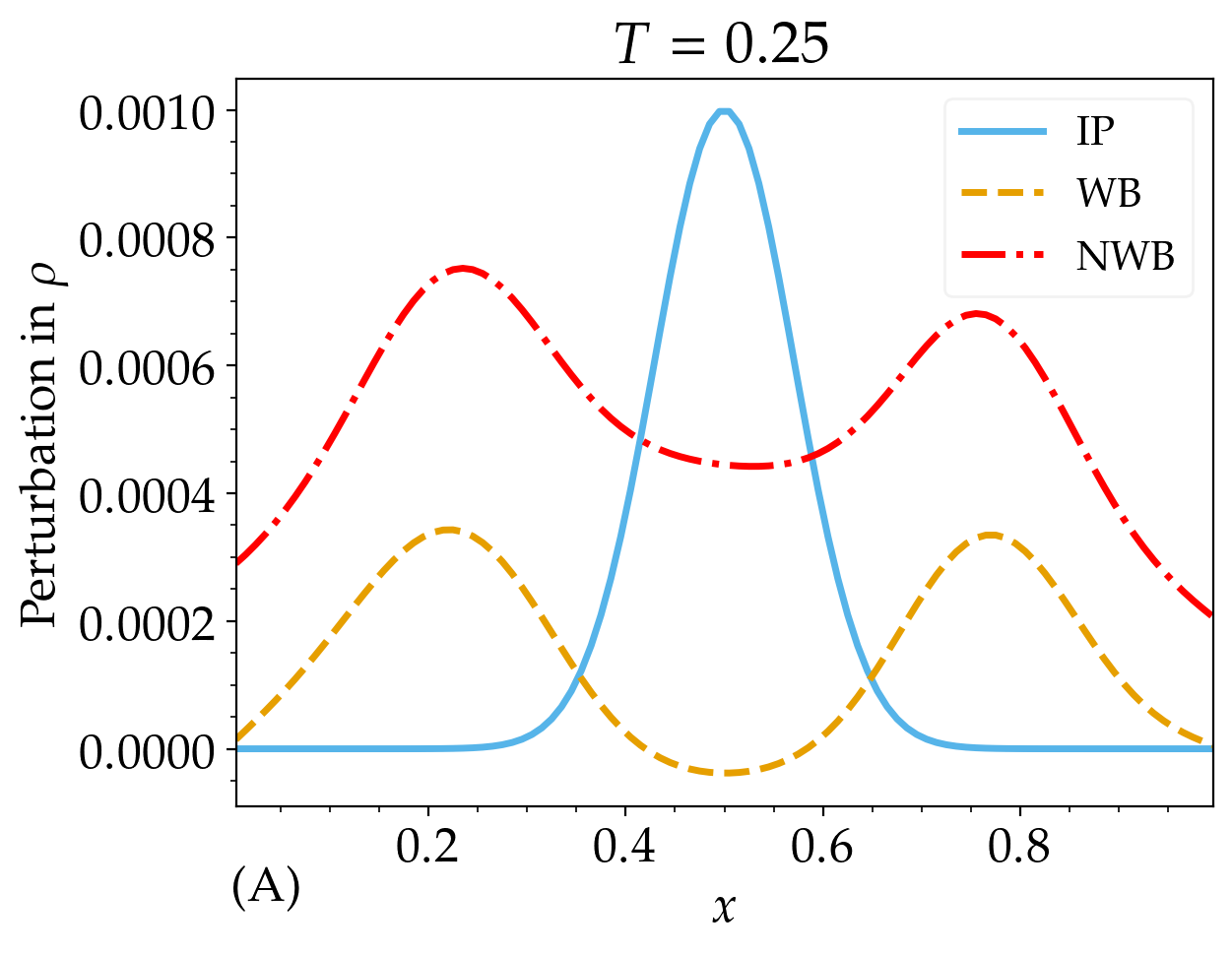}
  \includegraphics[height=0.18\textheight]{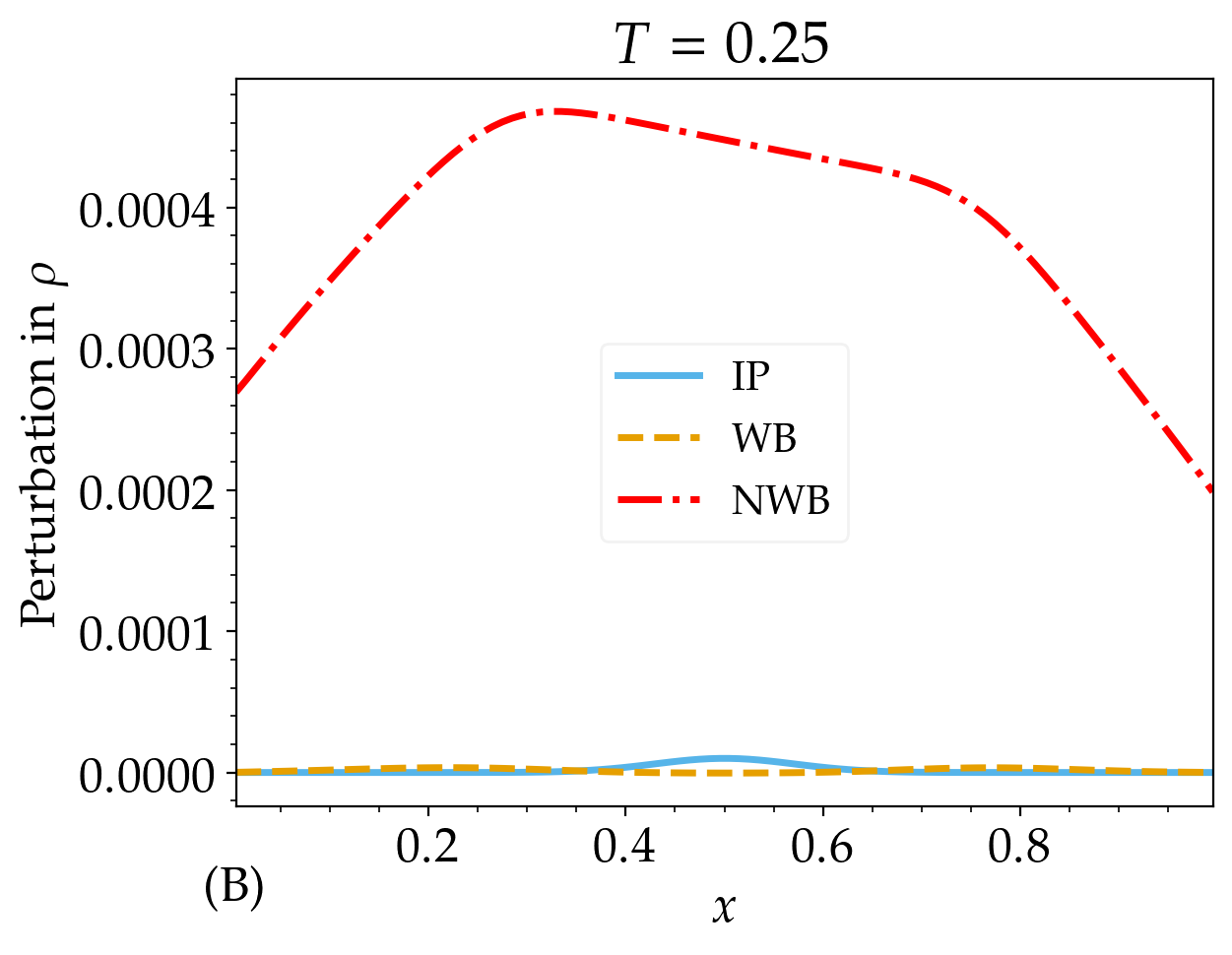}
  \includegraphics[height=0.18\textheight]{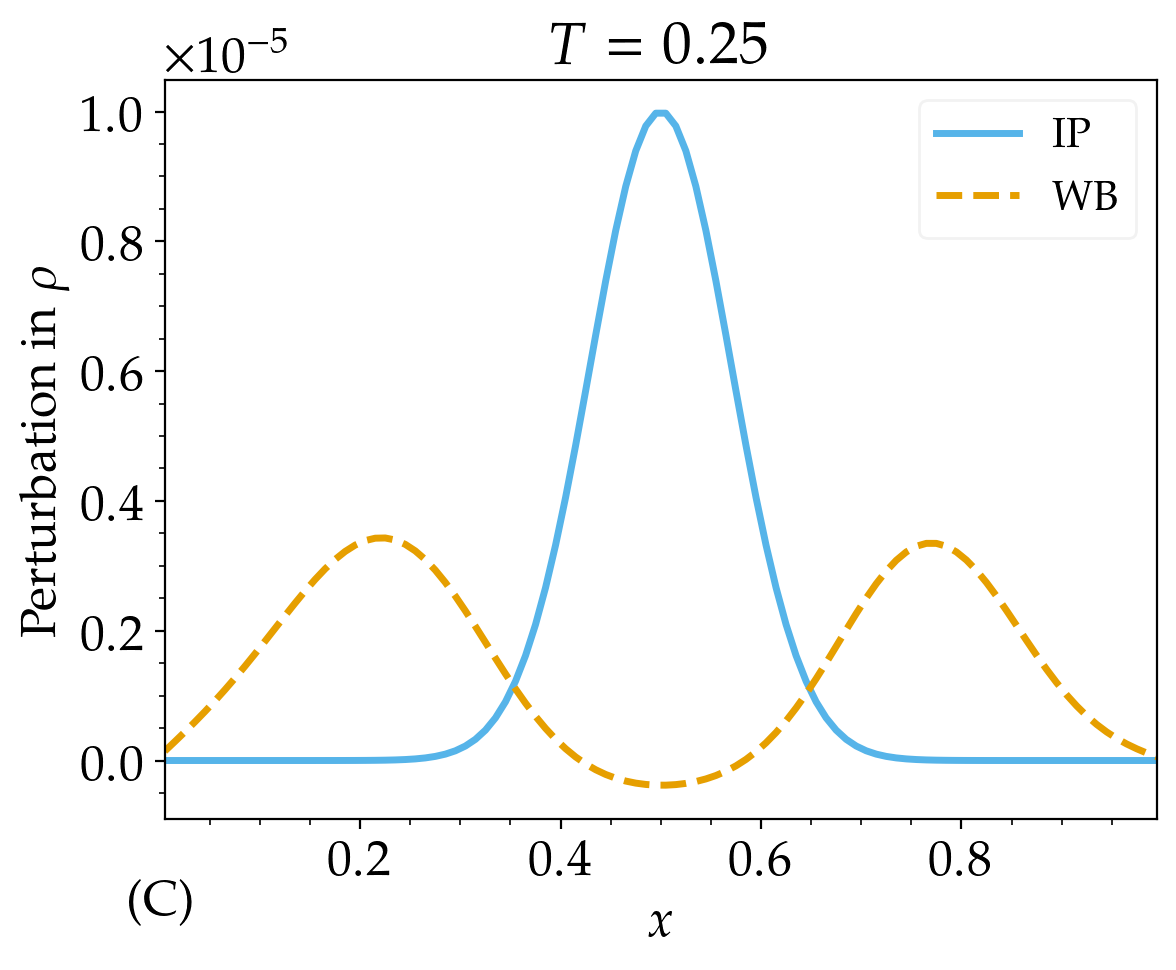}
  \caption{Comparison of the perturbations in density at time
    $T=0.25$ for $\veps = 1$. (A) $\zeta = 10^{-3}$ and (B) $\zeta 
    = 10^{-5}$. (C) Zoom of the plot in (B) for the initial
    perturbation and the well-balanced scheme.}
  \label{fig:pert_hydst_1d_nstiff}
\end{figure}

Next, we consider the stiff regime and set $\veps = 0.1, 0.01$ and
$0.001$. The amplitude of the perturbation in \eqref{eq:1d_pert_den} is
taken as $\zeta = \veps^2$ to simulate a well-prepared initial
data. We notice a blow-up of the density computed by the Rusanov 
scheme when $\veps=0.1$, and for smaller values of $\veps$ the Rusanov 
code even crashes. However, in all the cases, the well-balanced AP scheme
is able to produce the correct solution even on a coarse mesh of $100$
points. The results obtained are plotted in Figures
\ref{fig:pert_hydst_1d_stiff}(A)-\ref{fig:pert_hydst_1d_stiff}(C). 
  
\begin{figure}[htbp]
  \centering
  \includegraphics[height=0.18\textheight]{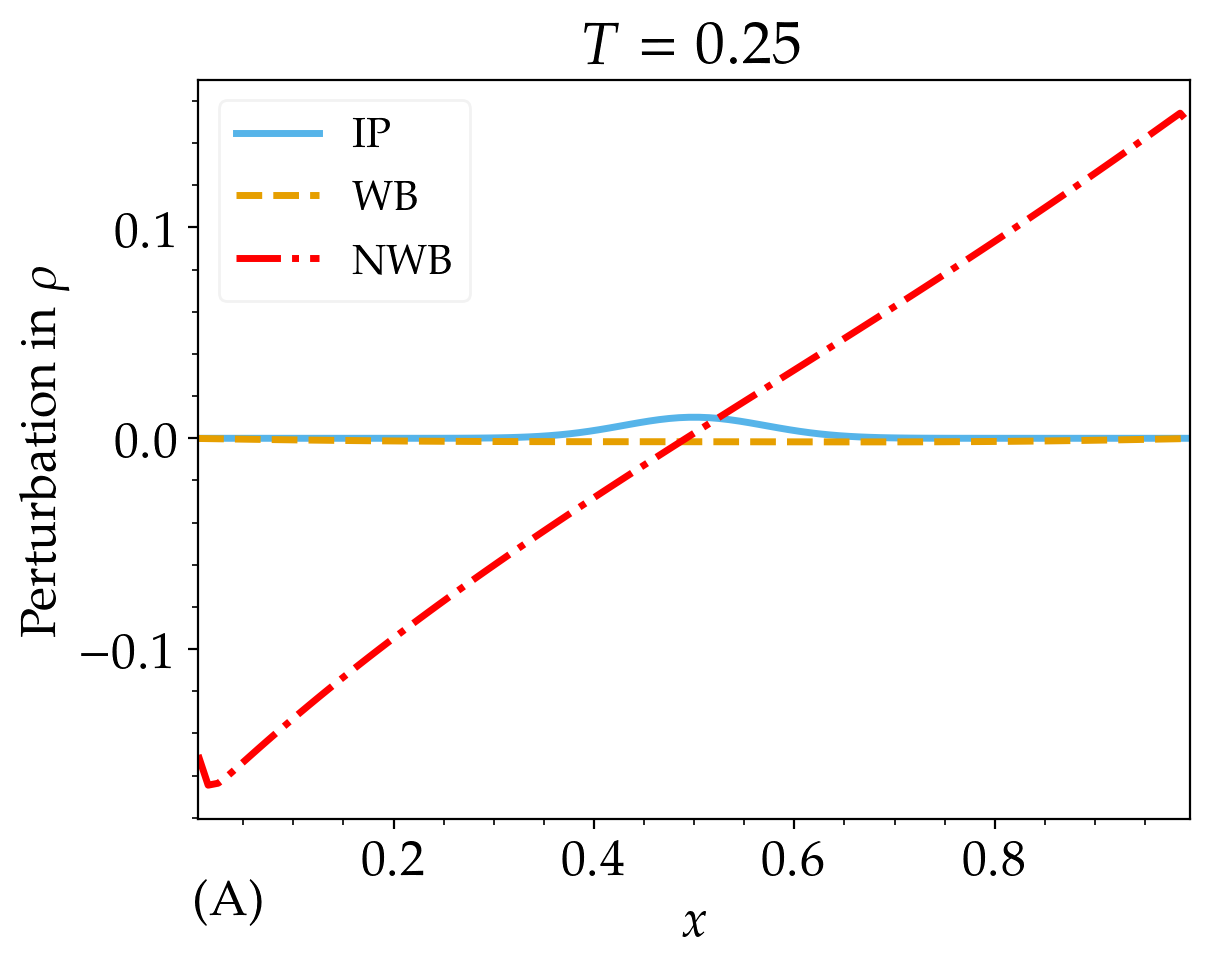}
  \includegraphics[height=0.18\textheight]{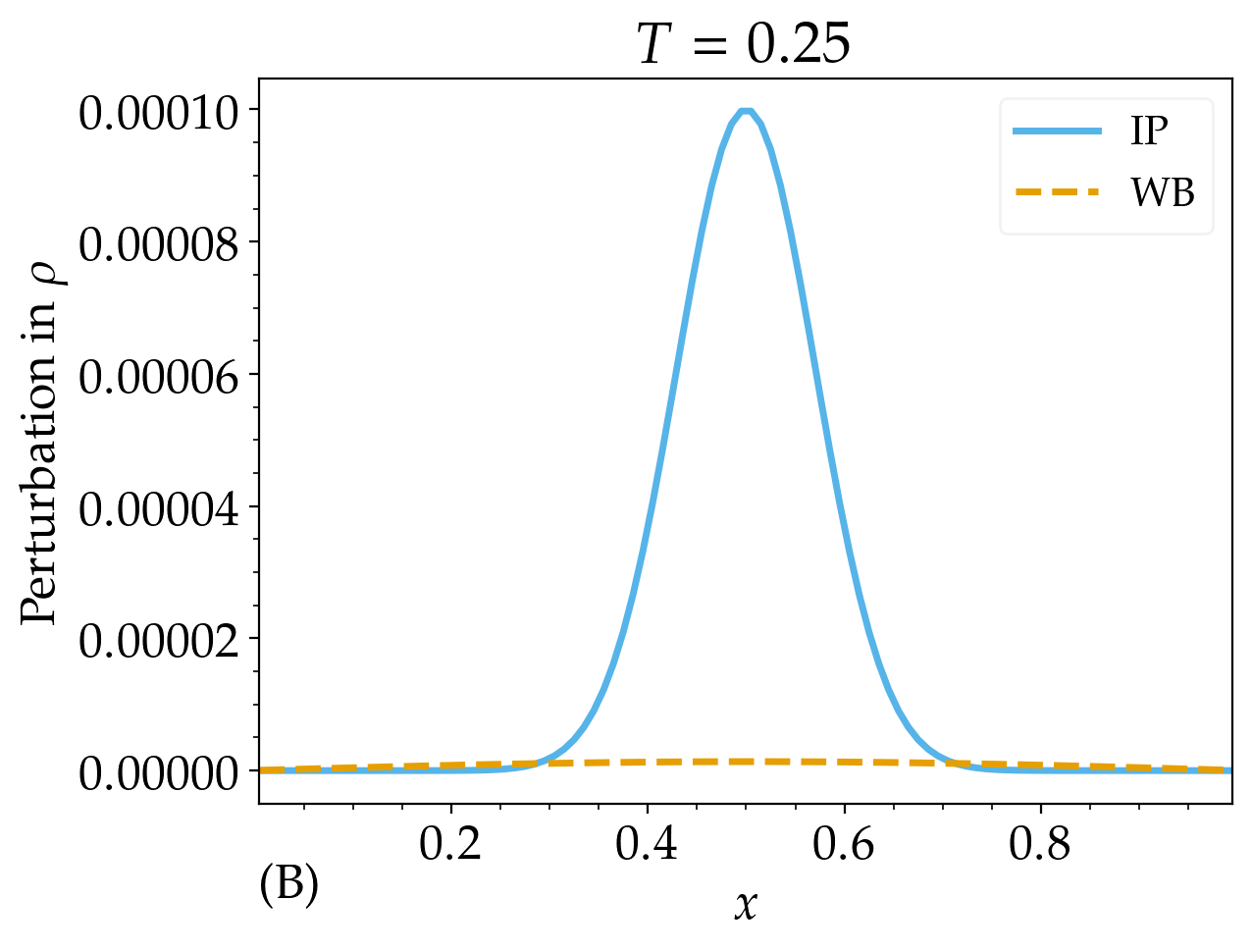} 
  \includegraphics[height=0.18\textheight]{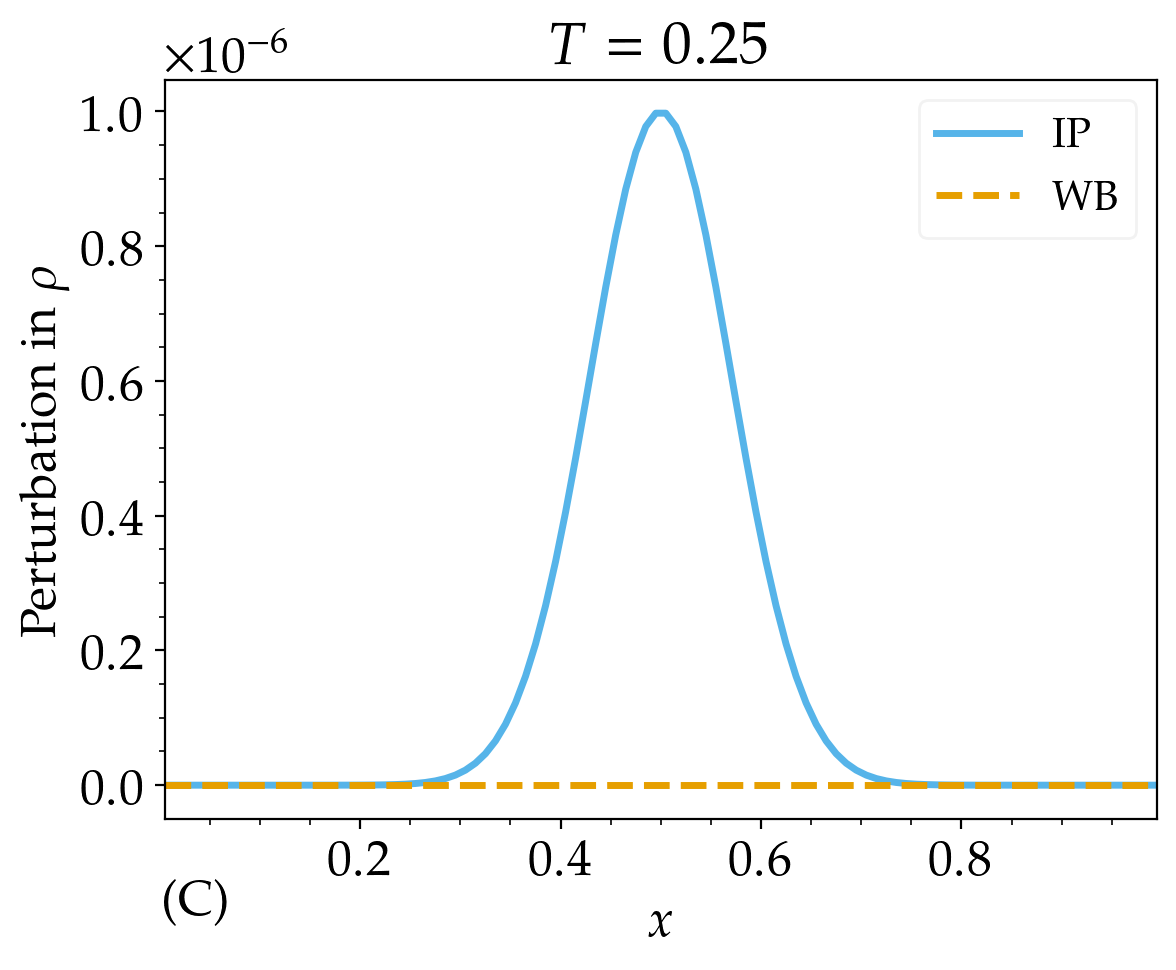}
  \caption{Comparison of the perturbations in density at time
    $T=0.25$. (A) $\veps = 0.1$, (B) $\veps = 0.01$ and (C) $\veps =
    0.001$. The amplitude of the perturbation is $\zeta = \veps^2$. The
    non well-balanced scheme crashes beyond $\veps=0.1$.}
  \label{fig:pert_hydst_1d_stiff}
\end{figure}

\subsection{2D Perturbation of Hydrostatic Steady State}
\label{subsec:2d_pert_hyd}
We consider the following two-dimensional (2D) test case where we add
a small perturbation to the hydrostatic density. The computational
domain is $[0,1]\times[0,1]$ and the initial conditions are given by 
\begin{align*}
    \rho(0,x,y) &=  \Big(1 - \frac{\gm -
                  1}{\gm}\phi(x,y)\Big)^{\frac{1}{\gm - 1}} + \zeta
                  e^{-100[(x-0.3)^2 + (y-0.3)^2]}, \\
  (u, v)(0,x, y) &= (0, 0).
\end{align*}
We choose the gravitational potential function $\phi(x,y)=x+y$
and the adiabatic constant $\gm = 1.4$. Note that $\phi$ is not
aligned with the coordinate directions. The domain is discretised with
$50\times50$ grid points and transmissive boundary conditions are
taken on all sides of the domain. We simulate the flow using the
well-balanced AP scheme and a non well-balanced explicit Rusanov
scheme and assess the performance of the two.     

First, we consider the non-stiff ($\veps = 1$) case and set the
amplitude of the perturbation to $\zeta = 10^{-1}$ and $10^{-3}$. The
final time of the run is taken as $T=0.05$. In Figures
\ref{fig:pert_hydst_2d_zm1_nstiff} and
\ref{fig:pert_hydst_2d_zm3_nstiff} we plot the density perturbations
obtained for the two values of $\zeta$. The results clearly indicate
the superior performance of the well-balanced scheme over its non
well-balanced counterpart which produces distortions near the
boundary. The circular structure of the perturbations are well
resolved by the well-balanced scheme despite the gravitational field
being non-aligned with the coordinate axes. 

\begin{figure}[htbp]
  \centering
  \includegraphics[height=0.30\textheight]{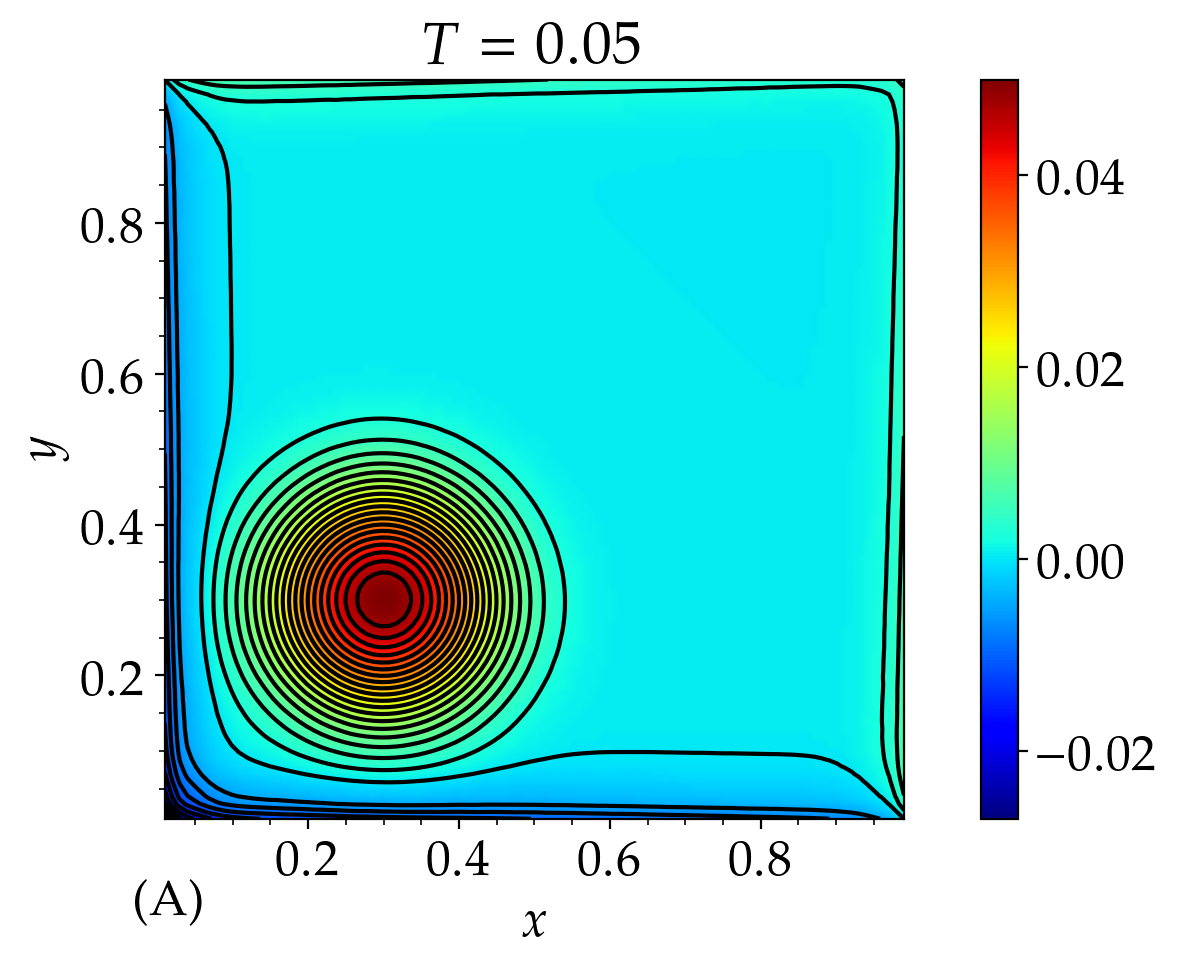}
  \includegraphics[height=0.30\textheight]{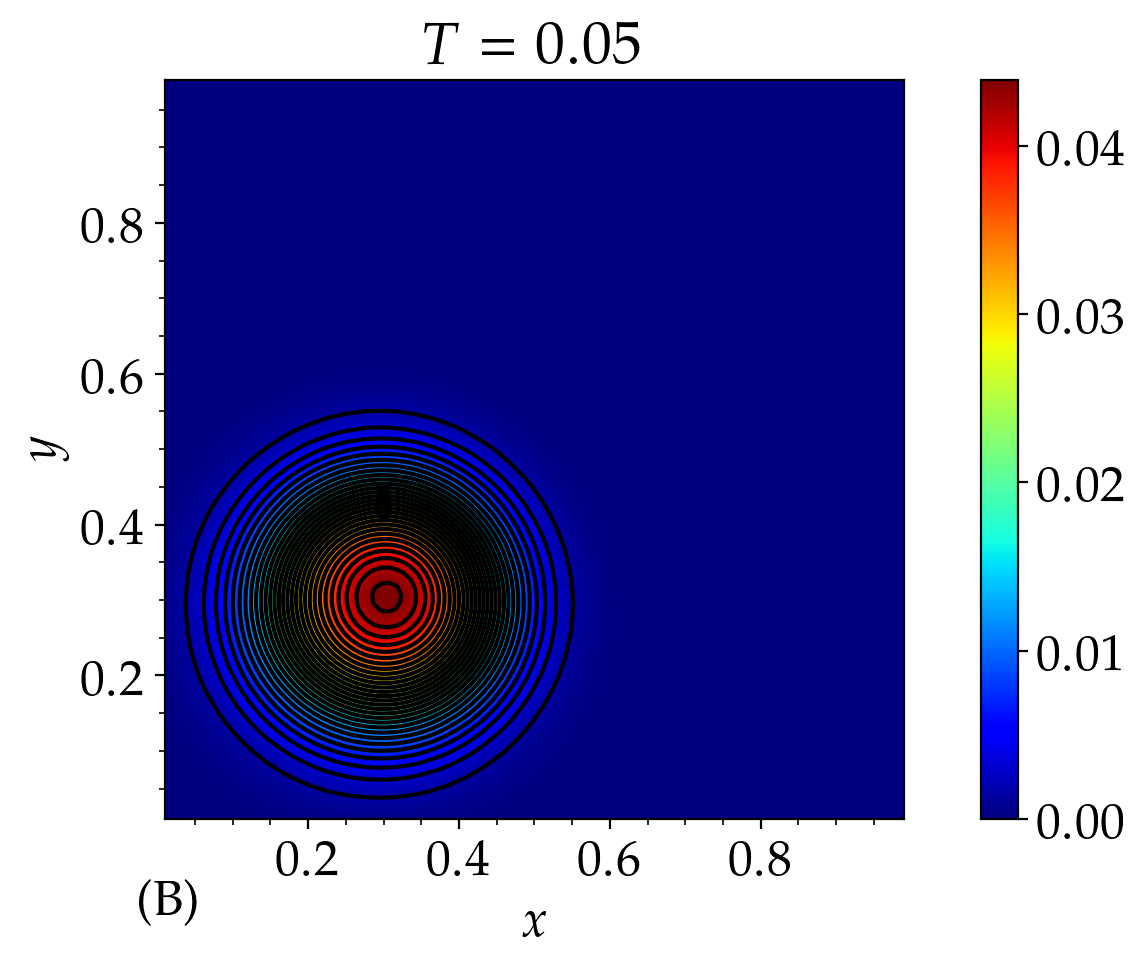} 
  \caption{Pseudo-colour plots and contours of the density
    perturbations obtained by (A) the non well-balanced explicit
    scheme and (B) the well-balanced scheme at time $T=0.05$ with
    $\veps = 1$ and $\zeta=10^{-1}$.} 
  \label{fig:pert_hydst_2d_zm1_nstiff}
\end{figure}
\begin{figure}[htbp]
  \centering
  \includegraphics[height=0.30\textheight]{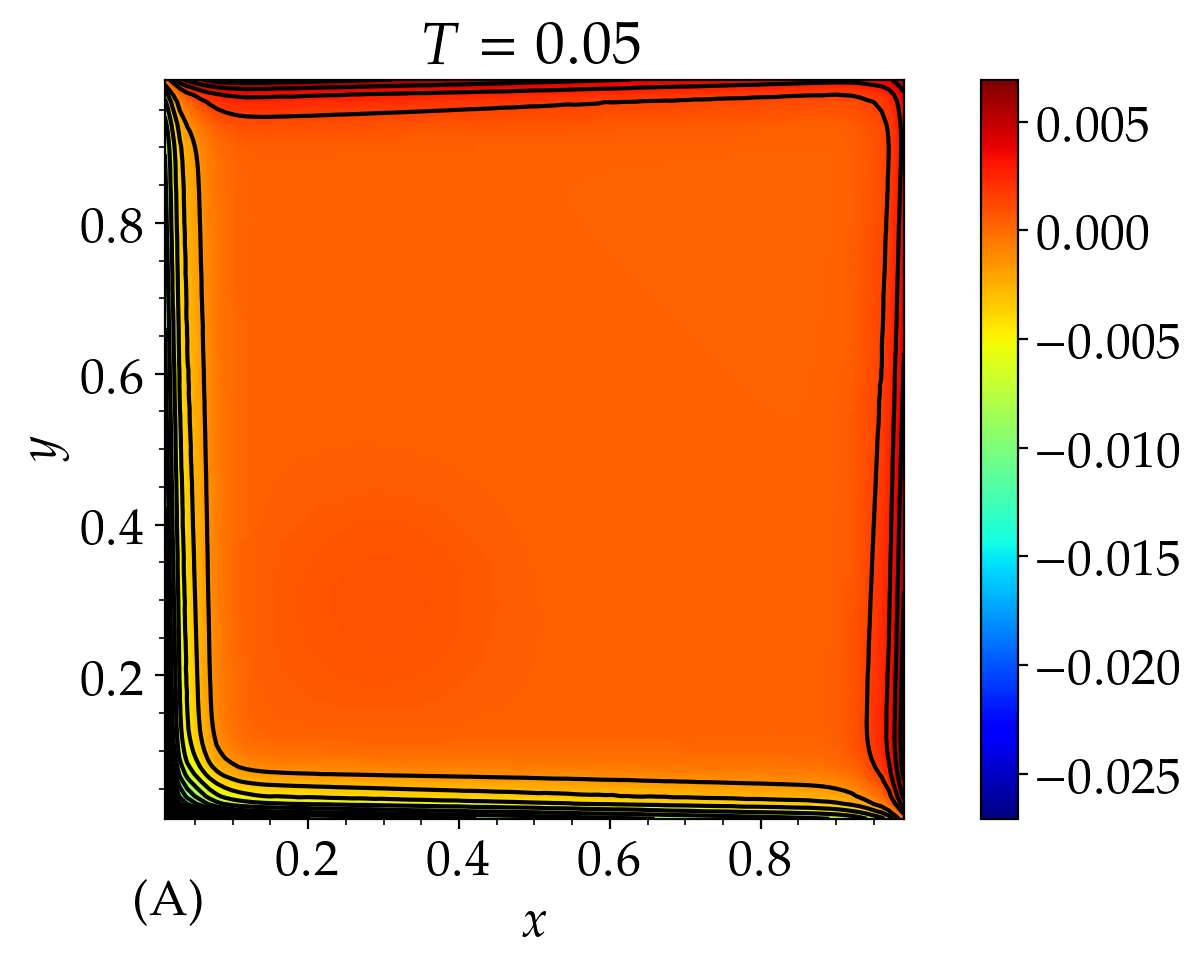}
  \includegraphics[height=0.30\textheight]{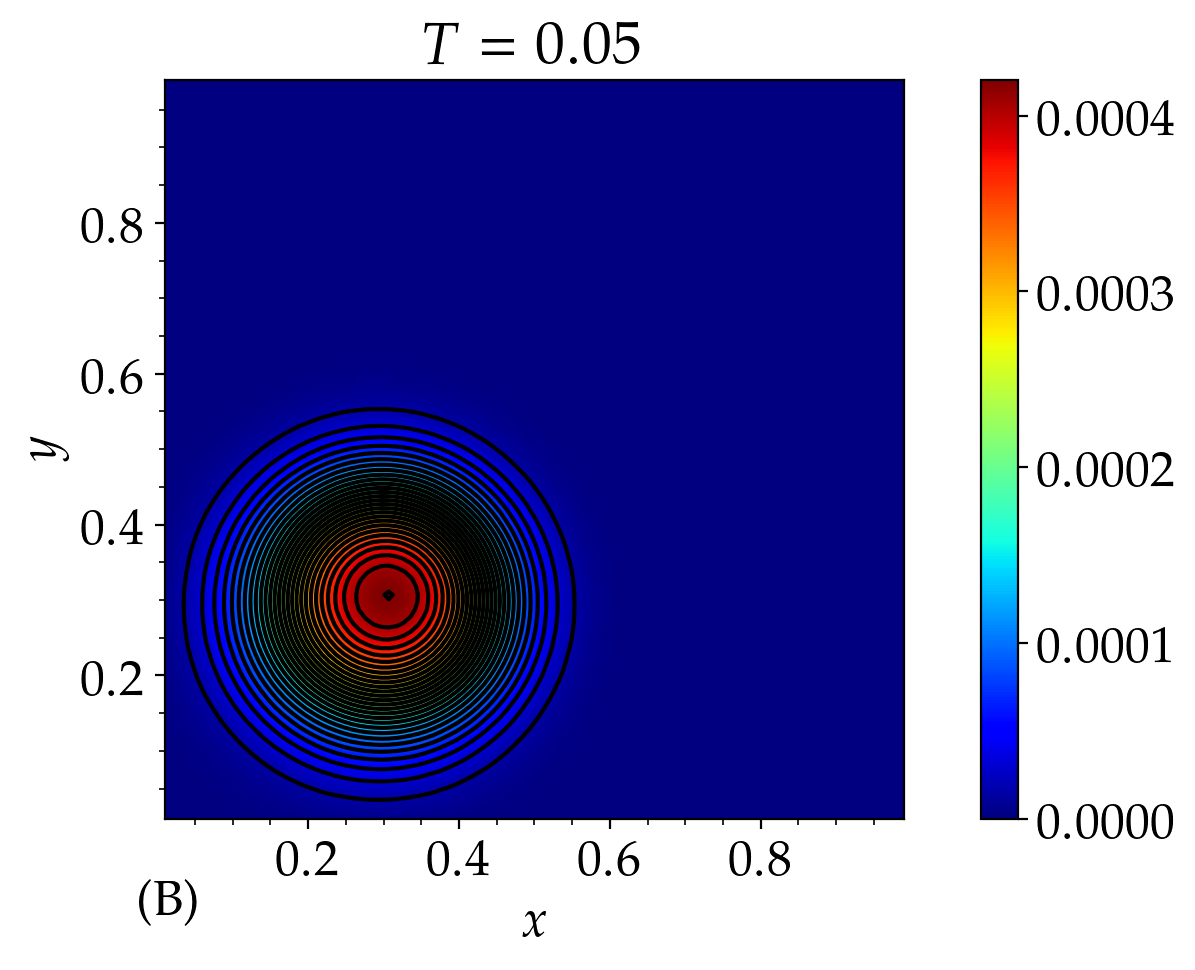} 
  \caption{Pseudo-colour plots and contours of the density
    perturbations obtained by (A) the non well-balanced explicit
    scheme and (B) the well-balanced scheme at time $T=0.05$ with
    $\veps = 1$ and $\zeta=10^{-3}$.}
  \label{fig:pert_hydst_2d_zm3_nstiff}
\end{figure}

Next, we assess the performance of the schemes in the stiff regimes by
setting $\veps = 0.1$ and $0.01$. In order to keep the data
well-prepared, we set $\zeta = \veps^2$ as before. Similar to the 1D
case, the non well-balanced code crashes for the chosen values
of $\veps$. Figure\,\ref{fig:pert_hydst_2d_stiff} clearly indicates that 
well-balanced AP scheme is able to accurately resolve the circular
structure of the perturbations.    

\begin{figure}[htbp]
  \centering
  \includegraphics[height=0.30\textheight]{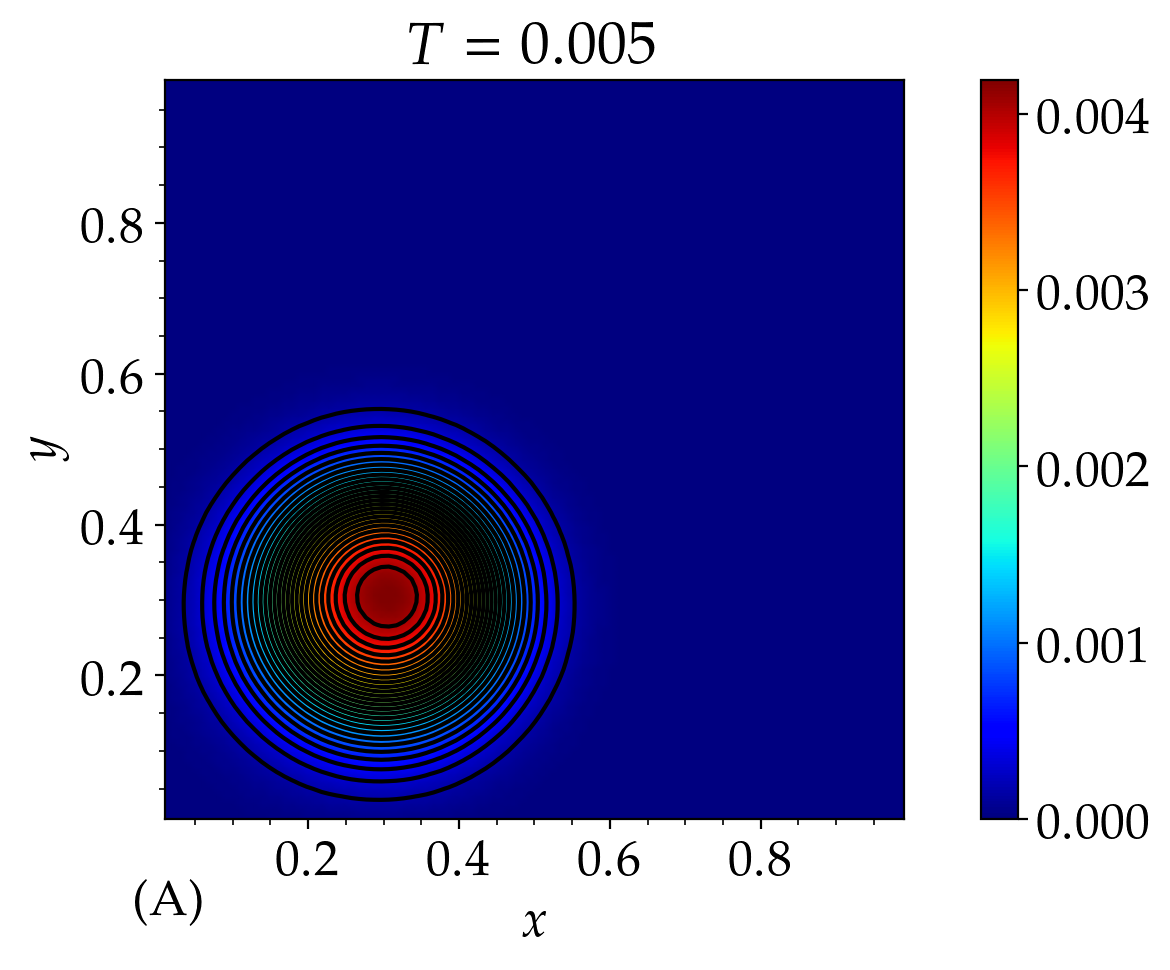}
  \includegraphics[height=0.30\textheight]{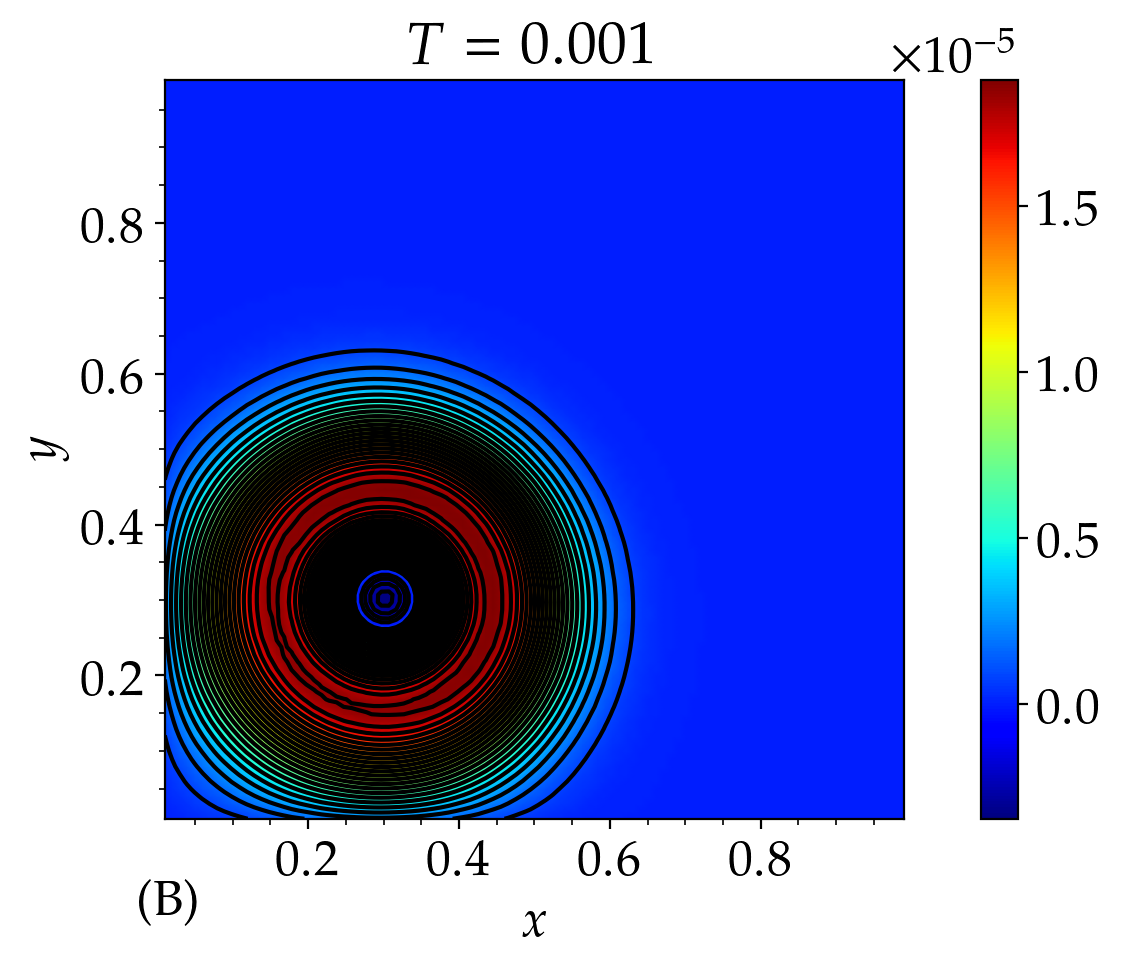} 
  \caption{Pseudo-colour plots and contours of the density perturbation
    computed using the well-balanced scheme. (A) $\veps=0.1$,
    $T=0.005$ (B) $\veps=0.01$, $T=0.001$.}
  \label{fig:pert_hydst_2d_stiff}
\end{figure}
\subsection{Stationary Vortex}
\label{subsec:stat_vort}
We consider the following test problem inspired by \cite{Zen18}
which simulates a stationary vortex under the action of a
gravitational field. The vortex test problem is a benchmark to assess
the low Mach number properties of numerical schemes and the
dependence of their dissipation on the Mach number. Since an exact
stationary solution is available, we also make use this problem to
corroborate the uniform first order convergence of the scheme with
respect to $\veps$. 

The computational domain is $[0,1]\times[0,1]$, the pressure law is
$\wp(\rho)=\rho^2$, and the gravitational potential is radially
symmetric with $\phi(r) = r^2$, where $r=\sqrt{(x-0.5)^2+(y-0.5)^2}$
denotes the distance from the center of the domain $(0.5, 0.5)$.

We define a radially symmetric and piecewise continuous angular
velocity, independent of $\phi$, as
\begin{equation*}
  u_\theta(r) =\begin{cases}
    a_1 r, &\mathrm{if}\;r\leq r_1,
    \\
    a_2 + a_3 r, &\mathrm{if}\;r_1<r\leq r_2,
    \\
    0, &\mathrm{otherwise}.
  \end{cases}
\end{equation*}
Here the inner and outer radii are $r_1 = 0.2,\; r_2 =
0.4$ and the constants are chosen as   
\begin{equation*}
  a_1 = \frac{\bar{a}}{r_1},\; a_2 = -\frac{\bar{a}r_2}{r_1-r_2},\;
  a_3
  = \frac{\bar{a}}{r_1 - r_2},    
\end{equation*}
with $\bar{a} = 0.1$. Since the vortex is stationary and has zero
radial velocity, it satisfies the balance
\begin{equation}
  \label{eq:stat_vtx}
  \frac{1}{\veps^2}\D_r\wp(\rho)=
  \frac{\rho u^2_{\theta}}{r}-\frac{1}{\veps^2}\rho\D_r\phi.
\end{equation}
We choose $\rho(0)=1$, use the given expressions for $u_\theta$
and $\phi$, and integrate \eqref{eq:stat_vtx} to obtain the following
stationary solution of the Euler system
\eqref{eq:cons_mas}-\eqref{eq:cons_mom} in polar co-ordinates:  
\begin{align*}
    \rho(0,x,y) &= 1 + \frac{\veps^2}{2}\int_0^r \frac{u^2_\theta(s)}{s}\dd s - \half r^2,
    \\
    u(0,x,y) &= u_\theta(r)(y - 0.5)/r,
    \\
    v(0,x,y) &= u_\theta(r)(0.5 - x)/r.
\end{align*}
The computational domain is discretised into $25\times 25$, $50\times
50$, $100\times 100$, $200\times 200$ and $400\times 400$ mesh points
and the final time is set to $T=1$. Periodic boundary conditions are
implemented on all sides of the domain.    

The $L^1$-errors in $\rho,\; \rho u$ and $\rho v$ about the stationary
solution along with the experimental order of convergence (EOC) for
different values of $\veps$ are tabulated in
Table\,\ref{tab:acc_tst}. The EOC values clearly indicate the uniform
first order convergence of the scheme irrespective of the value of
$\veps$.  

\begin{table}[ht]
\centering
\begin{adjustbox}{width=0.8\textwidth}
\small
\begin{tabular}{rlrrrrrrr}
  \hline
   $\veps$ & $N_x = N_y$ & $L^1$-error in $\rho$ & EOC & $L^1$-error in $\rho u$ & EOC &$L^1$-error in $\rho v$ & EOC\\ 
  \hline
                    &25	 &8.3926E-07 &--		&1.2063E-03 &--		&1.2063E-03 &--     \\
                    &50	 &4.6492E-07 &0.8521	&6.4911E-04 &0.8940	&6.4911E-04 &0.8940  \\
 $10^{-1}$          &100 &2.5632E-07 &0.8590	   &3.5668E-04 &0.8638 &3.5668E-04 &0.8638  \\
                    &200 &1.3345E-07 &0.9416	&1.8573E-04 &0.9414	&1.8573E-04 &0.9414  \\
                    &400 &6.8504E-08 &0.9620	&9.4904E-05 &0.9687	&9.4904E-05 &0.9687  \\
   \hline
                 &25	 &8.3044E-09 &--		&1.1826E-03 &--		&1.1826E-03 &--    \\
                 &50	 &4.4802E-09 &0.8903 	&6.1962E-04 &0.9325	&6.1962E-04 &0.9325 \\
 $10^{-2}$       &100 &2.3266E-09 &0.9453	&3.2613E-04 &0.9259	&3.2613E-04 &0.9259 \\
                 &200 &1.2075E-09 &0.9462	&1.6965E-04 &0.9429	&1.6965E-04 &0.9429 \\
                 &400 &6.3285E-10 &0.9321	&8.8317E-05 &0.9418	&8.8317E-05 &0.9418 \\
    \hline
                &25	 &8.2967E-11 &--		&1.1816E-03 &--		&1.1816E-03 &-- \\
                &50	 &4.4623E-11 &0.8948	&6.1780E-04 &0.9355	&6.1780E-04 &0.9355 \\
 $10^{-3}$      &100 &2.3021E-11 &0.9548	   &3.2427E-04 &0.9302 &3.2427E-04 &0.9302 \\
                &200 &1.1800E-11 &0.9642	&1.6771E-04 &0.9512	&1.6771E-04 &0.9512 \\
                &400 &6.0366E-12 &0.9670	&8.5966E-05 &0.9641	&8.5966E-05 &0.9641 \\            
    \hline
\end{tabular}
\end{adjustbox}
\caption{$L^1$-errors and EOC with respect to the exact solution.} 
\label{tab:acc_tst}
\end{table} 
In order to asses the dissipation of the scheme and its dependence on
$\veps$, we compute the flow Mach numbers $M=\sqrt{(u^2 + 
  v^2)/(\gm p/\rho)}$, the relative kinetic energies and vorticities 
for different values of $\veps \in \{10^{-1}, 10^{-2},
10^{-3}, 10^{-4}\}$. A comparison of the initial Mach number with the
ones at $T=1$ for different values of $\veps$ in
Figure\,\ref{fig:stn_vrtx_eps_M} shows that the dissipation of the
scheme is independent of $\veps$ as the distortion of the vortex
remains the same for different $\veps$.   
\begin{figure}[htbp]
  \centering
  \includegraphics[height=0.30\textheight]{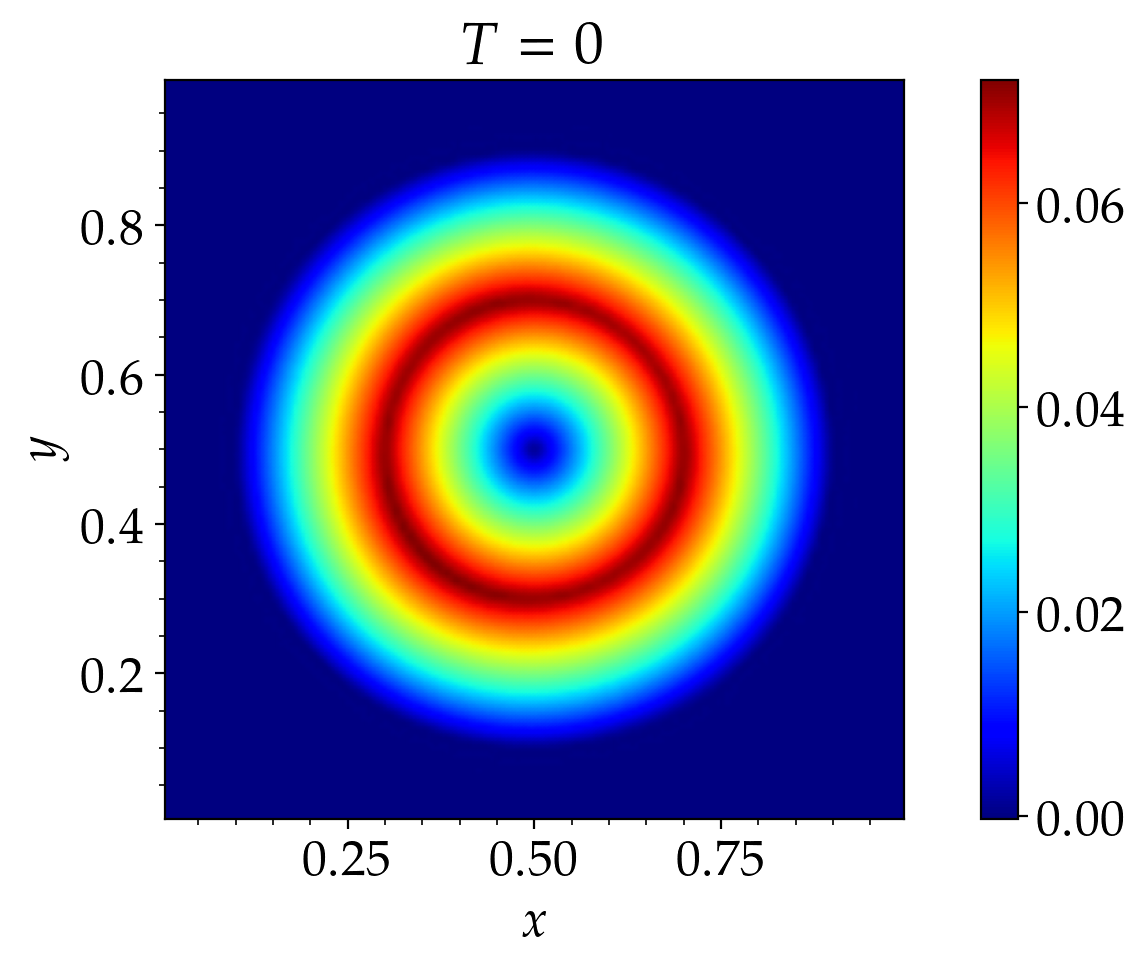}
  \caption{Pseudo-colour plot of the flow Mach number at time $T=0$.}
  \label{fig:stn_vrtx_ini_M}
\end{figure}
\begin{figure}[htbp]
  \centering
  \includegraphics[height=0.23\textheight]{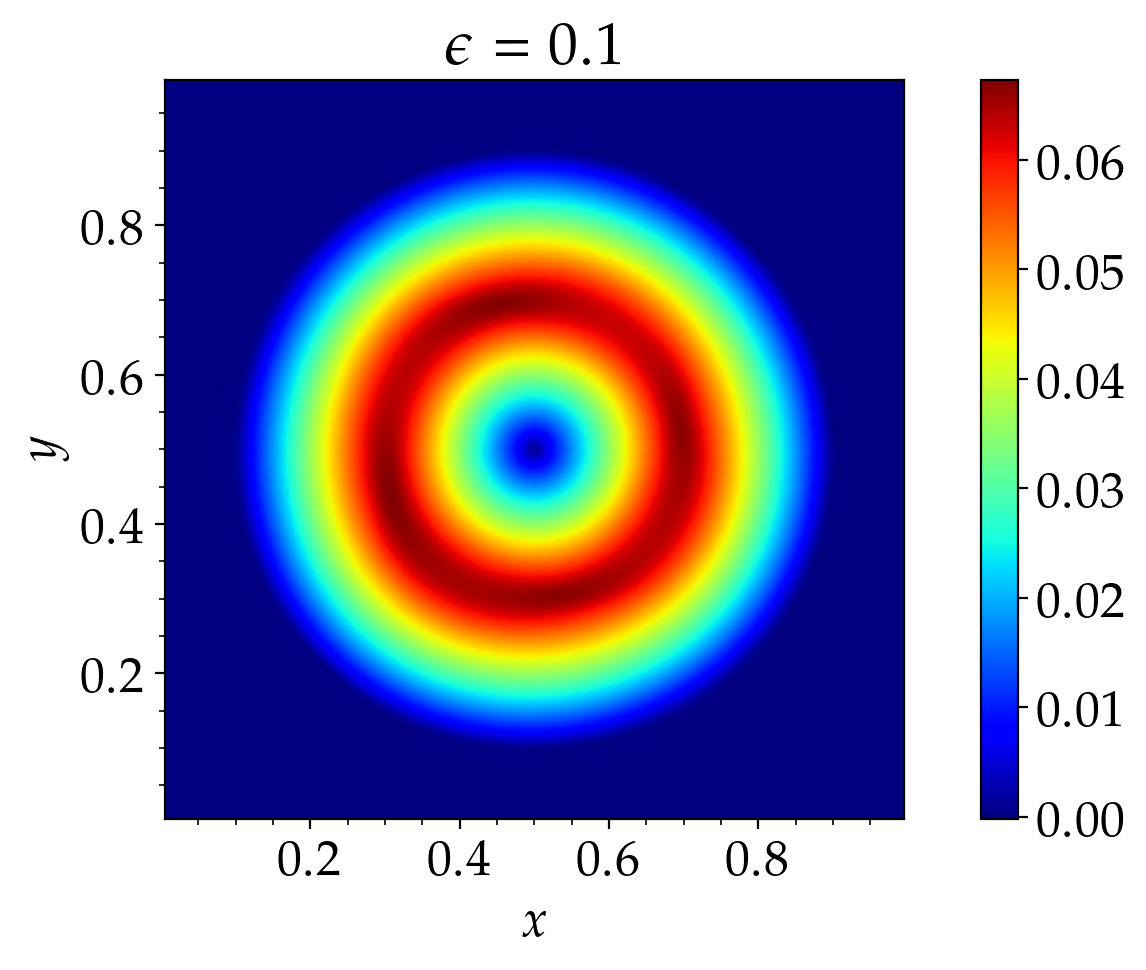}
  \includegraphics[height=0.23\textheight]{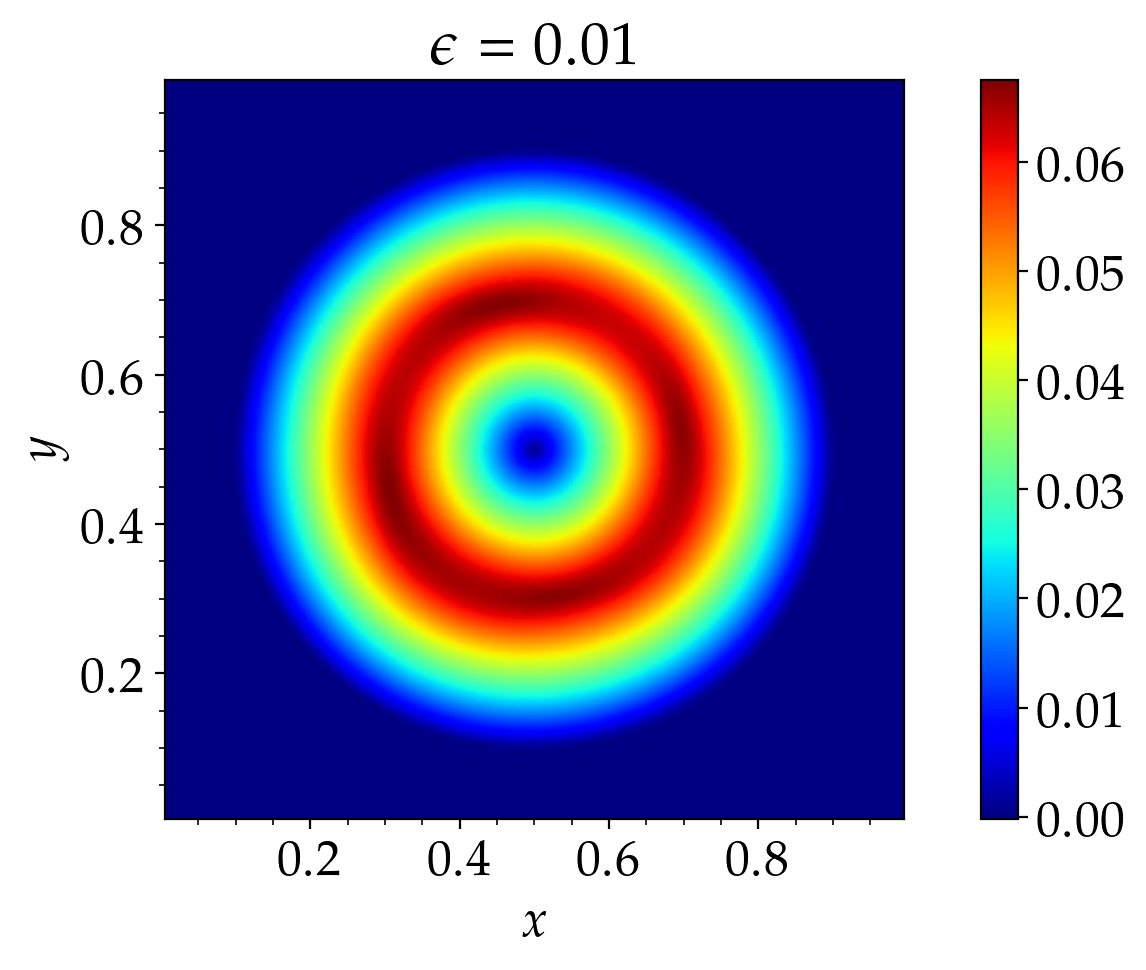} 
  \includegraphics[height=0.23\textheight]{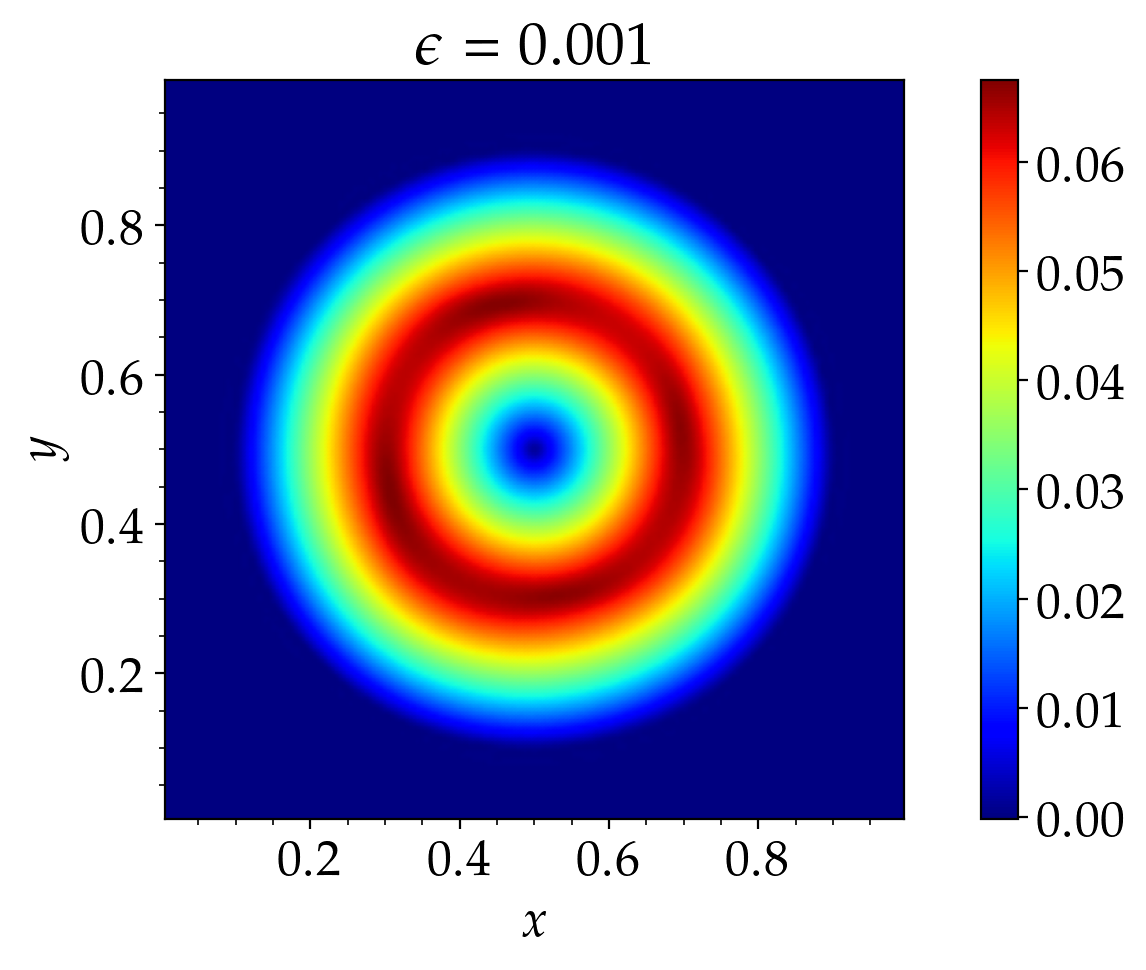}
  \includegraphics[height=0.23\textheight]{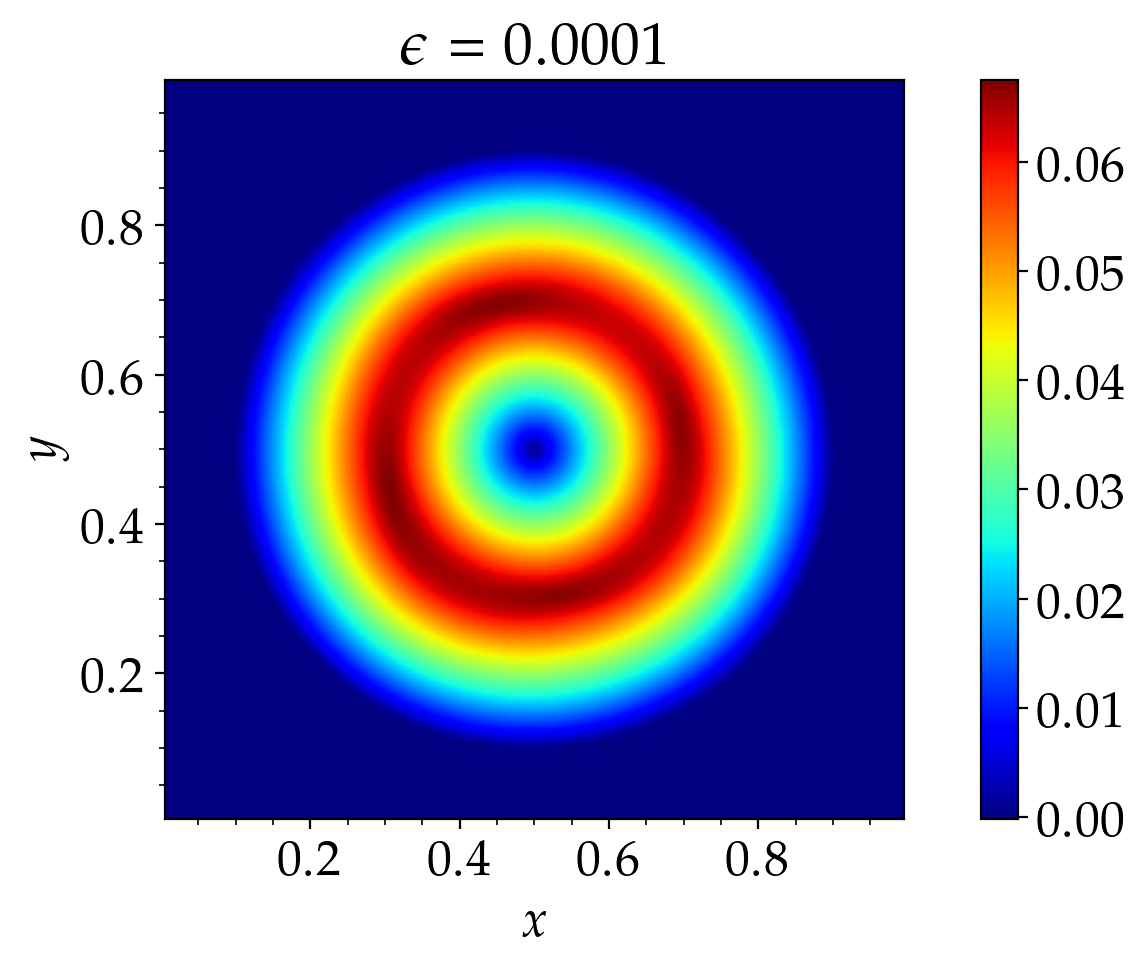}  
  \caption{Pseudo-colour plots of the flow Mach number at time $T=1$
    for different values of $\veps$.} 
  \label{fig:stn_vrtx_eps_M}
\end{figure}
In order to further corroborate the above fact, we plot in
Figure\,\ref{fig:stn_vrtx_rel_KE_vort} the relative kinetic energies
versus time and a cross-section of the vorticity at the final
time. Clearly, the figures corresponding to different $\veps$ are
indistinguishable.    
\begin{figure}[htbp]
  \centering
  \includegraphics[height=0.25\textheight]{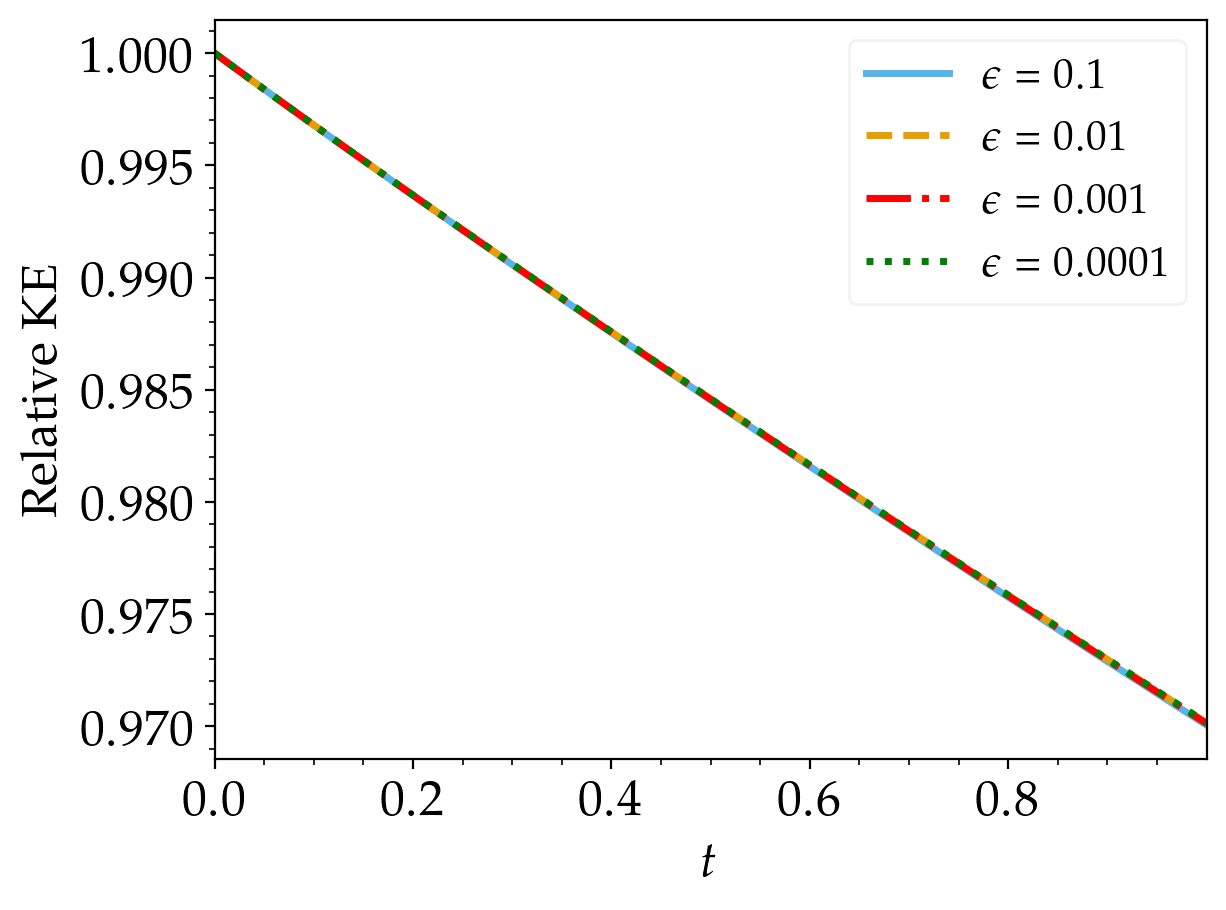}
  \includegraphics[height=0.25\textheight]{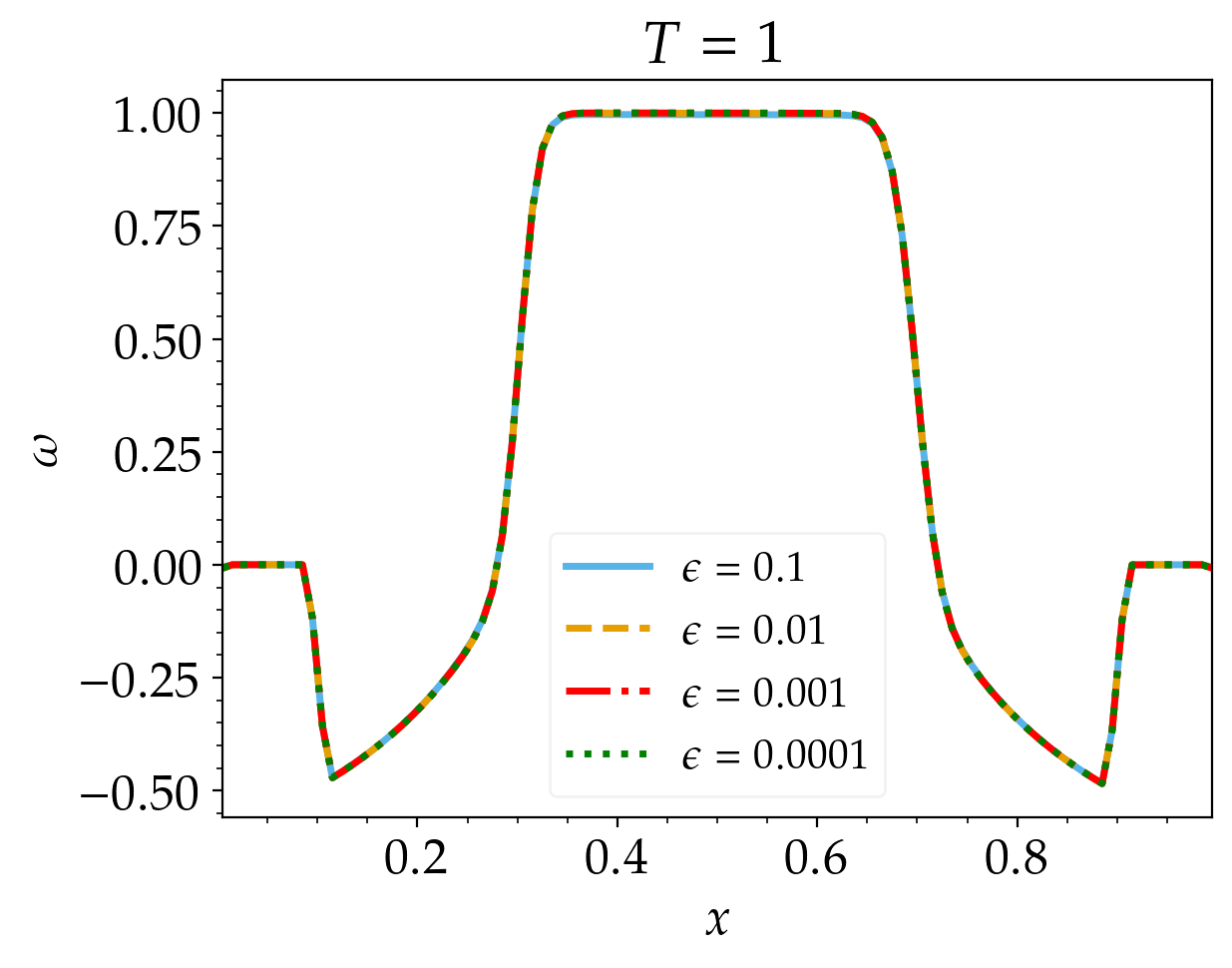}    
  \caption{Comparison of relative kinetic energies (left) and cross
    sections of the vorticities (right) for different values of
    $\veps$.}  
    \label{fig:stn_vrtx_rel_KE_vort}
  \end{figure}
  
\section{Conclusions}
\label{sec:concl}
We have constructed an energy stable, structure preserving,
well-balanced and AP scheme for the barotropic Euler system with
gravity in the anelastic limit. The energy stability is achieved by
introducing a velocity shift in the convective fluxes of mass and
momenta. The semi-implicit in time and finite volume in space scheme
preserves the positivity of density and is weakly consistent with the
continuous model upon mesh refinement. The numerical scheme admits 
the discrete hydrostatic states as solutions and the stability of
numerical solutions with respect to relative energy leads to the
well-balancing feature. The AP property of the scheme is rigorously
shown on the basis of apriori energy estimates. 

Results from several benchmark case studies confirm the theoretical
findings. Simulation of stationary initial data clearly shows that the
scheme approximates hydrostatic steady states upto machine precision
irrespective of the value of $\veps$, which substantiate its
well-balancing property. The proposed method can accurately resolve
flows that arise from small perturbation of steady states, where a
classical explicit scheme often fails. We have demonstrated the scheme's
capabilities to maintain the positivity of density and to capture
shock discontinuities in the compressible regime. The scheme performs
well for a vortex benchmark for low Mach number flows and its
dissipation seems uniform across different values of $\veps$ tending 
towards zero.

\bibliography{references}
\bibliographystyle{abbrv}

\end{document}